\documentclass[12pt]{article}

\usepackage[a4paper,margin=1in]{geometry}
\usepackage{amsmath,amssymb,amsthm,mathtools}
\usepackage{bm}
\usepackage{enumitem}
\usepackage{cite}
\usepackage[hidelinks]{hyperref}
\usepackage{microtype}

\numberwithin{equation}{section}

\newtheorem{theorem}{Theorem}[section]
\newtheorem{proposition}[theorem]{Proposition}
\newtheorem{lemma}[theorem]{Lemma}
\newtheorem{corollary}[theorem]{Corollary}
\newtheorem{remark}[theorem]{Remark}

\newcommand{\Kfluid}{\mathcal K}

\newcommand{\T}{\mathbb T_{2L}}

\newcommand{\dd}{\,\mathrm d}
\newcommand{\avg}[1]{\left\langle #1\right\rangle}
\newcommand{\Etot}{\mathcal H}
\newcommand{\Eflex}{\mathcal E_{\mathrm{fl}}}
\newcommand{\Bflex}{\mathcal B}
\newcommand{\R}{\mathbb R}
\newcommand{\psit}{\widetilde\psi}
\newcommand{\Pt}{\widetilde P_{\mathrm{ext}}}

\title{Stabilization of  hydroelastic waves in deep water }
\author{Jian Li$^{a}$,~~~Shaojie Yang$^{b}$%
\thanks{Corresponding authors: jianli\_jakura@163.com (Jian Li);
shaojieyang@kust.edu.cn (Shaojie Yang)}\\~\\[1ex]
\small $^a$ Institute for Advanced Study, Shenzhen University,\\
\small Shenzhen, Guangdong 518060, China\\~\\[1ex]
\small $^b$ Department of Mathematics, Faculty of Science, Kunming University of Science and Technology,\\
\small Kunming, Yunnan 650500, China}
\date{}

\begin{document}
\maketitle

\begin{abstract}
We consider the stabilization of two-dimensional periodic hydroelastic waves in deep water beneath a massless nonlinear Cosserat sheet by means of a localized pressure feedback. For sufficiently regular solutions, under an explicit condition relating gravity, flexural rigidity and the localization profile, we prove a uniform spacetime estimate and a localized dissipation observability inequality. As a consequence, the hydroelastic Hamiltonian decays quantitatively as long as the solution remains in a small geometric regime. The proof combines an energy dissipation identity with a nonlinear multiplier method adapted to the infinite-depth geometry, together with a coercivity analysis of the Cosserat bending energy.\\




\end{abstract}

\clearpage
\setcounter{tocdepth}{2}
\tableofcontents
\clearpage

\section{Introduction}\label{sec:introduction}

Hydroelastic water waves describe the interaction between a fluid free
surface and a deformable elastic structure.  Such interactions arise in polar
environments, where floating ice covers are affected by ocean waves and moving
loads and may also be used as roads or aircraft runways; related problems
include the interaction of large-amplitude waves with compliant sea-ice floes
\cite{KorobkinParauVandenBroeck,Squire1996,Marko2003,HegartySquire2008}.
Several models arise according to how the elastic layer is idealized.  If the
inertia of the layer is retained, one obtains a genuine fluid--plate system;
in the massless setting the elastic response enters instead through the
surface energy.  We work with the latter model, with the bending energy given
by the nonlinear Cosserat curvature functional.  Thus the fluid is governed
by the irrotational Euler equations, while the restoring force on the
interface contains a nonlinear fourth-order geometric term.  Our purpose is
to understand whether this system can be stabilized by a pressure feedback
acting only on a fixed portion of the surface.

The fluid occupies an infinite-depth periodic domain and the interface is
assumed to remain a graph.  We use the surface variables introduced by
Zakharov \cite{Zakharov1968}: the elevation \(\eta\) and the trace \(\psi\) of
the velocity potential.  In terms of the Dirichlet--Neumann operator
\(G(\eta)\), the equations take the form
\begin{equation}\label{eq:intro-system}
\begin{aligned}
  \eta_t
  &=G(\eta)\psi,\\
  \psi_t+g\eta+N(\eta)\psi+\beta\Bflex(\eta)
  &=-P_{\mathrm{ext}}.
\end{aligned}
\end{equation}
Here \(P_{\mathrm{ext}}=P_{\mathrm{ext}}(t,x)\) denotes the exterior pressure
acting on the hydroelastic interface.  The Cosserat restoring force is given in graph coordinates by
\begin{equation*}
  \Bflex(\eta)
  =\frac{\partial^2}{\partial x^2}
  \left(
    \frac{\eta_{xx}}{(1+\eta_x^2)^{5/2}}
  \right)
  +\frac{5}{2}\frac{\partial}{\partial x}\left(
    \frac{\eta_x\eta_{xx}^2}{(1+\eta_x^2)^{7/2}}
  \right).
\end{equation*}  The Hamiltonian formulation of the water-wave
system goes back to Zakharov, while the Dirichlet--Neumann formulation was
developed by Craig and Sulem \cite{CraigSulem1993}; see also
\cite{LannesBook}.  The existence and structure of nonlinear hydroelastic travelling waves have
been investigated in various settings. In particular, Toland
\cite{Toland2007} established the existence of heavy hydroelastic travelling
waves by a variational approach, while the nonlinear curvature formulation was
further developed in the geometric works of Plotnikov and Toland
\cite{Toland2008,PlotnikovToland}. Fully nonlinear deep-water computations were
also carried out by Guyenne and Părău \cite{GuyenneParauDeep}.  In Ref.~\cite{r1}, we derived the energy relations for hydroelastic waves and studied their associated properties.  The point relevant here is that the elastic term is
not a lower-order perturbation of the gravity-wave equation: after localization
it generates fourth-order contributions whose sign is not apparent from the
usual water-wave multipliers.

The Cauchy theory for hydroelastic waves has been developed in a number of
settings.  We refer to \cite{AmbroseSiegel,LiuAmbrose,WangYangLWP} for local
well-posedness results in Sobolev spaces, and to
\cite{YangCubic,WanYangLowReg} for more recent developments in two dimensions.
For the periodic Cosserat model with exterior pressure, including the infinite-depth case, we refer to  \cite{WanYangControl} for a related functional framework.  In this paper, the stabilization estimates are established for sufficiently regular solutions. The available local Cauchy theory provides the natural setting for the evolution problem.

The control and stabilization of water waves has a longer history.  For linearized surface waves, boundary control and stabilizability were studied by Reid and Russell, and related control problems were subsequently considered for irregular domains and for gravity--capillary waves
\cite{ReidRussell1985,Reid1986,Reid1995}.  For the nonlinear Euler free boundary
problem, Alazard proved boundary observability for gravity waves and developed
localized pressure stabilization for gravity and gravity--capillary waves
\cite{AlazardBoundary,AlazardGravity,AlazardCapillary}.  Local exact
controllability for periodic gravity--capillary waves was obtained by Alazard,
Baldi and Han-Kwan \cite{AlazardControl}.  The global-in-time Morawetz estimate
of Alazard, Ifrim and Tataru provides a complementary nonlinear local-energy
decay mechanism, in both finite and infinite depth \cite{AlazardIfrimTataru}.
These works show how localized information at the free surface can be converted
into quantitative control of the fluid energy, through multiplier, Pohozaev and
dispersive arguments.  We also mention the recent progress on the control of
three-dimensional water waves.  In particular, Zhu developed a control theory
for three-dimensional water waves in the framework of rational mechanics,
which provides an important reference for extending boundary control ideas
beyond the two-dimensional setting \cite{Zhu3DControl}.  The observability,
control and stabilization problems for wave equations have a long history in
the PDE control literature; we refer to the general survey of Zuazua and the
sharp boundary stabilization results of Bardos, Lebeau and Rauch
\cite{ZuazuaControl,BardosLebeauRauch1992}.  These developments motivate the
study of analogous stabilization mechanisms for nonlinear free boundary
systems.

For hydroelastic waves the control theory is much more recent.  Wan and Yang
proved exact controllability for two-dimensional periodic waves in finite and
infinite depth, their result concerns a control chosen to
steer the solution between prescribed states \cite{WanYangControl}.  Here the pressure is fixed by
the state through \eqref{eq:intro-feedback}, and the problem is to determine whether the resulting localized dissipation controls the full nonlinear hydroelastic Hamiltonian.  This stabilization mechanism is also distinct from the spectral stability of unforced travelling hydroelastic waves studied in \cite{BlythParauWang}.

Let \(m\) be the localized dilation function introduced in
Lemma~\ref{lem:cutoff}, and set \(\chi=1-m_x\).  The function \(\chi\) is
nonnegative and supported near the edge of the symmetry cell.  We impose the
feedback
\begin{equation}\label{eq:intro-feedback}
  P_{\mathrm{ext}}=\lambda\chi\eta_t,
  \qquad \lambda>0.
\end{equation}
The corresponding Hamiltonian is
\begin{equation*}
  \mathcal H(\eta,\psi)
  =
  \frac12\int_{\T}\psi G(\eta)\psi\,\mathrm dx
  +\frac g2\int_{\T}\eta^2\,\mathrm dx
  +\frac\beta2\int_{\T}
  \frac{\eta_{xx}^2}{(1+\eta_x^2)^{5/2}}\,\mathrm dx,
\end{equation*}
the feedback law \eqref{eq:intro-feedback} gives the exact
identity
\begin{equation}\label{eq:intro-dissipation}
  \frac{\mathrm d}{\mathrm dt}\mathcal H(t)
  =-\lambda\int_{\T}\chi|\eta_t|^2\,\mathrm dx.
\end{equation}
Thus the feedback gives an exact dissipation law, but the dissipation term contains
only the surface velocity on the set where \(\chi\) is supported.  Stabilization
therefore requires a quantitative recovery of the full Hamiltonian from this
localized information.  For gravity waves, such a recovery can be obtained by
combining a localized spatial multiplier with a Pohozaev identity.  The
hydroelastic problem has an additional difficulty at the highest differential
order: localization interacts directly with the fourth-order Cosserat force.
After integration by parts, derivatives of the localization profile enter the
principal elastic quadratic form, producing terms of the same order as the
bending energy rather than lower-order perturbations.  The gravity-wave
multiplier therefore cannot simply be reused; its elastic component has to be
chosen so that the fourth-order contribution remains coercive.  This is the
role of the multiplier introduced below.

The main result is the following uniform stabilization estimate.

\begin{theorem}
\label{thm:stabilization}
Fix \(g,\beta,L,\lambda>0\), let \(m\) be the localization function introduced
above and constructed in Lemma~\ref{lem:cutoff}, and set
\[
  \chi=1-m_x,
  \qquad
  M_3=\|m_{xxx}\|_{L^\infty(\T)}.
\]
Assume
\begin{equation}\label{eq:strict-parameter-gap}
  g>\frac{25}{9}\,\beta M_3^2,
\end{equation}
then there exist constants
\[
  \delta_*=\delta_*(g,\beta,L,m)>0,
  \qquad
  C=C(g,\beta,L,\lambda,m)>0,
\]
with the following property.  Let \(T>0\), and let \((\eta,\psi)\) be a
real-valued, even and \(2L\)-periodic solution of \eqref{eq:intro-system} on
\([0,T]\), with \(P_{\mathrm{ext}}\) given by \eqref{eq:intro-feedback}.
Let \(\phi\) be the associated finite-energy harmonic extension of \(\psi\),
and set
\[
  \mathcal S=\T\times(-\infty,0),
  \qquad
  u(t,x,z)=\phi(t,x,z+\eta(t,x)).
\]
Assume that, for some \(s\geq5\),
\begin{equation}\label{eq:theorem-regularity}
\begin{aligned}
  \eta
  &\in C\bigl([0,T];H^{s+3/2}(\T)\bigr)
  \cap C^1\bigl([0,T];H^{s-1}(\T)\bigr),\\
  \psi
  &\in C\bigl([0,T];H^s(\T)\bigr)
  \cap C^1\bigl([0,T];H^{s-5/2}(\T)\bigr),\\
  \nabla_{x,z}u
  &\in C\bigl([0,T];H^{s-1}(\mathcal S)\bigr).
\end{aligned}
\end{equation}
If
\begin{equation}\label{eq:localized-geometric-smallness}
  \sup_{0\leq t\leq T}
  \left(
    \|\eta_x(t)\|_{L^\infty(\T)}
    +\|\chi\eta_{xx}(t)\|_{L^\infty(\T)}
  \right)
  \leq\delta_*,
\end{equation}
then
\begin{equation}\label{eq:integrated-energy-estimate}
  \int_0^T\mathcal H(t)\,\mathrm dt
  \leq C\mathcal H(0),
\end{equation}
and consequently
\begin{equation}\label{eq:one-step-decay}
  \mathcal H(T)
  \leq\frac{C}{T}\mathcal H(0).
\end{equation}
The constants \(\delta_*\) and \(C\) are independent of the particular
solution and of \(T\).
\end{theorem}

In what follows, all solutions are understood to satisfy
\eqref{eq:theorem-regularity} on the time interval under consideration.

The smallness condition is global for the slope, while the second derivative
enters only through the localized quantity \(\chi\eta_{xx}\).  Combining
\eqref{eq:one-step-decay} with the exact dissipation identity
\eqref{eq:intro-dissipation} gives a localized observability estimate.

\begin{corollary}
\label{cor:dissipation-observability}
Under the assumptions of Theorem~\ref{thm:stabilization}, if \(T>C\), then
\begin{equation}\label{eq:dissipation-observability}
  \mathcal H(0)
  \leq
  \frac{\lambda}{1-C/T}
  \int_0^T\int_{\T}
  \chi(x)|\eta_t(t,x)|^2\,\mathrm dx\,\mathrm dt.
\end{equation}
In particular, if \(T\geq2C\), then
\begin{equation}\label{eq:dissipation-observability-2C}
  \mathcal H(0)
  \leq
  2\lambda
  \int_0^T\int_{\T}
  \chi(x)|\eta_t(t,x)|^2\,\mathrm dx\,\mathrm dt.
\end{equation}
\end{corollary}

Since the constants in Theorem~\ref{thm:stabilization} are uniform within the
same geometric regime, the one-step estimate can be iterated along global
solutions that remain in that regime.

\begin{corollary}
\label{cor:optimized-exponential}
Assume the hypotheses and notation of Theorem~\ref{thm:stabilization}, and let
\((\eta,\psi)\) be a global solution of \eqref{eq:intro-system}, with
\(P_{\mathrm{ext}}\) given by \eqref{eq:intro-feedback}, whose restriction to
every finite time interval satisfies the regularity assumptions of
Theorem~\ref{thm:stabilization}.  If
\begin{equation}\label{eq:global-localized-geometric-smallness}
  \sup_{t\geq0}
  \left(
    \|\eta_x(t)\|_{L^\infty(\T)}
    +\|\chi\eta_{xx}(t)\|_{L^\infty(\T)}
  \right)
  \leq\delta_*,
\end{equation}
then
\begin{equation}\label{eq:continuous-exponential-decay}
  \mathcal H(t)
  \leq
  \exp\!\left(1-\frac{t}{\mathrm e C}\right)\mathcal H(0),
  \qquad t\geq0.
\end{equation}
More generally, for every \(\tau>C\) gives
\begin{equation}\label{eq:general-exponential-family}
  \mathcal H(t)
  \leq
  \frac{\tau}{C}
  \exp\!\left(-\frac{\log(\tau/C)}{\tau}\,t\right)\mathcal H(0),
  \qquad t\geq0.
\end{equation}
Within this one-step iteration family, the exponential rate
\(\log(\tau/C)/\tau\) is maximal at \(\tau=\mathrm e C\).
\end{corollary}

The main difficulty created by elasticity occurs at the highest differential
order.  To see the mechanism, consider at the flat state the one-parameter
family
\[
  \zeta_r=\partial_x(m\eta)+r(1-m_x)\eta-\frac14\eta.
\]
The principal part of the elastic pairing is
\begin{equation}\label{eq:intro-principal-elastic-pairing}
  \int_{\T}
  \left[
    r-\frac14+\left(\frac52-r\right)m_x
  \right]\eta_{xx}^2\,\mathrm dx,
\end{equation}
up to terms involving derivatives of the localization profile.  Thus the
principal fourth-order coefficient depends on the damping profile for every
\(r\neq5/2\).  The choice \(r=5/2\) is the unique one for which this
dependence disappears, and \eqref{eq:intro-principal-elastic-pairing} reduces
to the fixed positive coefficient \(9/4\) in front of
\(\|\eta_{xx}\|_{L^2}^2\).  The nonlinear flexural virial identity proved
below retains this coercive structure and isolates the remaining localization
errors at lower order.  After combining these terms with the gravity and
bending energies, an optimized two-parameter Young splitting gives the factor
\(25/9\) in \eqref{eq:strict-parameter-gap}.  The value \(25/9\) is the
sharp constant for this coercivity splitting in the flat-interface limit.

The fluid part of the argument also has to be adjusted to infinite depth.  We
combine the localized multiplier identity with a Pohozaev formula and a
weighted stress identity on truncated fluid domains.  The harmonic extension
of a periodic finite-energy trace decays below the graph, so the lower boundary
terms vanish as the truncation is sent to infinity.  In this way the artificial lower boundary leaves no contribution in the infinite-depth limit, while the remaining vertical trace on
the boundary of the even symmetry cell has a favorable sign.  These identities
are then coupled to the elastic virial estimate and the exact dissipation law.
The argument is carried out directly at the nonlinear level and does not rely on a linearized observability estimate.

The argument proceeds directly in physical space and does not require a spectral decomposition.  The geometric assumption \eqref{eq:localized-geometric-smallness} requires global smallness of the slope, while the second derivative enters only through the localized quantity \(\chi\eta_{xx}\).  Such
distinction is important because the physical Hamiltonian controls the surface
at the \(H^2\)-level but does not control \(\eta_{xx}\) in \(L^\infty\).  The constant \(C\) in \eqref{eq:integrated-energy-estimate} is independent of the solution and of the length of the time interval, as long as \eqref{eq:localized-geometric-smallness} holds. The uniformity allows the estimate to be iterated on successive time intervals.

The paper is organized as follows.  Section~\ref{sec:model} gives the
free boundary formulation and the Hamiltonian structure of the Cosserat model.
The localized feedback is introduced in Section~\ref{sec:feedback}.
Section~\ref{sec:flex-virial} contains the nonlinear elastic virial identity
and the coercivity estimate.  The fluid multiplier, Pohozaev and weighted stress
identities are proved in Section~\ref{sec:fluid-identities}, and the remaining
error terms are estimated in Section~\ref{sec:estimates}.  The proof of Theorem~\ref{thm:stabilization} and its observability and decay
corollaries is completed in Section~\ref{sec:main-theorem}.
Finally, Section~\ref{sec:linear-check} discusses the linearized system and
proves the corresponding stabilization estimate.

\section{Hydroelastic waves and Hamiltonian formulation}
\label{sec:model}
In this section, we introduce the hydroelastic waves model and recall its Hamiltonian formulation. We  derive the free boundary problem and then describe the corresponding energy structure that will be used throughout the stabilization analysis.
\subsection{The free boundary problem}

We work in a two-dimensional Cartesian coordinate system \((x,y)\), where
\(x\) is the horizontal coordinate and \(y\) is the vertical coordinate,
with gravity acting in the negative \(y\)-direction.  Fix \(L>0\) and identify
\(\T=\R/(2L\mathbb Z)\) with the representative interval \([-L,L]\).  The
unknown \(\eta=\eta(t,x)\) denotes the vertical displacement of the
hydroelastic interface from the undisturbed level \(y=0\).  At time
\(t\), the elastic free surface is
\[
  \Gamma(t)=\{(x,\eta(t,x)):x\in\T\},
\]
and the fluid occupies the infinite-depth domain
\begin{equation*}
  \Omega(t)=\{(x,y)\in\T\times\R:\ y<\eta(t,x)\}.
\end{equation*}
We assume throughout that the interface remains a graph and impose evenness
in the horizontal variable.  Restricting an even periodic solution to
\([0,L]\) gives the symmetry conditions
\begin{equation*}
  \phi_x(t,0,y)=\phi_x(t,L,y)=0,
  \qquad
  \eta_x(t,0)=\eta_x(t,L)=0.
\end{equation*}
All odd spatial derivatives of a smooth even profile vanish at the endpoints,
in particular \(\eta_{xxx}=0\).  These are symmetry conditions, not separate
mechanical edge conditions.

We consider a massless elastic sheet governed by the nonlinear Cosserat
bending law.  More precisely, following the special Cosserat hydroelastic
model derived by Plotnikov and Toland \cite{PlotnikovToland}, the elastic
restoring force is written in graph coordinates as
\begin{equation}\label{eq:flex-force}
  \Bflex(\eta)
  =
  \frac{\partial^2}{\partial x^2}\left(
    \frac{\eta_{xx}}{(1+\eta_x^2)^{5/2}}
  \right)
  +\frac{5}{2}\frac{\partial}{\partial x}\left(
    \frac{\eta_x\eta_{xx}^2}{(1+\eta_x^2)^{7/2}}
  \right).
\end{equation}
This nonlinear force is the variational derivative of the Cosserat bending
energy proportional to the squared curvature integrated with respect to arc
length.  The variational relation is given in \eqref{eq:first-variation}.  For the
modelling background of this elastic law, see also
\cite{PlotnikovToland,Toland2008}.  The present model retains only this
bending response of the sheet: sheet inertia, prestress, and stretching energy
are not included.

The fluid beneath the elastic surface is incompressible, inviscid and
irrotational.  Hence its velocity is represented by a potential
\(\phi=\phi(t,x,y)\), with velocity field \(\nabla_{x,y}\phi\).  Throughout
the paper we work in the finite-energy class appropriate to the infinite-depth domain,
\begin{equation}\label{eq:finite-dirichlet-energy}
  \iint_{\Omega(t)}
  |\nabla_{x,y}\phi(t,x,y)|^2\,\dd x\dd y<\infty.
\end{equation}
The velocity potential satisfies
\begin{equation}\label{eq:deep-water-laplace}
  \begin{aligned}
    \Delta_{x,y}\phi&=0
    &&\text{in }\Omega(t),\\
    \nabla_{x,y}\phi(t,x,y)&\longrightarrow0
    &&\text{as }y\to-\infty.
  \end{aligned}
\end{equation}
The kinematic boundary condition expresses that the elastic free surface is
transported by the fluid:
\begin{equation*}
  \eta_t
  =\bigl(\phi_y-\eta_x\phi_x\bigr)\big|_{y=\eta}.
\end{equation*}
The dynamic boundary condition is
\begin{equation}\label{eq:dynamic-eulerian}
  \phi_t+\frac12(\phi_x^2+\phi_y^2)
  +g\eta+\beta\Bflex(\eta)
  =-P_{\mathrm{ext}}
  \qquad\text{on }y=\eta(t,x).
\end{equation}
Here \(g>0\) is the gravitational coefficient and \(\beta>0\) is the
flexural coefficient.  The nonlinear Cosserat restoring force
\(\Bflex(\eta)\) is defined in \eqref{eq:flex-force}, so the elastic
contribution in \eqref{eq:dynamic-eulerian} is \(\beta\Bflex(\eta)\).
The function \(P_{\mathrm{ext}}=P_{\mathrm{ext}}(t,x)\) denotes the exterior
pressure applied to the hydroelastic surface.  The localized feedback law for
\(P_{\mathrm{ext}}\) will be specified in Section~\ref{sec:feedback}.

We normalize the reference level of the periodic interface by imposing
\begin{equation}\label{eq:zero-mean-initial}
  \int_{\T}\eta(0,x)\,\dd x=0,
\end{equation}
which fixes the mean elevation of the graph.  Lemma~\ref{lem:zero-mode}
shows that the kinematic boundary condition preserves it in time.

The qualitative condition at infinite depth in
\eqref{eq:deep-water-laplace}, together with periodicity and finite Dirichlet
energy, yields a quantitative decay estimate below any horizontal line lying
strictly beneath the free surface.  The bounds \eqref{eq:deep-decay-zero}--\eqref{eq:deep-decay-derivatives}
will be used repeatedly to justify truncation of the infinite-depth domain and
the subsequent passage to the limit.

\begin{lemma}[Decay at infinite depth]
\label{lem:deep-decay}
Fix \(t\) and suppress the time variable from the notation.  Choose
\[
  y_0<\min_{x\in\T}\eta(x),
\]
and define
\begin{equation*}
  \phi_\infty
  :=
  \frac{1}{2L}\int_{\T}\phi(x,y_0)\,\dd x.
\end{equation*}
Then
\begin{equation}\label{eq:deep-decay-zero}
  \|\phi(\cdot,y)-\phi_\infty\|_{L^2(\T)}
  \leq
  C_0 \mathrm{e}^{\pi(y-y_0)/L},
  \qquad y\leq y_0.
\end{equation}
Moreover, for every \(a,b\in\mathbb N_0\) with \(a+b\geq1\), we have
\begin{equation}\label{eq:deep-decay-derivatives}
  \bigl\|
    \partial_x^a\partial_y^b\phi(\cdot,y)
  \bigr\|_{L^2(\T)}
  \leq
  C_{a,b}\mathrm{e}^{\pi(y-y_0)/L},
  \qquad y\leq y_0.
\end{equation}
\end{lemma}

\begin{proof}
	Since \(y_0<\min_{\T}\eta\), the portion of the fluid domain below
	\(y=y_0\) is the flat half-cylinder
	\[
	\T\times(-\infty,y_0].
	\]
	For \(y\leq y_0\), expand \(\phi\) in Fourier series as
	\begin{equation*}
		\phi(x,y)
		=
		\sum_{n\in\mathbb Z}\widehat{\phi}_n(y)\mathrm{e}^{ik_nx},
		\qquad
		k_n=\frac{\pi n}{L},
	\end{equation*}
	where
	\[
	\widehat{\phi}_n(y)
	=
	\frac{1}{2L}
	\int_{\T}\phi(x,y)\mathrm{e}^{-ik_nx}\,\dd x.
	\]
	Since \(\Delta\phi=0\), each Fourier coefficient satisfies
	\begin{equation}\label{eq:deep-fourier-ode}
		\widehat{\phi}_n''-k_n^2\widehat{\phi}_n=0.
	\end{equation}
	
	For \(n\neq0\), the general solution of
	\eqref{eq:deep-fourier-ode} is
	\[
	\widehat{\phi}_n(y)
	=
	A_n\mathrm{e}^{|k_n|(y-y_0)}
	+
	B_n\mathrm{e}^{-|k_n|(y-y_0)}.
	\]
	Since \(y-y_0\to-\infty\) as \(y\to-\infty\), the first term decays,
	whereas the second one grows exponentially.  To see directly that the
	growing mode is incompatible with finite Dirichlet energy, observe that
	\[
	\widehat{\phi}_n'(y)
	=
	|k_n|A_n\mathrm{e}^{|k_n|(y-y_0)}
	-
	|k_n|B_n\mathrm{e}^{-|k_n|(y-y_0)}.
	\]
	Since \(k_n^2=|k_n|^2\), expanding the two squares gives the exact
	identity
	\[
	|\widehat{\phi}_n'(y)|^2
	+
	k_n^2|\widehat{\phi}_n(y)|^2
	=
	2|k_n|^2
	\left(
	|A_n|^2\mathrm{e}^{2|k_n|(y-y_0)}
	+
	|B_n|^2\mathrm{e}^{-2|k_n|(y-y_0)}
	\right),
	\]
	the cross terms cancelling identically.  Hence, if \(B_n\neq0\), then
	\begin{align*}
		&\int_{-\infty}^{y_0}
		\left(
		|\widehat{\phi}_n'(y)|^2
		+
		k_n^2|\widehat{\phi}_n(y)|^2
		\right)\,\dd y
		\geq
		2|k_n|^2|B_n|^2
		\int_{-\infty}^{y_0}
		\mathrm{e}^{-2|k_n|(y-y_0)}\,\dd y
		=+\infty.
	\end{align*}
	On the other hand, since
	\[
	\T\times(-\infty,y_0]\subset\Omega(t),
	\]
	the finite-energy assumption \eqref{eq:finite-dirichlet-energy} gives
	\[
	\int_{-\infty}^{y_0}\int_{\T}
	|\nabla\phi(x,y)|^2\,\dd x\dd y<\infty.
	\]
	By Parseval's identity and Tonelli's theorem, we get
	\begin{align*}
		&\int_{-\infty}^{y_0}\int_{\T}
		|\nabla\phi(x,y)|^2\,\dd x\dd y
		=
		2L\sum_{m\in\mathbb Z}
		\int_{-\infty}^{y_0}
		\left(
		|\widehat{\phi}_m'(y)|^2
		+
		k_m^2|\widehat{\phi}_m(y)|^2
		\right)\,\dd y.
	\end{align*}
	Every term in the sum is nonnegative.  Hence the finiteness of the
	left-hand side implies, for each \(m\in\mathbb Z\),
	\[
	\int_{-\infty}^{y_0}
	\left(
	|\widehat{\phi}_m'(y)|^2
	+
	k_m^2|\widehat{\phi}_m(y)|^2
	\right)\,\dd y<\infty.
	\]
	For the fixed mode \(n\neq0\), however, we have just shown that
	\(B_n\neq0\) would make this integral infinite.  This contradicts the
	finite Dirichlet energy.  Therefore, we obtain
	\[
	B_n=0,
	\]
	and 
	\begin{equation*}
		\widehat{\phi}_n(y)
		=
		a_n\mathrm{e}^{|k_n|(y-y_0)},
		\qquad
		a_n:=\widehat{\phi}_n(y_0),
		\qquad n\neq0.
	\end{equation*}
	
	For the zero Fourier mode, \eqref{eq:deep-fourier-ode} gives
	\[
	\widehat{\phi}_0''=0,
	\]
	and thus
	\[
	\widehat{\phi}_0(y)=A_0+B_0y.
	\]
	Since
	\[
	\int_{-\infty}^{y_0}
	|\widehat{\phi}_0'(y)|^2\,\dd y<\infty,
	\]
	we must have \(B_0=0\).  Hence the zero Fourier mode is independent of
	\(y\), and by the definition of the Fourier coefficients,
	\[
	\widehat{\phi}_0
	=
	\frac{1}{2L}\int_{\T}\phi(x,y_0)\,\dd x
	=
	\phi_\infty.
	\]
	Consequently, we have
	\begin{equation}\label{eq:deep-fourier-decaying-part}
		\phi(x,y)-\phi_\infty
		=
		\sum_{n\neq0}
		a_n\mathrm{e}^{|k_n|(y-y_0)}\mathrm{e}^{ik_nx}.
	\end{equation}
	
	We first prove \eqref{eq:deep-decay-zero}.  By Parseval's identity, one has
	\[
	\|\phi(\cdot,y)-\phi_\infty\|_{L^2(\T)}^2
	=
	2L
	\sum_{n\neq0}
	|a_n|^2\mathrm{e}^{2|k_n|(y-y_0)}.
	\]
	Since \(y-y_0\leq0,\)  \(|k_n|\geq\frac{\pi}{L}\) and \(
	 n\neq0,\)
	we have
	\[
	\mathrm{e}^{2|k_n|(y-y_0)}
	\leq
	\mathrm{e}^{2\pi(y-y_0)/L}.
	\]
	Therefore
	\[
	\|\phi(\cdot,y)-\phi_\infty\|_{L^2(\T)}^2
	\leq
	\mathrm{e}^{2\pi(y-y_0)/L}
	2L\sum_{n\neq0}|a_n|^2,
	\]
	which proves \eqref{eq:deep-decay-zero}.
	
	Now let \(a,b\in\mathbb N_0\) with \(a+b\geq1\).  Differentiating
	\eqref{eq:deep-fourier-decaying-part} gives
	\[
	\partial_x^a\partial_y^b\phi(x,y)
	=
	\sum_{n\neq0}
	(ik_n)^a|k_n|^b
	a_n\mathrm{e}^{|k_n|(y-y_0)}\mathrm{e}^{ik_nx}.
	\]
	Another application of Parseval's identity yields
	\begin{align*}
		\bigl\|
		\partial_x^a\partial_y^b\phi(\cdot,y)
		\bigr\|_{L^2(\T)}^2
		&=
		2L\sum_{n\neq0}
		|k_n|^{2(a+b)}
		|a_n|^2
		\mathrm{e}^{2|k_n|(y-y_0)}
		\\
		&\leq
		\mathrm{e}^{2\pi(y-y_0)/L}
		2L\sum_{n\neq0}
		|k_n|^{2(a+b)}|a_n|^2.
	\end{align*}
	The last sum is finite because
	\(\phi(\cdot,y_0)\) is smooth.  Taking square roots proves
	\eqref{eq:deep-decay-derivatives}.  Notice that the argument uses only
	periodicity, harmonicity, smoothness at the fixed level \(y=y_0\), and the
	finite-energy condition \eqref{eq:finite-dirichlet-energy}.  In particular,
	it provides the quantitative form of the qualitative decay condition at infinite depth in
	\eqref{eq:deep-water-laplace}.
\end{proof}

\subsection{Cosserat flexural energy and force}

The nonlinear Cosserat bending energy associated with the restoring force
\eqref{eq:flex-force} can be written directly in graph coordinates as
\begin{equation}\label{eq:flex-energy}
  \Eflex(\eta)
  =
  \frac12\int_{\T}
  \frac{\eta_{xx}^2}{(1+\eta_x^2)^{5/2}}\,\dd x.
\end{equation}
For later use we verify directly that \(\Bflex(\eta)\) is the variational
derivative of \(\Eflex\), without introducing auxiliary variables.  For every
smooth periodic variation \(v\),
\begin{equation}\label{eq:first-variation}
  D\Eflex(\eta)[v]
  =
  \int_{\T}\left[
    -\frac52
    \frac{\eta_x\eta_{xx}^2}{(1+\eta_x^2)^{7/2}}\,v_x
    +
    \frac{\eta_{xx}}{(1+\eta_x^2)^{5/2}}\,v_{xx}
  \right]\dd x
  =
  \int_{\T}\Bflex(\eta)v\,\dd x.
\end{equation}
Indeed, differentiating \eqref{eq:flex-energy} in the direction \(v\) gives
the first integral in \eqref{eq:first-variation}, and two periodic
integrations by parts give the second equality.

For comparison with the intrinsic formulation, if
\[
  \kappa
  =
  \frac{\eta_{xx}}{(1+\eta_x^2)^{3/2}},
  \qquad
  \partial_s
  =
  (1+\eta_x^2)^{-1/2}\partial_x
\]
denote the curvature and arc-length derivative, respectively, then the same
force can be written as
\begin{equation}\label{eq:flex-geometric-form}
  \Bflex(\eta)
  =
  \kappa_{ss}+\frac12\kappa^3.
\end{equation}
Thus \eqref{eq:flex-force} and \eqref{eq:flex-geometric-form} are the graph
and intrinsic representations of the same Cosserat restoring force.

\subsection{Dirichlet--Neumann operator and reduction to the interface}

We now rewrite the free boundary problem in terms of quantities defined on the
moving interface.  The use of the surface variables \((\eta,\psi)\) goes back
to Zakharov \cite{Zakharov1968}, while the Dirichlet--Neumann formulation was
developed by Craig and Sulem \cite{CraigSulem1993}.  See also
\cite{LannesBook} for later analytic treatments.  For a graph
\(\eta\) and a periodic boundary value \(\psi\), let \(\phi\) denote the
harmonic extension satisfying the finite-energy condition
\eqref{eq:finite-dirichlet-energy}, with \(\phi|_{y=\eta}=\psi\).  We define the Dirichlet--Neumann operator by
\begin{equation}\label{eq:def-DNO}
  G(\eta)\psi
  =\bigl(\phi_y-\eta_x\phi_x\bigr)\big|_{y=\eta}.
\end{equation}
Thus \(G(\eta)\psi\) is the outward normal derivative multiplied by
\(\sqrt{1+\eta_x^2}\).
Set
\begin{equation*}
  V=\phi_x|_{y=\eta},
  \qquad
  B=\phi_y|_{y=\eta}.
\end{equation*}
Differentiating \(\psi(t,x)=\phi(t,x,\eta(t,x))\) in \(x\) gives
\begin{equation}\label{eq:VB-relations-model}
  \psi_x=V+\eta_xB,
  \qquad
  G(\eta)\psi=B-\eta_xV.
\end{equation}
Hence
\begin{equation*}
  V=\frac{\psi_x-\eta_xG(\eta)\psi}{1+\eta_x^2},
  \qquad
  B=\frac{G(\eta)\psi+\eta_x\psi_x}{1+\eta_x^2}.
\end{equation*}
Define
\begin{equation}\label{eq:N-definition}
  N(\eta)\psi
  =\frac12\psi_x^2
  -\frac12\frac{\bigl(G(\eta)\psi+\eta_x\psi_x\bigr)^2}
                    {1+\eta_x^2},
\end{equation}
which implies
\begin{equation}\label{eq:N-VB}
  N(\eta)\psi
  =\frac12V^2-\frac12B^2+\eta_xVB.
\end{equation}

The free boundary conditions can now be expressed entirely through the
surface variables.  First, the kinematic condition gives
\[
  \eta_t
  =B-\eta_xV
  =G(\eta)\psi.
\]
Next, differentiating
\(\psi(t,x)=\phi(t,x,\eta(t,x))\) with respect to time yields
\[
  \psi_t=\phi_t|_{y=\eta}+B\eta_t.
\]
Using the dynamic boundary condition together with
\(\eta_t=B-\eta_xV\), we obtain
\[
  \psi_t
  =-g\eta-\beta\Bflex(\eta)-P_{\mathrm{ext}}
  -\left(\frac12V^2-\frac12B^2+\eta_xVB\right).
\]
Therefore, by \eqref{eq:N-VB}, the free boundary problem becomes
\begin{equation}\label{eq:hydroelastic-system}
  \begin{cases}
    \eta_t=G(\eta)\psi,\\[2mm]
    \psi_t+g\eta+N(\eta)\psi+\beta\Bflex(\eta)
    =-P_{\mathrm{ext}}.
  \end{cases}
\end{equation}
Conversely, the same identities reconstruct the kinematic and dynamic
boundary conditions from \eqref{eq:hydroelastic-system}.

\subsection{Kinetic energy, Hamiltonian derivatives, and energy balance}

\begin{lemma}[Boundary representation of the kinetic energy]
\label{lem:kinetic-boundary}
For the finite-energy harmonic extension \(\phi\) used in the definition
\eqref{eq:def-DNO},
\begin{equation}\label{eq:kinetic-boundary}
  \frac12\iint_{\Omega(t)}|\nabla\phi|^2\,\dd y\dd x
  =\frac12\int_{\T}\psi G(\eta)\psi\,\dd x.
\end{equation}
Thus \(G(\eta)\) is nonnegative on real functions and \(G(\eta)1=0\).
\end{lemma}

\begin{proof}
	For \(R>0\) sufficiently large so that
	\[
	-R<\min_{x\in\T}\eta(t,x),
	\]
	set
	\[
	\Omega_R(t)
	:=
	\Omega(t)\cap\{y>-R\}
	=
	\{(x,y)\in\T\times\R:\ -R<y<\eta(t,x)\}.
	\]
	Since \(\phi\) is harmonic in \(\Omega_R(t)\), Green's first identity gives
	\begin{equation}\label{eq:green-truncated-kinetic}
		\iint_{\Omega_R(t)}|\nabla\phi|^2\,\dd y\dd x
		=
		\int_{\partial\Omega_R(t)}
		\phi\,\partial_n\phi\,\dd S.
	\end{equation}
	Here \(n\) denotes the outward unit normal.  Since the horizontal
	variable is periodic, there are no vertical side contributions when
	\(\Omega_R(t)\) is regarded as a domain in the cylinder
	\(\T\times\R\).  Equivalently, if one works on the representative cell
	\([-L,L]\), the two vertical-side terms cancel by periodicity.
	
	The boundary of \(\Omega_R(t)\) therefore consists of the free surface
	\[
	\Gamma(t)=\{(x,\eta(t,x)):x\in\T\}
	\]
	and the artificial lower boundary
	\[
	\Gamma_R=\{(x,-R):x\in\T\}.
	\]
	
	On the free surface,
	\[
	n
	=
	\frac{(-\eta_x,1)}
	{\sqrt{1+\eta_x^2}},
	\qquad
	\dd S
	=
	\sqrt{1+\eta_x^2}\,\dd x.
	\]
	Hence
	\begin{align*}
		\int_{\Gamma(t)}
		\phi\,\partial_n\phi\,\dd S
		&=
		\int_{\T}
		\psi
		\frac{\phi_y-\eta_x\phi_x}
		{\sqrt{1+\eta_x^2}}
		\sqrt{1+\eta_x^2}\,\dd x\\
		&=
		\int_{\T}
		\psi\,
		\bigl(\phi_y-\eta_x\phi_x\bigr)\big|_{y=\eta}
		\,\dd x\\
		&=
		\int_{\T}\psi G(\eta)\psi\,\dd x.
	\end{align*}
	
	On the lower boundary \(\Gamma_R\), the outward unit normal is
	\(n=(0,-1)\).  Therefore
	\[
	\int_{\Gamma_R}
	\phi\,\partial_n\phi\,\dd S
	=
	-\int_{\T}
	\phi(x,-R)\phi_y(x,-R)\,\dd x.
	\]
	Lemma~\ref{lem:deep-decay}, together with the description of the zero
	Fourier mode in its proof, implies
	\begin{equation}\label{eq:kinetic-lower-boundary-vanish}
		\int_{\T}
		\phi(x,-R)\phi_y(x,-R)\,\dd x
		\longrightarrow0
		\qquad\text{as }R\to\infty.
	\end{equation}
	Indeed, this follows from the decomposition
	\[
	\phi
	=
	\bigl(\phi-\phi_\infty\bigr)+\phi_\infty.
	\]
	Since the zero Fourier mode is independent of \(y\),
	\[
	\int_{\T}\phi_y(x,-R)\,\dd x=0,
	\]
	and hence
	\begin{align*}
		\left|
		\int_{\T}\phi(x,-R)\phi_y(x,-R)\,\dd x
		\right|
		&=
		\left|
		\int_{\T}
		\bigl(\phi(x,-R)-\phi_\infty\bigr)
		\phi_y(x,-R)\,\dd x
		\right|
		\\
		&\leq
		\|\phi(\cdot,-R)-\phi_\infty\|_{L^2(\T)}
		\|\phi_y(\cdot,-R)\|_{L^2(\T)}
		\\
		&\leq
		C \mathrm{e}^{-2\pi R/L},
	\end{align*}
	Here the last inequality follows from Lemma~\ref{lem:deep-decay}.
	
	Substituting the two boundary contributions into
	\eqref{eq:green-truncated-kinetic}, we obtain
	\begin{equation*}
		\iint_{\Omega_R(t)}|\nabla\phi|^2\,\dd y\dd x
		=
		\int_{\T}\psi G(\eta)\psi\,\dd x
		-
		\int_{\T}\phi(x,-R)\phi_y(x,-R)\,\dd x.
	\end{equation*}
	As \(R\to\infty\), the domains \(\Omega_R(t)\) increase to
	\(\Omega(t)\).  Since \(|\nabla\phi|^2\geq0\), the monotone convergence
	theorem gives
	\[
	\lim_{R\to\infty}
	\iint_{\Omega_R(t)}|\nabla\phi|^2\,\dd y\dd x
	=
	\iint_{\Omega(t)}|\nabla\phi|^2\,\dd y\dd x.
	\]
	Together with \eqref{eq:kinetic-lower-boundary-vanish}, this yields
	\[
	\iint_{\Omega(t)}|\nabla\phi|^2\,\dd y\dd x
	=
	\int_{\T}\psi G(\eta)\psi\,\dd x.
	\]
	
	The same identity immediately implies, for every real boundary datum
	\(\psi\),
	\[
	\int_{\T}\psi G(\eta)\psi\,\dd x
	=
	\iint_{\Omega(t)}|\nabla\phi|^2\,\dd y\dd x
	\geq0.
	\]
	Thus \(G(\eta)\) is nonnegative in the quadratic-form sense.
	
	Finally, if the boundary datum is the constant function \(1\), its
	finite-energy harmonic extension is the constant function
	\(\phi\equiv1\).  Hence
	\[
	G(\eta)1
	=
	\bigl(\phi_y-\eta_x\phi_x\bigr)\big|_{y=\eta}
	=0.
	\]
	This completes the proof.
\end{proof}

Zakharov showed that the irrotational water-wave problem is Hamiltonian in
the surface variables \((\eta,\psi)\) \cite{Zakharov1968}.  In the
Dirichlet--Neumann formulation, the fluid kinetic energy is represented by the
boundary quadratic form
\(\frac12\int_{\T}\psi G(\eta)\psi\,\dd x\).  See
\cite{CraigSulem1993} and, for analytic treatments of this formulation,
\cite{LannesBook}.  Adding the gravitational and Cosserat
bending energies, we define the total  Hamiltonian by
\begin{equation}\label{eq:total-energy}
  \Etot(\eta,\psi)
  =\frac12\int_{\T}\psi G(\eta)\psi\,\dd x
  +\frac g2\int_{\T}\eta^2\,\dd x
  +\beta\Eflex(\eta).
\end{equation}

For the kinetic part
\[
  \Kfluid(\eta,\psi)=\frac12\int_{\T}\psi G(\eta)\psi\,\dd x,
\]
we shall use the following two variational identities:
\begin{align*}
  D_\psi\Kfluid(\eta,\psi)[\dot\psi]
  &=\int_{\T}G(\eta)\psi\,\dot\psi\,\dd x,\\
  D_\eta\Kfluid(\eta,\psi)[\dot\eta]
  &=\int_{\T}N(\eta)\psi\,\dot\eta\,\dd x.
\end{align*}
Indeed, the first identity follows from self-adjointness of the Dirichlet--Neumann operator,
\begin{equation*}
  \int_{\T}fG(\eta)g\,\dd x
  =\iint_{\Omega(\eta)}\nabla\phi_f\cdot\nabla\phi_g\,\dd y\dd x
  =\int_{\T}gG(\eta)f\,\dd x,
\end{equation*}
where the lower boundary contribution vanishes by Lemma~\ref{lem:deep-decay}
and the same Cauchy--Schwarz estimate used in the proof of
Lemma~\ref{lem:kinetic-boundary}.  For the variation with respect to the
graph, we use the shape derivative of the Dirichlet--Neumann operator.  In the
present notation it reads (see \cite{LannesJAMS,LannesBook})
\begin{equation}\label{eq:DNO-shape-derivative}
  G'(\eta)[\dot\eta]\psi
  =-G(\eta)(B\dot\eta)-\partial_x(V\dot\eta).
\end{equation}
Combining \eqref{eq:DNO-shape-derivative} with
\eqref{eq:VB-relations-model}, self-adjointness, and periodic integration by
parts gives
\[
  D_\eta\Kfluid[\dot\eta]
  =\frac12\int_{\T}
    \left[-(B-\eta_xV)B+(V+\eta_xB)V\right]\dot\eta\,\dd x
  =\int_{\T}N(\eta)\psi\,\dot\eta\,\dd x.
\]
Together with \eqref{eq:first-variation}, these identities give
\begin{equation*}
  \frac{\delta\Etot}{\delta\psi}=G(\eta)\psi,
  \qquad
  \frac{\delta\Etot}{\delta\eta}
  =g\eta+N(\eta)\psi+\beta\Bflex(\eta).
\end{equation*}
Combining these variational identities with
\eqref{eq:hydroelastic-system} gives the forced Hamiltonian form
\begin{equation}\label{eq:hamiltonian-system}
  \eta_t=\frac{\delta\Etot}{\delta\psi},
  \qquad
  \psi_t=-\frac{\delta\Etot}{\delta\eta}-P_{\mathrm{ext}}.
\end{equation}
Consequently, every solution satisfying the regularity assumptions of
Theorem~\ref{thm:stabilization} satisfies the energy balance
\begin{equation}\label{eq:general-energy-balance}
  \frac{\dd}{\dd t}\Etot(t)
  =-\int_{\T}P_{\mathrm{ext}}\eta_t\,\dd x.
\end{equation}
Indeed, differentiating \(\Etot\) along the flow and using
\eqref{eq:hamiltonian-system} gives
\[
  \frac{\dd}{\dd t}\Etot
  =\int_{\T}\left(
  \frac{\delta\Etot}{\delta\eta}\eta_t
  +\frac{\delta\Etot}{\delta\psi}\psi_t\right)\dd x
  =-\int_{\T}P_{\mathrm{ext}}\eta_t\,\dd x.
\]

\subsection{Zero modes}

Set
\[
  \avg{f}=\frac1{2L}\int_{\T}f(x)\,\dd x,
  \qquad
  \psit=\psi-\avg{\psi},
  \qquad
  \Pt=P_{\mathrm{ext}}-\avg{P_{\mathrm{ext}}}.
\]

\begin{lemma}[Zero-mode identities]
\label{lem:zero-mode}
Let \((\eta,\psi,P_{\mathrm{ext}})\) be a solution of
\eqref{eq:hydroelastic-system} on \([0,T]\).  Then, for every
\(t\in[0,T]\), we have
\begin{align}
  \int_{\T}G(\eta)\psi\,\dd x&=0,
  \label{eq:G-zero}\\
  \int_{\T}\Bflex(\eta)\,\dd x&=0,
  \label{eq:B-zero}\\
  \int_{\T}N(\eta)\psi\,\dd x&=0.
  \label{eq:N-zero-mean}
\end{align}
Consequently, \(\int_{\T}\eta(t,x)\,\dd x\) is conserved.  In view of
\eqref{eq:zero-mean-initial}, we have
\[
  \int_{\T}\eta(t,x)\,\dd x=0
  \qquad\text{for all times of existence}.
\]
Moreover, we have
\begin{equation}\label{eq:mean-psi-evolution}
  \frac{\dd}{\dd t}\avg{\psi}
  +\avg{P_{\mathrm{ext}}}=0.
\end{equation}
\end{lemma}

\begin{proof}
We prove the three zero-mode identities separately.

For \eqref{eq:G-zero}, fix a time \(t\) and, for \(R>0\) sufficiently large,
consider the truncated domain
\[
  \Omega_R(t)
  =\{(x,y)\in\T\times\R:-R<y<\eta(t,x)\}.
\]
Since \(\Delta\phi=0\), the divergence theorem applied to \(\nabla\phi\)
gives
\[
  0
  =\iint_{\Omega_R(t)}\Delta\phi\,\dd y\dd x
  =\int_{\partial\Omega_R(t)}\partial_n\phi\,\dd S.
\]
On the free surface, with outward unit normal
\[
  n=\frac{(-\eta_x,1)}{\sqrt{1+\eta_x^2}},
  \qquad
  \dd S=\sqrt{1+\eta_x^2}\,\dd x,
\]
we have
\[
  \partial_n\phi\,\dd S
  =\bigl(\phi_y-\eta_x\phi_x\bigr)\big|_{y=\eta}\,\dd x
  =G(\eta)\psi\,\dd x.
\]
On the artificial lower boundary \(y=-R\), the outward normal is
\((0,-1)\), so its contribution is
\(-\int_{\T}\phi_y(x,-R)\,\dd x\).  Hence
\[
  \int_{\T}G(\eta)\psi\,\dd x
  =\int_{\T}\phi_y(x,-R)\,\dd x.
\]
The proof of Lemma~\ref{lem:deep-decay} shows that the zero Fourier mode of
\(\phi\) is independent of \(y\).  Hence
\[
  \int_{\T}\phi_y(x,-R)\,\dd x=0
\]
for all sufficiently large \(R\).  This proves \eqref{eq:G-zero}.

For \eqref{eq:B-zero}, integrate the graph-coordinate expression
\eqref{eq:flex-force} over one period.  Periodicity gives
\begin{align*}
  \int_{\T}\Bflex(\eta)\,\dd x
  &=\int_{\T}\partial_x^2\left(
    \frac{\eta_{xx}}{(1+\eta_x^2)^{5/2}}
  \right)\dd x
  +\frac52\int_{\T}\partial_x\left(
    \frac{\eta_x\eta_{xx}^2}{(1+\eta_x^2)^{7/2}}
  \right)\dd x
  =0.
\end{align*}

It remains to prove \eqref{eq:N-zero-mean}.  Introduce the vector field
\[
  X
  =\left(
    \phi_x\phi_y,
    \frac12(\phi_y^2-\phi_x^2)
  \right).
\]
A direct differentiation, using \(\Delta\phi=0\), yields
\[
  \operatorname{div}X
  =\phi_{xx}\phi_y+\phi_x\phi_{xy}
   +\phi_y\phi_{yy}-\phi_x\phi_{xy}
  =\phi_y(\phi_{xx}+\phi_{yy})
  =0.
\]
Apply the divergence theorem to \(X\) on \(\Omega_R(t)\).  On the free
surface it is convenient to use the non-normalized outward normal
\((-\eta_x,1)\).  The corresponding flux density is
\begin{align*}
  X\cdot(-\eta_x,1)
  &=-\eta_x\phi_x\phi_y
    +\frac12(\phi_y^2-\phi_x^2)\\
  &=-\left(
    \frac12V^2-\frac12B^2+\eta_xVB
  \right)
  =-N(\eta)\psi,
\end{align*}
where \(V=\phi_x|_{y=\eta}\) and \(B=\phi_y|_{y=\eta}\).
On the lower boundary \(y=-R\), the flux is
\[
  -\frac12\bigl(\phi_y^2-\phi_x^2\bigr)(x,-R).
\]
Lemma~\ref{lem:deep-decay} gives
\[
  \left|
  \frac12\int_{\T}
  \bigl(\phi_y^2-\phi_x^2\bigr)(x,-R)\,\dd x
  \right|
  \leq
  \frac12\|\nabla\phi(\cdot,-R)\|_{L^2(\T)}^2
  \leq C\mathrm{e}^{-2\pi R/L}.
\]
Thus the lower-boundary contribution tends to zero as \(R\to\infty\).
There are no side contributions on the periodic cylinder.  Equivalently, the
two vertical-side contributions on \([-L,L]\) cancel by periodicity.  Passing
to the limit therefore gives
\[
  0=-\int_{\T}N(\eta)\psi\,\dd x,
\]
which proves \eqref{eq:N-zero-mean}.

Finally, integrating the kinematic equation in \eqref{eq:hydroelastic-system}
and using \eqref{eq:G-zero} gives
\[
  \frac{\dd}{\dd t}\int_{\T}\eta\,\dd x
  =\int_{\T}G(\eta)\psi\,\dd x
  =0.
\]
Thus the spatial integral of \(\eta\) is conserved, and
\eqref{eq:zero-mean-initial} implies
\[
  \int_{\T}\eta(t,x)\,\dd x=0.
\]
Integrating the dynamic equation in \eqref{eq:hydroelastic-system} and using
\(\int_{\T}\eta(t,x)\,\dd x=0\), together with \eqref{eq:B-zero} and
\eqref{eq:N-zero-mean} yields
\[
  \frac{\dd}{\dd t}\int_{\T}\psi\,\dd x
  =-\int_{\T}P_{\mathrm{ext}}\,\dd x.
\]
Dividing by \(2L\) proves \eqref{eq:mean-psi-evolution}.
\end{proof}
\section{Construction of the localized feedback}\label{sec:feedback}

Our localization is adapted from Alazard's stabilization arguments for
gravity waves \cite{AlazardGravity} and for water waves with surface tension
\cite{AlazardCapillary}, where a localized pressure feedback is combined with a
spatial multiplier.  The multiplier is chosen here to agree with the dilation
field \(x\partial_x\) away
from the damping region and to vanish at the endpoints of the symmetry cell.
Lemma~\ref{lem:cutoff} also records the bounds on the localization profile
used later in the stabilization argument.

\begin{lemma}[Localized multiplier and damping profile]\label{lem:cutoff}
Fix \(0<\delta<L\).  There exists an even \(2L\)-periodic function
\(\vartheta\in C^\infty(\T)\) such that
\[
  0\leq\vartheta\leq1,\qquad
  \vartheta=1\ \text{on }|x|\leq L-\delta,\qquad
  \vartheta=0\ \text{near }x=\pm L,
\]
and \(x\vartheta_x(x)\leq0\) for \(0\leq x\leq L\).  Define
\begin{equation}\label{eq:m-chi-construction}
  m(x)=x\vartheta(x),
  \qquad
  \chi(x)=1-m_x(x)=1-\vartheta(x)-x\vartheta_x(x),\quad x\in[-L,L],
\end{equation}
and extend \(m\) periodically,  then \(m\) is smooth, odd, and
\(2L\)-periodic, while \(\chi\) is smooth, even, and nonnegative.  Moreover, we have
\begin{equation}\label{eq:m-assumptions}
  m(x)=x\quad\text{for }|x|\leq L-\delta,
  \qquad
  m(\pm L)=0,
  \qquad
  \chi=0\quad\text{for }|x|\leq L-\delta,
\end{equation}
\begin{equation}\label{eq:x-minus-m}
  |x-m(x)|\leq L\chi(x)
  \qquad (x\in[-L,L]),
\end{equation}
and
\begin{equation}\label{eq:M3-positive}
  \|m_{xxx}\|_{L^\infty(\T)}>0.
\end{equation}
\end{lemma}

\begin{proof}
	Choose \(\vartheta\in C^\infty([0,L])\) such that
	\[
	0\leq\vartheta\leq1,
	\qquad
	\vartheta=1
	\quad\text{on }[0,L-\delta],
	\]
	such that \(\vartheta=0\) in a neighborhood of \(L\), and such that
	\[
	\vartheta_x\leq0
	\qquad\text{on }[0,L].
	\]
	Such a cutoff is obtained by a standard smooth monotone transition between
	the values \(1\) and \(0\) inside \((L-\delta,L)\).  Since \(\vartheta\)
	is constant in neighborhoods of both \(0\) and \(L\), its even extension
	to \([-L,L]\) is smooth at \(x=0\), and the resulting function and all its
	derivatives agree at \(x=-L\) and \(x=L\).  Its \(2L\)-periodic extension
	therefore defines a smooth even function on \(\T\).
	
	For \(0\leq x\leq L\), we have \(x\geq0\) and
	\(\vartheta_x(x)\leq0\), then
	\[
	x\vartheta_x(x)\leq0.
	\]
	Define \(m\) and \(\chi\) by \eqref{eq:m-chi-construction}.  Since
	\(\vartheta\) is even, \(m(x)=x\vartheta(x)\) is odd on \([-L,L]\).
	Moreover, \(\vartheta\) vanishes near \(x=\pm L\), and hence \(m\) also
	vanishes there.  Thus its periodic extension is smooth and odd.  Since the
	derivative of an odd periodic function is even,
	\[
	\chi=1-m_x
	\]
	is smooth, even, and \(2L\)-periodic.
	
	Using
	\[
	m_x=\vartheta+x\vartheta_x,
	\]
	we have
	\[
	\chi
	=
	1-\vartheta-x\vartheta_x.
	\]
	On \(0\leq x\leq L\), both
	\[
	1-\vartheta(x)\geq0,
	\qquad
	-x\vartheta_x(x)\geq0,
	\]
	and therefore
	\[
	\chi(x)\geq0.
	\]
	Since \(\chi\) is even, the same conclusion holds on the whole
	representative interval \([-L,L]\).
	
	On \(|x|\leq L-\delta\), we have
	\(\vartheta=1\) and \(\vartheta_x=0\).  Hence
	\[
	m(x)=x,
	\qquad
	\chi(x)=0.
	\]
	Since \(\vartheta\) vanishes near \(x=\pm L\), we also have
	\[
	m(\pm L)=0.
	\]
	This proves \eqref{eq:m-assumptions}.
	
	We next prove \eqref{eq:x-minus-m}.  For \(0\leq x\leq L\),
	\[
	x-m(x)
	=
	x(1-\vartheta(x)).
	\]
	Because \(0\leq x\leq L\) and \(1-\vartheta\geq0\),
	\[
	0
	\leq
	x-m(x)
	\leq
	L(1-\vartheta(x)).
	\]
	Furthermore, \(-x\vartheta_x(x)\geq0\), and 	\[
	1-\vartheta(x)
	\leq
	1-\vartheta(x)-x\vartheta_x(x)
	=
	\chi(x).
	\]
	Consequently, we have
	\[
	|x-m(x)|
	=
	x-m(x)
	\leq
	L\chi(x),
	\qquad 0\leq x\leq L.
	\]
	Now \(x-m(x)\) is odd, whereas \(\chi\) is even.  Hence, for
	\(-L\leq x\leq0\),
	\[
	|x-m(x)|
	=
	|(-x)-m(-x)|
	\leq
	L\chi(-x)
	=
	L\chi(x).
	\]
	Thus \eqref{eq:x-minus-m} holds on all of \([-L,L]\).
	
	Finally, suppose for contradiction that
	\[
	m_{xxx}\equiv0
	\qquad\text{on }\T.
	\]
	Let \(\widetilde m\) denote the smooth \(2L\)-periodic lift of \(m\) to
	\(\R\).  Then
	\[
	\widetilde m'''(x)=0
	\qquad\text{for all }x\in\R,
	\]
	so
	\[
	\widetilde m(x)=a x^2+b x+c
	\]
	for some constants \(a,b,c\).  Since \(\widetilde m\) is periodic, then
	\[
	\widetilde m(x+2L)-\widetilde m(x)=0
	\qquad\text{for all }x\in\R.
	\]
	Expanding the left-hand side gives
	\[
	4aLx+4aL^2+2Lb=0
	\qquad\text{for all }x,
	\]
	and therefore \(a=b=0\).  Thus \(\widetilde m\) is constant.
	Because \(m\) is odd, this constant must be zero.  Hence
	\[
	m\equiv0,
	\]
	which contradicts
	\[
	m(x)=x
	\qquad\text{for }|x|\leq L-\delta,
	\]
	where \(L-\delta>0\).  Therefore
	\[
	\|m_{xxx}\|_{L^\infty(\T)}>0,
	\]
	which proves \eqref{eq:M3-positive}.
\end{proof}

In analogy with the localized velocity-pressure feedback used in
\cite{AlazardGravity,AlazardCapillary}, we define the exterior pressure by
\begin{equation}\label{eq:feedback}
  P_{\mathrm{ext}}(t,x)=\lambda\chi(x)\eta_t(t,x),
  \qquad \lambda>0.
\end{equation}
Because \(\chi\) vanishes on \(|x|\leq L-\delta\), the pressure is localized
near the endpoints of the symmetry cell.  Near \(x=\pm L\), one has
\(\chi=1\) and \(\chi_x=0\).  Evenness gives
\(\partial_x\eta_t=0\) there, and hence
\(\partial_xP_{\mathrm{ext}}=0\).  Thus the feedback law is compatible with
the even periodic symmetry class.  Its sign at the Hamiltonian level follows
directly from the energy identity.

\begin{proposition}[Energy dissipation]\label{prop:dissipation}
Every solution of
\eqref{eq:hydroelastic-system}--\eqref{eq:feedback} satisfies
\begin{equation}\label{eq:energy-dissipation}
  \frac{\dd}{\dd t}\Etot(t)
  =
  -\lambda\int_{\T}\chi(x)\eta_t(t,x)^2\,\dd x
  \leq0.
\end{equation}
Consequently, we have
\begin{equation}\label{eq:damping-integral}
  \lambda\int_0^T\int_{\T}
  \chi\bigl(G(\eta)\psi\bigr)^2\,\dd x\dd t
  \leq \Etot(0).
\end{equation}
\end{proposition}

\begin{proof}
By the general balance law \eqref{eq:general-energy-balance}, we have
\[
  \frac{\dd}{\dd t}\Etot(t)
  =-\int_{\T}P_{\mathrm{ext}}\eta_t\,\dd x.
\]
Together with \eqref{eq:feedback} gives
\[
  \frac{\dd}{\dd t}\Etot(t)
  =-\lambda\int_{\T}\chi\eta_t^2\,\dd x\leq0,
\]
because \(\lambda>0\) and \(\chi\geq0\).  Integration from \(0\) to
\(T\) yields
\[
  \lambda\int_0^T\int_{\T}\chi\eta_t^2\,\dd x\dd t
  =\Etot(0)-\Etot(T)\leq\Etot(0).
\]
Finally, the kinematic equation \(\eta_t=G(\eta)\psi\) gives
\eqref{eq:damping-integral}.
\end{proof}

\section{The localized flexural virial identity}\label{sec:flex-virial}

The energy identity controls only the damped normal velocity.  Recovering the
full hydroelastic energy requires a multiplier whose elastic contribution has a
positive fourth-order principal part.  The multiplier is identified at the
linearized level, its nonlinear contribution is then computed exactly, and a
quantitative coercive lower bound is obtained.

\subsection{Choice of the elastic multiplier}

The multiplier adapted to the fourth-order elastic energy is
\begin{equation}\label{eq:zeta-rho}
  \zeta
  =
  \partial_x(m\eta)
  +
  \frac52(1-m_x)\eta
  -
  \frac14\eta
  =
  m\eta_x+\left(\frac94-\frac32m_x\right)\eta,
\end{equation}
and
\begin{equation}\label{eq:rho}
  \varrho
  =
  \zeta+\eta-x\eta_x
  =
  (m-x)\eta_x
  +
  \left(\frac{13}{4}-\frac32m_x\right)\eta.
\end{equation}

The following flat-state calculation explains the choice of the coefficient
\(5/2\) in \eqref{eq:zeta-rho}. This value is uniquely determined by the
requirement that the coefficient of the highest-order quadratic term be
independent of the localization derivative \(m_x\). This condition is essential for the coercivity of the bending contribution,
since the localization function \(\chi\) (equivalently \(m_x=1-\chi\))
should enter only through the damping mechanism and lower-order terms, rather
than the principal elastic energy.

\begin{lemma}
\label{lem:unique-flex-multiplier}
For \(r\in\R\), set
\[
  \zeta_r
  =\partial_x(m\eta)+r(1-m_x)\eta-\frac14\eta.
\]
The complete quadratic identity at the flat state is
\begin{align}
  \int_{\T}\Bflex'(0)\eta\,\zeta_r\,\dd x
  &=\int_{\T}
  \left[
    r-\frac14+\left(\frac52-r\right)m_x
  \right]\eta_{xx}^2\,\dd x
  \notag\\
  &\quad
  +(3-2r)\int_{\T}m_{xx}\eta_x\eta_{xx}\,\dd x
  +(1-r)\int_{\T}m_{xxx}\eta\eta_{xx}\,\dd x.
  \label{eq:quadratic-r}
\end{align}
Consequently, \(r=5/2\) is the unique choice for which the coefficient of
the highest-order quadratic term \(\eta_{xx}^2\) is independent of \(m_x\).
With this choice, the coefficient reduces to \(9/4\).
\end{lemma}

\begin{proof}
Since \(\Bflex'(0)\eta=\eta_{xxxx}\), we have
\[
  \zeta_r
  =
  m\eta_x+
  \left(r-\frac14+(1-r)m_x\right)\eta.
\]
For the first part, two periodic integrations by parts give
\[
  \int_{\T}\eta_{xxxx}m\eta_x\,\dd x
  =
  \frac32\int_{\T}m_x\eta_{xx}^2\,\dd x
  +
  \int_{\T}m_{xx}\eta_x\eta_{xx}\,\dd x.
\]
For the second part, applying the same integrations by parts directly to
\(\left(r-\frac14+(1-r)m_x\right)\eta\) yields
\begin{align*}
  &\int_{\T}
  \eta_{xxxx}
  \left(r-\frac14+(1-r)m_x\right)\eta\,\dd x\\
  =&
  \int_{\T}
  \left(r-\frac14+(1-r)m_x\right)\eta_{xx}^2\,\dd x
  +2(1-r) \int_{\T}m_{xx}\eta_x\eta_{xx}\,\dd x
  +(1-r)\int_{\T}m_{xxx}\eta\eta_{xx}\,\dd x.
\end{align*}
Combining the two identities and collecting like terms, we obtain
\eqref{eq:quadratic-r}.  Notice that the coefficient of the
\(m_x\eta_{xx}^2\)-term is \(\frac52-r\).  Hence the dependence of the
principal quadratic term on \(m_x\) disappears precisely for \(r=5/2\), in
which case its coefficient reduces to \(9/4\).
\end{proof}

\subsection{Exact nonlinear elastic identity}

The linear calculation fixes the multiplier.  Its contribution to the full
nonlinear elastic force is given by the following identity.

\begin{proposition}[Exact localized flexural virial formula]
\label{prop:flex-virial}
With \(\zeta\) defined by \eqref{eq:zeta-rho}, then the following holds
\begin{align}
  \int_{\T}\Bflex(\eta)\zeta\,\dd x
  &=
  \int_{\T}
  \frac{\eta_{xx}^2}{(1+\eta_x^2)^{7/2}}
  \left[
    \frac94
    -
    \left(\frac{27}{8}-\frac54m_x\right)\eta_x^2
  \right]\dd x
  \notag\\
  &\quad
  -2\int_{\T}
  m_{xx}
  \frac{\eta_x\eta_{xx}}{(1+\eta_x^2)^{5/2}}\,\dd x
  -\frac32\int_{\T}
  m_{xxx}
  \frac{\eta\eta_{xx}}{(1+\eta_x^2)^{5/2}}\,\dd x
  \notag\\
  &\quad
  +\frac{15}{4}\int_{\T}
  m_{xx}
  \frac{\eta\eta_x\eta_{xx}^2}
       {(1+\eta_x^2)^{7/2}}\,\dd x.
  \label{eq:flex-virial}
\end{align}
\end{proposition}

\begin{proof}
By \eqref{eq:zeta-rho}, we have
\[
  \zeta
  =
  m\eta_x+
  \left(\frac94-\frac32m_x\right)\eta.
\]
We evaluate these two contributions separately, always using
\eqref{eq:first-variation} in its original \(\eta\)-variables.

For \(v=m\eta_x\),
\[
  v_x=m_x\eta_x+m\eta_{xx},
  \qquad
  v_{xx}=m_{xx}\eta_x+2m_x\eta_{xx}+m\eta_{xxx}.
\]
Therefore
\begin{align*}
  \int_{\T}\Bflex(\eta)m\eta_x\,\dd x
  &=
  \int_{\T}\left[
    -\frac52m_x
    \frac{\eta_x^2\eta_{xx}^2}{(1+\eta_x^2)^{7/2}}
    -\frac52m
    \frac{\eta_x\eta_{xx}^3}{(1+\eta_x^2)^{7/2}}
  \right]\dd x
  \\
  &\quad
  +\int_{\T}\left[
    m_{xx}\frac{\eta_x\eta_{xx}}{(1+\eta_x^2)^{5/2}}
    +2m_x\frac{\eta_{xx}^2}{(1+\eta_x^2)^{5/2}}
    +m\frac{\eta_{xx}\eta_{xxx}}{(1+\eta_x^2)^{5/2}}
  \right]\dd x.
\end{align*}
The two terms containing \(m\) combine because
\[
  \partial_x\left(
    \frac12
    \frac{\eta_{xx}^2}{(1+\eta_x^2)^{5/2}}
  \right)
  =
  \frac{\eta_{xx}\eta_{xxx}}{(1+\eta_x^2)^{5/2}}
  -\frac52
  \frac{\eta_x\eta_{xx}^3}{(1+\eta_x^2)^{7/2}}.
\]
Integrating this total derivative by parts, we obtain
\begin{align}
  \int_{\T}\Bflex(\eta)m\eta_x\,\dd x
  &=
  \int_{\T}
  m_x
  \frac{\eta_{xx}^2}{(1+\eta_x^2)^{7/2}}
  \left(\frac32-\eta_x^2\right)\dd x
  +\int_{\T}
  m_{xx}
  \frac{\eta_x\eta_{xx}}{(1+\eta_x^2)^{5/2}}\,\dd x.
  \label{eq:m-etax-part}
\end{align}

For
\[
  v=\left(\frac94-\frac32m_x\right)\eta,
\]
direct differentiation gives
\begin{align*}
  v_x
  &=
  \left(\frac94-\frac32m_x\right)\eta_x
  -\frac32m_{xx}\eta,\\
  v_{xx}
  &=
  \left(\frac94-\frac32m_x\right)\eta_{xx}
  -3m_{xx}\eta_x
  -\frac32m_{xxx}\eta.
\end{align*}
Substituting the above into \eqref{eq:first-variation} yields
\begin{align}
  \int_{\T}
  \Bflex(\eta)
  \left(\frac94-\frac32m_x\right)\eta\,\dd x
  \notag&=
  \int_{\T}
  \left(\frac94-\frac32m_x\right)
  \frac{\eta_{xx}^2}{(1+\eta_x^2)^{7/2}}
  \left(1-\frac32\eta_x^2\right)\dd x
  \notag\\
  &\qquad
  -3\int_{\T}
  m_{xx}
  \frac{\eta_x\eta_{xx}}{(1+\eta_x^2)^{5/2}}\,\dd x
  -\frac32\int_{\T}
  m_{xxx}
  \frac{\eta\eta_{xx}}{(1+\eta_x^2)^{5/2}}\,\dd x
  \notag\\
  &\qquad
  +\frac{15}{4}\int_{\T}
  m_{xx}
  \frac{\eta\eta_x\eta_{xx}^2}
       {(1+\eta_x^2)^{7/2}}\,\dd x.
  \label{eq:weighted-eta-part}
\end{align}
Combining \eqref{eq:m-etax-part} and
\eqref{eq:weighted-eta-part}, note first that the
\(m_{xx}\eta_x\eta_{xx}\)-terms reduce to
\[
  -2\int_{\T}
  m_{xx}\frac{\eta_x\eta_{xx}}
  {(1+\eta_x^2)^{5/2}}\,\dd x.
\]
The remaining principal terms simplify according to
\begin{align*}
  m_x\left(\frac32-\eta_x^2\right)
  +
  \left(\frac94-\frac32m_x\right)
  \left(1-\frac32\eta_x^2\right)
  =
  \frac94-
  \left(\frac{27}{8}-\frac54m_x\right)\eta_x^2.
\end{align*}
Substituting above two simplifications proves \eqref{eq:flex-virial}.
\end{proof}

\subsection{Coercivity of the elastic contribution}

The exact identity contains three lower-order terms involving derivatives of
\(m\).  The following estimate quantifies their effect and gives a precise
condition under which the fourth-order principal term remains positive.

\begin{lemma}[Coercivity of the flexural virial form]
\label{lem:flex-coercivity}
Set
\[
  M_1=\|m_x\|_{L^\infty(\T)},
  \qquad
  M_2=\|m_{xx}\|_{L^\infty(\T)},
  \qquad
  M_3=\|m_{xxx}\|_{L^\infty(\T)},
  \qquad
  K_1=\frac{27}{8}+\frac54M_1.
\]
Assume
\[
  \|\eta_x\|_{L^\infty}\leq\varepsilon_1,
  \qquad
  \|\eta\|_{L^\infty}\leq\varepsilon_0,
  \qquad
  \frac94-K_1\varepsilon_1^2\geq0.
\]
For arbitrary \(d_1,d_2>0\), define
\begin{align}
  c_{\mathrm{fl}}
  &:={}
  \frac{\frac94-K_1\varepsilon_1^2}{1+\varepsilon_1^2}
  -\frac{M_3d_1}{2}(1+\varepsilon_1^2)^{5/2}
  -\frac{3M_3d_2}{4}
  -\frac{15M_2}{4}\varepsilon_0\varepsilon_1,
  \label{eq:cfl}\\
  C_{\mathrm{fl}}
  &:={}
  \frac{M_3}{2d_1}+\frac{3M_3}{4d_2}.
  \label{eq:Cfl}
\end{align}
Then
\begin{equation}\label{eq:flex-coercivity}
  \int_{\T}\Bflex(\eta)\zeta\,\dd x
  \geq
  c_{\mathrm{fl}}
  \int_{\T}\frac{\eta_{xx}^2}{(1+\eta_x^2)^{5/2}}\,\dd x
  -C_{\mathrm{fl}}\int_{\T}\eta^2\,\dd x.
\end{equation}
\end{lemma}

\begin{proof}
We estimate the four terms on the right-hand side of
\eqref{eq:flex-virial} directly in the variables \(\eta_x\) and
\(\eta_{xx}\).

For the principal integral, the bounds
\(|m_x|\leq M_1\) and
\(\|\eta_x\|_{L^\infty}\leq\varepsilon_1\) imply
\[
  \frac94-
  \left(\frac{27}{8}-\frac54m_x\right)\eta_x^2
  \geq
  \frac94-K_1\varepsilon_1^2.
\]
Moreover, we have
\[
  1+\eta_x^2\leq1+\varepsilon_1^2,
\]
and hence
\[
  \frac1{(1+\eta_x^2)^{7/2}}
  \geq
  \frac1{1+\varepsilon_1^2}
  \frac1{(1+\eta_x^2)^{5/2}}.
\]
Since \(\frac94-K_1\varepsilon_1^2\geq0\), it follows that
\[
  \int_{\T}
  \frac{\eta_{xx}^2}{(1+\eta_x^2)^{7/2}}
  \left[
    \frac94-
    \left(\frac{27}{8}-\frac54m_x\right)\eta_x^2
  \right]\dd x
  \geq
  \frac{\frac94-K_1\varepsilon_1^2}{1+\varepsilon_1^2}
  \int_{\T}
  \frac{\eta_{xx}^2}{(1+\eta_x^2)^{5/2}}\,\dd x.
\]

For the second term in \eqref{eq:flex-virial}, observe that
\[
  \partial_x\left((1+\eta_x^2)^{-3/2}\right)
  =
  -3\frac{\eta_x\eta_{xx}}{(1+\eta_x^2)^{5/2}}.
\]
Therefore, we obtain
\begin{align*}
  -2\int_{\T}
  m_{xx}
  \frac{\eta_x\eta_{xx}}{(1+\eta_x^2)^{5/2}}\,\dd x
  &=
  \frac23\int_{\T}
  m_{xx}
  \partial_x\left((1+\eta_x^2)^{-3/2}\right)\dd x
  \\
  &=
  -\frac23\int_{\T}
  m_{xxx}(1+\eta_x^2)^{-3/2}\,\dd x
  \\
  &=
  -\frac23\int_{\T}
  m_{xxx}
  \left[(1+\eta_x^2)^{-3/2}-1\right]\dd x,
\end{align*}
where the last equality uses
\(\int_{\T}m_{xxx}\,\dd x=0\).  Since
\(z\mapsto(1+z)^{-3/2}\) has derivative bounded in absolute value by
\(3/2\) for \(z\geq0\), the mean value theorem gives
\[
  \left|(1+\eta_x^2)^{-3/2}-1\right|
  \leq\frac32\eta_x^2.
\]
Consequently, we have
\[
  \left|
  -2\int_{\T}
  m_{xx}
  \frac{\eta_x\eta_{xx}}{(1+\eta_x^2)^{5/2}}\,\dd x
  \right|
  \leq
  M_3\int_{\T}\eta_x^2\,\dd x.
\]
By periodicity, it has
\[
  \int_{\T}\eta_x^2\,\dd x
  =
  -\int_{\T}\eta\eta_{xx}\,\dd x.
\]
Using Young's inequality and
\[
  \eta_{xx}^2
  \leq
  (1+\varepsilon_1^2)^{5/2}
  \frac{\eta_{xx}^2}{(1+\eta_x^2)^{5/2}},
\]
we obtain
\begin{align*}
  -2\int_{\T}
  m_{xx}
  \frac{\eta_x\eta_{xx}}{(1+\eta_x^2)^{5/2}}\,\dd x
  &\geq
  -\frac{M_3d_1}{2}(1+\varepsilon_1^2)^{5/2}
  \int_{\T}
  \frac{\eta_{xx}^2}{(1+\eta_x^2)^{5/2}}\,\dd x
  -\frac{M_3}{2d_1}\int_{\T}\eta^2\,\dd x.
\end{align*}

For the third term in \eqref{eq:flex-virial}, weighted Young's inequality
gives
\begin{align*}
  \left|
  -\frac32\int_{\T}
  m_{xxx}
  \frac{\eta\eta_{xx}}{(1+\eta_x^2)^{5/2}}\,\dd x
  \right|
  &\leq
  \frac32M_3
  \int_{\T}
  \left(
    \frac{|\eta_{xx}|}{(1+\eta_x^2)^{5/4}}
  \right)
  \left(
    \frac{|\eta|}{(1+\eta_x^2)^{5/4}}
  \right)\dd x
  \\
  &\leq
  \frac{3M_3d_2}{4}
  \int_{\T}
  \frac{\eta_{xx}^2}{(1+\eta_x^2)^{5/2}}\,\dd x
  +
  \frac{3M_3}{4d_2}
  \int_{\T}
  \frac{\eta^2}{(1+\eta_x^2)^{5/2}}\,\dd x
  \\
  &\leq
  \frac{3M_3d_2}{4}
  \int_{\T}
  \frac{\eta_{xx}^2}{(1+\eta_x^2)^{5/2}}\,\dd x
  +
  \frac{3M_3}{4d_2}
  \int_{\T}\eta^2\,\dd x.
\end{align*}

Finally, we have
\[
  |\eta\eta_x|
  \leq\varepsilon_0\varepsilon_1,
  \qquad
  \frac1{(1+\eta_x^2)^{7/2}}
  \leq
  \frac1{(1+\eta_x^2)^{5/2}},
\]
so the last term in \eqref{eq:flex-virial} satisfies
\[
  \frac{15}{4}\int_{\T}
  m_{xx}
  \frac{\eta\eta_x\eta_{xx}^2}
       {(1+\eta_x^2)^{7/2}}\,\dd x
  \geq
  -\frac{15M_2}{4}\varepsilon_0\varepsilon_1
  \int_{\T}
  \frac{\eta_{xx}^2}{(1+\eta_x^2)^{5/2}}\,\dd x.
\]
Combining the lower bounds obtained for the four terms on the right-hand side of
\eqref{eq:flex-virial} yields \eqref{eq:flex-coercivity}, with \(c_{\mathrm{fl}}\) and
\(C_{\mathrm{fl}}\) given by \eqref{eq:cfl}--\eqref{eq:Cfl}.
\end{proof}
The constants in Lemma~\ref{lem:flex-coercivity} still contain the two free Young parameters
\(d_1\) and \(d_2\).  We now choose them so that, for a prescribed loss
\(\theta R_{\mathrm{fl}}\) in the principal flexural coefficient, the
lower-order constant \(C_{\mathrm{fl}}\) is minimized.
\begin{proposition}
\label{prop:flex-parameter-feasibility}
With the notation of Lemma~\ref{lem:flex-coercivity}, set
\begin{equation}\label{eq:Rfl}
  R_{\mathrm{fl}}
  :=
  \frac{\frac94-K_1\varepsilon_1^2}{1+\varepsilon_1^2}
  -\frac{15M_2}{4}\varepsilon_0\varepsilon_1.
\end{equation}
By \eqref{eq:M3-positive}, \(M_3>0\).  Assume \(R_{\mathrm{fl}}>0\), and fix
\(\theta\in(0,1)\).  Define
\begin{equation}\label{eq:KY}
  K_Y(\varepsilon_1)
  :=M_3\left(
      \frac12(1+\varepsilon_1^2)^{5/4}+\frac34
    \right)
\end{equation}
and choose
\begin{equation}\label{eq:optimized-d1-d2}
  d_1
  =\frac{\theta R_{\mathrm{fl}}}
  {K_Y(\varepsilon_1)(1+\varepsilon_1^2)^{5/4}},
  \qquad
  d_2
  =\frac{\theta R_{\mathrm{fl}}}{K_Y(\varepsilon_1)}.
\end{equation}
Then the constants in Lemma~\ref{lem:flex-coercivity} are
\begin{equation}\label{eq:optimized-cfl-Cfl}
  c_{\mathrm{fl}}=(1-\theta)R_{\mathrm{fl}},
  \qquad
  C_{\mathrm{fl}}
  =\frac{K_Y(\varepsilon_1)^2}{\theta R_{\mathrm{fl}}}.
\end{equation}
Moreover, among all \(d_1,d_2>0\) for which the total Young loss in
\eqref{eq:cfl} equals \(\theta R_{\mathrm{fl}}\), namely
\begin{equation}\label{eq:fixed-young-loss}
  \frac{M_3d_1}{2}(1+\varepsilon_1^2)^{5/2}
  +\frac{3M_3d_2}{4}
  =\theta R_{\mathrm{fl}},
\end{equation}
the choice \eqref{eq:optimized-d1-d2} minimizes \(C_{\mathrm{fl}}\).
\end{proposition}

\begin{proof}
Substituting \eqref{eq:optimized-d1-d2} into the two Young-loss terms in
\eqref{eq:cfl} gives
\begin{align*}
  \frac{M_3d_1}{2}(1+\varepsilon_1^2)^{5/2}
  +\frac{3M_3d_2}{4}
 =
  \theta R_{\mathrm{fl}}
  \frac{
    \frac{M_3}{2}(1+\varepsilon_1^2)^{5/4}
    +\frac{3M_3}{4}
  }{
    K_Y(\varepsilon_1)
  }
  =
  \theta R_{\mathrm{fl}},
\end{align*}
Indeed, by \eqref{eq:KY}, the numerator in the middle fraction is
\(K_Y(\varepsilon_1)\).  Therefore \eqref{eq:cfl} yields
\[
  c_{\mathrm{fl}}
  =
  R_{\mathrm{fl}}-\theta R_{\mathrm{fl}}
  =
  (1-\theta)R_{\mathrm{fl}}.
\]
Substituting \eqref{eq:optimized-d1-d2} into the definition
\eqref{eq:Cfl} gives
\begin{align*}
  C_{\mathrm{fl}}
  &=
  \frac{M_3}{2d_1}+\frac{3M_3}{4d_2}
  \\
  &=
  \frac{K_Y(\varepsilon_1)}{\theta R_{\mathrm{fl}}}
  \left[
    \frac{M_3}{2}(1+\varepsilon_1^2)^{5/4}
    +\frac{3M_3}{4}
  \right]
  \\
  &=
  \frac{K_Y(\varepsilon_1)^2}{\theta R_{\mathrm{fl}}},
\end{align*}
Here the bracket in the second line is \(K_Y(\varepsilon_1)\) by
\eqref{eq:KY}.  This proves the second identity in
\eqref{eq:optimized-cfl-Cfl}.

To verify optimality, let \(d_1,d_2>0\) satisfy
\eqref{eq:fixed-young-loss}.  The Cauchy--Schwarz inequality gives
\begin{equation}\label{eq:young-parameter-cauchy-schwarz}
\begin{aligned}
  \left[
    \frac{M_3d_1}{2}(1+\varepsilon_1^2)^{5/2}
    +\frac{3M_3d_2}{4}
  \right]
  \left[
    \frac{M_3}{2d_1}
    +\frac{3M_3}{4d_2}
  \right]
\geq
  \left[
    \frac{M_3}{2}(1+\varepsilon_1^2)^{5/4}
    +\frac{3M_3}{4}
  \right]^2
  =
  K_Y(\varepsilon_1)^2.
\end{aligned}
\end{equation}
By \eqref{eq:fixed-young-loss}, the first factor on the left-hand side of
\eqref{eq:young-parameter-cauchy-schwarz} is \(\theta R_{\mathrm{fl}}\),
whereas the second factor is \(C_{\mathrm{fl}}\) by \eqref{eq:Cfl}.
Therefore
\[
  C_{\mathrm{fl}}
  \geq
  \frac{K_Y(\varepsilon_1)^2}
       {\theta R_{\mathrm{fl}}}.
\]
Equality in \eqref{eq:young-parameter-cauchy-schwarz} holds precisely when
\[
  d_2=(1+\varepsilon_1^2)^{5/4}d_1.
\]
Combining this equality condition with \eqref{eq:fixed-young-loss} gives
\eqref{eq:optimized-d1-d2}.  Hence the choice
\eqref{eq:optimized-d1-d2} minimizes \(C_{\mathrm{fl}}\) among all pairs
satisfying the prescribed total Young loss \eqref{eq:fixed-young-loss}.
\end{proof}

At the flat interface, the optimized splitting gives an explicit criterion
for the simultaneous positivity of the flexural and gravitational
coefficients.

\begin{remark}
	\label{rem:flat-threshold-optimal}
	Set \(\varepsilon_0=\varepsilon_1=0\).  Then
	\[
	R_{\mathrm{fl}}=\frac94,
	\qquad
	K_Y(0)=\frac54M_3.
	\]
	Hence Proposition~\ref{prop:flex-parameter-feasibility} gives
	\[
	c_{\mathrm{fl}}
	=\frac94(1-\theta),
	\qquad
	C_{\mathrm{fl}}
	=\frac{25M_3^2}{36\theta},
	\qquad 0<\theta<1.
	\]
	The two coefficients required in the coercivity estimate are positive
	precisely when
	\[
	c_{\mathrm{fl}}>0,
	\qquad
	\frac g4-\beta C_{\mathrm{fl}}>0.
	\]
	The first condition is equivalent to \(\theta<1\), while the second is
	equivalent to
	\[
	\theta>\frac{25\beta M_3^2}{9g}.
	\]
	Therefore such a choice of \(\theta\in(0,1)\) exists if and only if
	\[
	g>\frac{25}{9}\beta M_3^2.
	\]
	Thus \(25/9\) is the exact flat-interface threshold produced by the optimized
	two-parameter Young splitting of
	Proposition~\ref{prop:flex-parameter-feasibility}.  It is a threshold for
	this coercivity argument, rather than a stability threshold for the
	hydroelastic system itself.
\end{remark}

\section{Fluid multiplier identities}\label{sec:fluid-identities}

We begin by fixing the notation and symmetry properties used in the
multiplier identities of this section.  Set
\[
\mathcal H_g(t)
=
\frac12\int_{\T}\psi G(\eta)\psi\,\dd x
+\frac g2\int_{\T}\eta^2\,\dd x,
\]
and recall that
\[
\zeta
=
\partial_x(m\eta)+\frac52\chi\eta-\frac14\eta,
\qquad
\varrho
=
\zeta+\eta-x\eta_x.
\]
Whenever the coordinate \(x\) appears explicitly, it is understood on the
representative cell \([-L,L]\).  The function \(\zeta\) is \(2L\)-periodic,
whereas \(\varrho\), because of the term \(x\eta_x\), is regarded as a smooth
function on the closed cell \([-L,L]\).

The even symmetry of \(\eta\) and \(\psi\) is inherited by the harmonic
extension.  Indeed, the reflected function
\[
(x,y)\longmapsto\phi(-x,y)
\]
is harmonic in the same fluid domain, has the same Dirichlet trace on the
free surface, and belongs to the same finite-energy class.  Uniqueness of the
finite-energy Dirichlet problem therefore yields
\[
\phi(-x,y)=\phi(x,y).
\]
Hence \(\phi_x\) is odd, and
\[
\phi_x(0,y)=\phi_x(\pm L,y)=0,
\]
where the identities at \(x=\pm L\) also use the \(2L\)-periodicity of the
flow.  Moreover,
\[
G(\eta)\psi=\phi_y-\eta_x\phi_x
\]
is even.  These symmetry properties will be used in the boundary
cancellations arising in the Pohozaev and weighted stress identities.

Define the nonnegative boundary contribution
\begin{equation}\label{eq:Sigma-infty}
  \Sigma_\infty(t)
  =L\int_{-\infty}^{\eta(t,L)}\phi_y(t,L,y)^2\,\dd y\geq0.
\end{equation}

\subsection{Preliminary multiplier identities}

We first derive the identities needed to reorganize the gravitational energy
in the multiplier argument.  Combining the basic multiplier relation with a
weighted equipartition formula yields the intermediate gravity identity
\eqref{eq:intermediate-gravity}, which provides the starting point for the
Pohozaev reduction for the infinite-depth domain in Lemma~\ref{lem:pohozaev}.

\begin{lemma}[Basic multiplier identity]\label{lem:basic-multiplier}
Assume
\begin{equation}\label{eq:forced-gravity-system}
  \eta_t=G(\eta)\psi,
  \qquad
  \psi_t+g\eta+N(\eta)\psi=-F.
\end{equation}
Let
\(\Theta=-\eta\psi_t-\frac g2\eta^2\) and
\[
  R_a=\int_0^T\int_{\T}
  \left[G(\eta)\psi\,m\psi_x+N(\eta)\psi\,m\eta_x\right]
  \dd x\dd t.
\]
Then
\begin{align}
  \int_0^T\int_{\T}m_x\Theta\,\dd x\dd t+R_a
  &=-\left[\int_{\T}\partial_x(m\eta)\psi\,\dd x\right]_0^T
  -\int_0^T\int_{\T}Fm\eta_x\,\dd x\dd t.
  \label{eq:basic-multiplier}
\end{align}
\end{lemma}

\begin{proof}
	Set
	\[
	\mathcal A
	:=
	\int_0^T\int_{\T}
	\left(
	\eta_t m\psi_x-\psi_t m\eta_x
	\right)\,\dd x\dd t.
	\]
	We first rewrite \(\mathcal A\) by integration by parts.  Since \(m\) is
	independent of time,
	\begin{align*}
		\int_0^T\int_{\T}\eta_t m\psi_x\,\dd x\dd t
		&=
		\left[
		\int_{\T}m\eta\psi_x\,\dd x
		\right]_0^T
		-
		\int_0^T\int_{\T}m\eta\psi_{xt}\,\dd x\dd t.
	\end{align*}
	Using periodicity to integrate both terms in \(x\), we obtain
	\begin{align*}
		\left[
		\int_{\T}m\eta\psi_x\,\dd x
		\right]_0^T
		&=
		-\left[
		\int_{\T}\partial_x(m\eta)\psi\,\dd x
		\right]_0^T,
	\end{align*}
	and
	\begin{align*}
		-\int_0^T\int_{\T}m\eta\psi_{xt}\,\dd x\dd t
		&=
		\int_0^T\int_{\T}
		\partial_x(m\eta)\psi_t\,\dd x\dd t
		\\
		&=
		\int_0^T\int_{\T}
		\left(
		m_x\eta+m\eta_x
		\right)\psi_t\,\dd x\dd t.
	\end{align*}
	Substituting these identities into the definition of \(\mathcal A\), the
	term involving \(m\eta_x\psi_t\) cancels with the second term in
	\(\mathcal A\).  Hence
	\begin{equation}\label{eq:A-first-representation}
		\mathcal A
		=
		-\left[
		\int_{\T}\partial_x(m\eta)\psi\,\dd x
		\right]_0^T
		+
		\int_0^T\int_{\T}
		m_x\eta\psi_t\,\dd x\dd t.
	\end{equation}
	
	We next compute \(\mathcal A\) directly from
	\eqref{eq:forced-gravity-system}.  Since
	\[
	\eta_t=G(\eta)\psi,
	\qquad
	-\psi_t=g\eta+N(\eta)\psi+F,
	\]
	we have
	\begin{align*}
		\mathcal A
		=
		\int_0^T\int_{\T}
		G(\eta)\psi\,m\psi_x\,\dd x\dd t
		+
		\int_0^T\int_{\T}
		\left(
		g\eta+N(\eta)\psi+F
		\right)m\eta_x\,\dd x\dd t.
	\end{align*}
	By the definition of \(R_a\), this becomes
	\begin{align*}
		\mathcal A
		&=
		R_a
		+
		g\int_0^T\int_{\T}
		m\eta\eta_x\,\dd x\dd t
		+
		\int_0^T\int_{\T}
		Fm\eta_x\,\dd x\dd t.
	\end{align*}
	Since
	\[
	\eta\eta_x=\frac12\partial_x(\eta^2),
	\]
	periodicity gives
	\[
	\int_{\T}m\eta\eta_x\,\dd x
	=
	-\frac12\int_{\T}m_x\eta^2\,\dd x.
	\]
	Therefore
	\begin{equation}\label{eq:A-second-representation}
		\mathcal A
		=
		R_a
		-\frac g2
		\int_0^T\int_{\T}
		m_x\eta^2\,\dd x\dd t
		+
		\int_0^T\int_{\T}
		Fm\eta_x\,\dd x\dd t.
	\end{equation}
	
	Equations \eqref{eq:A-first-representation} and
	\eqref{eq:A-second-representation} yield
	\begin{align*}
		-\int_0^T\int_{\T}
		m_x\eta\psi_t\,\dd x\dd t
		-\frac g2
		\int_0^T\int_{\T}
		m_x\eta^2\,\dd x\dd t
		+R_a
		&=
		-\left[
		\int_{\T}\partial_x(m\eta)\psi\,\dd x
		\right]_0^T\\
		&\quad-
		\int_0^T\int_{\T}
		Fm\eta_x\,\dd x\dd t.
	\end{align*}
	Recalling that
	\[
	\Theta
	=
	-\eta\psi_t-\frac g2\eta^2,
	\]
	the left-hand side is precisely
	\[
	\int_0^T\int_{\T}m_x\Theta\,\dd x\dd t+R_a.
	\]
	This proves \eqref{eq:basic-multiplier}.
\end{proof}

\begin{lemma}[Weighted equipartition]\label{lem:weighted-equipartition}
For every smooth periodic weight \(w=w(x)\),
\begin{align}
  \frac g2\int_0^T\int_{\T}w\eta^2\,\dd x\dd t
  &=\frac12\int_0^T\int_{\T}w\psi G(\eta)\psi\,\dd x\dd t
  -\frac12\int_0^T\int_{\T}w\eta F\,\dd x\dd t
  \notag\\
  &\quad-\frac12\left[\int_{\T}w\eta\psi\,\dd x\right]_0^T
  -\frac12\int_0^T\int_{\T}w\eta N(\eta)\psi\,\dd x\dd t.
  \label{eq:weighted-equipartition}
\end{align}
\end{lemma}

\begin{proof}
Since \(w\) is independent of time and \(\eta_t=G(\eta)\psi\),
\begin{align*}
  \int_0^T\int_{\T}w\psi G(\eta)\psi\,\dd x\dd t
  &=\int_0^T\int_{\T}w\psi\eta_t\,\dd x\dd t\\
  &=\left[\int_{\T}w\eta\psi\,\dd x\right]_0^T
    -\int_0^T\int_{\T}w\eta\psi_t\,\dd x\dd t.
\end{align*}
The dynamic equation in \eqref{eq:forced-gravity-system} gives
\[
  -\psi_t=g\eta+N(\eta)\psi+F.
\]
Substitution therefore yields
\begin{align*}
  \int_0^T\int_{\T}w\psi G(\eta)\psi\,\dd x\dd t
  &=\left[\int_{\T}w\eta\psi\right]_0^T
    +g\int_0^T\int_{\T}w\eta^2\,\dd x\dd t\\
  &\quad+\int_0^T\int_{\T}w\eta N(\eta)\psi\,\dd x\dd t
    +\int_0^T\int_{\T}w\eta F\,\dd x\dd t.
\end{align*}
Solving this equality for the gravitational term and dividing by two gives
\eqref{eq:weighted-equipartition}.
\end{proof}

\begin{lemma}[Intermediate gravity identity]\label{lem:intermediate-gravity}
Set
\begin{equation}\label{eq:zeta0-rho0}
  \zeta_0=\partial_x(m\eta)+\frac12\chi\eta-\frac14\eta,
  \qquad
  \varrho_0=\zeta_0+\eta-x\eta_x.
\end{equation}
Then
\begin{align}
  \frac12\int_0^T\mathcal H_g(t)\,\dd t
  &=-\int_0^T\int_{\T}F\zeta_0\,\dd x\dd t
  +\frac12\int_0^T\int_{\T}\chi\psi G(\eta)\psi\,\dd x\dd t
  \notag\\
  &\quad-\left[\int_{\T}\zeta_0\psi\,\dd x\right]_0^T
  -\int_0^T\int_{\T}m\psi_xG(\eta)\psi\,\dd x\dd t
  \notag\\
  &\quad-\int_0^T\int_{\T}\zeta_0N(\eta)\psi\,\dd x\dd t.
  \label{eq:intermediate-gravity}
\end{align}
\end{lemma}

\begin{proof}
From the second equation in \eqref{eq:forced-gravity-system},
\[
  \Theta
  =\eta\bigl(F+N(\eta)\psi\bigr)+\frac g2\eta^2.
\]
Insert this expression in \eqref{eq:basic-multiplier}. 
The terms containing \(N(\eta)\psi\) combine into
\(\partial_x(m\eta)N(\eta)\psi\) because
\[
m_x\eta+m\eta_x=\partial_x(m\eta).
\]
Similarly, the terms involving \(F\) can be written as
\(F\partial_x(m\eta)\). Therefore, we obtain
\begin{align}
  \frac g2\int_0^T\int_{\T}m_x\eta^2\,\dd x\dd t
  &=-\left[\int_{\T}\partial_x(m\eta)\psi\,\dd x\right]_{0}^{T}
  -\int_0^T\int_{\T}F\partial_x(m\eta)\,\dd x\dd t
  \notag\\
  &\quad
  -\int_0^T\int_{\T}m\psi_xG(\eta)\psi\,\dd x\dd t
  -\int_0^T\int_{\T}\partial_x(m\eta)N(\eta)\psi\,\dd x\dd t.
  \label{eq:mx-potential}
\end{align}
Since \(m_x=1-\chi\), this identity contains the full gravitational
energy together with an additional localized contribution. We apply
Lemma~\ref{lem:weighted-equipartition} twice, first with \(w=1\) and then
with \(w=\chi\). 
First set
\[
  A_K=\frac12\int_0^T\int_{\T}\psi G(\eta)\psi\,\dd x\dd t,
  \qquad
  A_P=\frac g2\int_0^T\int_{\T}\eta^2\,\dd x\dd t.
\]
The choice \(w=1\) gives
\begin{align}
  A_P
  &=\frac12\int_0^T\mathcal H_g(t)\,\dd t
  -\frac14\int_0^T\int_{\T}\eta F\,\dd x\dd t
  -\frac14\left[\int_{\T}\eta\psi\,\dd x\right]_0^T
  \notag\\
  &\quad
  -\frac14\int_0^T\int_{\T}\eta N(\eta)\psi\,\dd x\dd t.
  \label{eq:AP-expression}
\end{align}
Indeed, the weighted equipartition formula first expresses \(A_P\) in
terms of \(A_K\), and then \(A_K+A_P=\int_0^T\mathcal H_g(t)\,\dd t\)
is used to eliminate \(A_K\).

The choice \(w=\chi\) gives
\begin{align}
  \frac g2\int_0^T\int_{\T}\chi\eta^2\,\dd x\dd t
  &=\frac12\int_0^T\int_{\T}
    \chi\psi G(\eta)\psi\,\dd x\dd t
  -\frac12\int_0^T\int_{\T}\chi\eta F\,\dd x\dd t
  \notag\\
  &\quad
  -\frac12\left[\int_{\T}\chi\eta\psi\,\dd x\right]_0^T
  -\frac12\int_0^T\int_{\T}
    \chi\eta N(\eta)\psi\,\dd x\dd t.
  \label{eq:weighted-chi-expression}
\end{align}
On the other hand, \(m_x=1-\chi\) turns \eqref{eq:mx-potential} into
\begin{align*}
  A_P
  &=-\left[\int_{\T}\partial_x(m\eta)\psi\,\dd x\right]_0^T
  -\int_0^T\int_{\T}F\partial_x(m\eta)\,\dd x\dd t\\
  &\quad
  -\int_0^T\int_{\T}m\psi_xG(\eta)\psi\,\dd x\dd t
  -\int_0^T\int_{\T}\partial_x(m\eta)N(\eta)\psi\,\dd x\dd t
  +\frac g2\int_0^T\int_{\T}\chi\eta^2\,\dd x\dd t.
\end{align*}
Insert \eqref{eq:AP-expression} on the left and
\eqref{eq:weighted-chi-expression} in the last term on the right.  Moving
the three contributions with coefficient \(-\frac14\eta\) to the right
gives
\begin{align*}
  \frac12\int_0^T\mathcal H_g(t)\,\dd t
  &=\frac12\int_0^T\int_{\T}\chi\psi G(\eta)\psi\,\dd x\dd t
  -\int_0^T\int_{\T}m\psi_xG(\eta)\psi\,\dd x\dd t\\
  &\quad
  -\int_0^T\int_{\T}
  F\left[\partial_x(m\eta)+\frac12\chi\eta-\frac14\eta\right]
  \dd x\dd t\\
  &\quad
  -\left[\int_{\T}
  \left[\partial_x(m\eta)+\frac12\chi\eta-\frac14\eta\right]
  \psi\,\dd x\right]_0^T\\
  &\quad
  -\int_0^T\int_{\T}
  \left[\partial_x(m\eta)+\frac12\chi\eta-\frac14\eta\right]
  N(\eta)\psi\,\dd x\dd t.
\end{align*}
The bracket is \(\zeta_0\), proving \eqref{eq:intermediate-gravity}.
\end{proof}

\subsection{Pohozaev and weighted stress identities}

The multiplier argument requires a Pohozaev identity for the term
\(x\psi_xG(\eta)\psi\) and a weighted stress identity for terms involving
\(N(\eta)\psi\).  In the infinite-depth setting, both identities are first
derived on truncated domains, and the contributions from the artificial
lower boundary must be shown to vanish as the truncation depth tends to
infinity.  They are established in Lemmas~\ref{lem:pohozaev} and
\ref{lem:harmonic-stress}, respectively.

\begin{lemma}[Pohozaev identity]\label{lem:pohozaev}
For each fixed \(t\),
\begin{equation}\label{eq:pohozaev}
  \int_{\T}x\psi_xG(\eta)\psi\,\dd x
  =\Sigma_\infty(t)
  +\int_{\T}(\eta-x\eta_x)N(\eta)\psi\,\dd x.
\end{equation}
\end{lemma}

\begin{proof}
	Fix \(t\) and suppress the time variable from the notation.  Choose
	\(R>0\) sufficiently large so that
	\[
	-R<\min_{x\in\T}\eta(x),
	\]
	and consider the truncated half-cell
	\[
	\Omega_{+,R}
	=
	\left\{
	(x,y):
	0<x<L,\quad -R<y<\eta(x)
	\right\}.
	\]
	Let
	\[
	S
	=
	\nabla\phi\otimes\nabla\phi
	-\frac12|\nabla\phi|^2I
	\]
	be the harmonic stress tensor, and set
	\[
	X=(x,y),
	\qquad
	J=SX.
	\]
	
	We first verify that \(J\) is divergence free.  Writing
	\(\phi_i=\partial_i\phi\), we have
	\[
	S_{ij}
	=
	\phi_i\phi_j
	-\frac12|\nabla\phi|^2\delta_{ij}.
	\]
	Since \(\Delta\phi=0\), we have
	\begin{align*}
		\sum_{j=1}^2\partial_jS_{ij}
		&=
		\sum_{j=1}^2
		\partial_j(\phi_i\phi_j)
		-\frac12\partial_i|\nabla\phi|^2
		\\
		&=
		\sum_{j=1}^2
		\phi_{ij}\phi_j
		+\phi_i\Delta\phi
		-\frac12\partial_i|\nabla\phi|^2
		\\
		&=0.
	\end{align*}
	Thus \(\operatorname{div}S=0\).  Moreover, since the space dimension is
	two,
	\[
	\operatorname{tr}S
	=
	|\nabla\phi|^2
	-\frac22|\nabla\phi|^2
	=0.
	\]
	Because \(X_i\) satisfies \(	\partial_jX_i=\delta_{ij},\)
	we obtain
	\begin{align*}
		\operatorname{div}J
		&=
		\sum_{j=1}^2
		\partial_j
		\left(
		\sum_{i=1}^2S_{ji}X_i
		\right)
		\\
		&=
		\sum_{i=1}^2
		\left(
		\sum_{j=1}^2\partial_jS_{ji}
		\right)X_i
		+
		\sum_{i,j=1}^2
		S_{ji}\partial_jX_i
		\\
		&=
		(\operatorname{div}S)\cdot X
		+
		\sum_{i=1}^2S_{ii}
		\\&=0.
	\end{align*}
	
	We now compute the flux of \(J\) across the four parts of
	\(\partial\Omega_{+,R}\).
	On the free surface
	\[
	\Gamma_+
	=
	\{(x,\eta(x)):0<x<L\},
	\]
	the outward unit normal is
	\[
	\nu
	=
	\frac{(-\eta_x,1)}
	{\sqrt{1+\eta_x^2}},
	\]
	and
	\[
	\dd S
	=
	\sqrt{1+\eta_x^2}\,\dd x.
	\]
	Hence
	\[
	\nu\,\dd S
	=
	(-\eta_x,1)\,\dd x.
	\]
	Writing
	\[
	n_*=(-\eta_x,1),
	\]
	we have on the free surface
	\[
	\nabla\phi\cdot n_*
	=
	\phi_y-\eta_x\phi_x
	=
	G(\eta)\psi
	\]
	and
	\[
	X\cdot n_*
	=
	\eta-x\eta_x.
	\]
	If
	\[
	V=\phi_x|_{y=\eta},
	\qquad
	B=\phi_y|_{y=\eta},
	\]
	then
	\[
	\nabla\phi\cdot X
	=
	xV+\eta B.
	\]
	Therefore
	\begin{align*}
		J\cdot n_*
		&=
		(SX)\cdot n_*
		\\
		&=
		(\nabla\phi\cdot X)
		(\nabla\phi\cdot n_*)
		-
		\frac12|\nabla\phi|^2(X\cdot n_*)
		\\
		&=
		(xV+\eta B)(B-\eta_xV)
		-
		\frac12(V^2+B^2)(\eta-x\eta_x).
	\end{align*}
	A direct rearrangement gives
	\begin{align*}
		J\cdot n_*
		&=
		x(V+\eta_xB)(B-\eta_xV)
		-(\eta-x\eta_x)
		\left(
		\frac12V^2-\frac12B^2+\eta_xVB
		\right).
	\end{align*}
	Using
	\[
	\psi_x=V+\eta_xB,
	\qquad
	G(\eta)\psi=B-\eta_xV,
	\]
	and \eqref{eq:N-VB}, we obtain
	\begin{equation*}
		J\cdot n_*
		=
		x\psi_xG(\eta)\psi
		-
		(\eta-x\eta_x)N(\eta)\psi.
	\end{equation*}
	Thus the free-surface contribution is
	\[
	\int_0^L
	\left[
	x\psi_xG(\eta)\psi
	-
	(\eta-x\eta_x)N(\eta)\psi
	\right]\dd x.
	\]
	
	We next consider the vertical sides.  On \(x=0\), the outward normal is
	\((-1,0)\).  The evenness of \(\phi\) in \(x\) implies
	\[
	\phi_x(0,y)=0.
	\]
	Since \(X=(0,y)\) there,
	\[
	J\cdot(-1,0)
	=
	-\bigl(S_{11}X_1+S_{12}X_2\bigr)
	=
	-y\phi_x\phi_y
	=0.
	\]
	Hence the side \(x=0\) contributes nothing.
	
	On \(x=L\), the outward normal is \((1,0)\).  The periodicity and evenness
	of \(\phi\) imply
	\[
	\phi_x(L,y)=0.
	\]
	Indeed, \(\phi_x\) is odd, while periodicity gives
	\(\phi_x(-L,y)=\phi_x(L,y)\), so both values must vanish.  Consequently,
	\[
	S_{11}(L,y)
	=
	-\frac12\phi_y(L,y)^2,
	\qquad
	S_{12}(L,y)=0.
	\]
	Since \(X=(L,y)\),
	\[
	J\cdot(1,0)
	=
	-\frac L2\phi_y(L,y)^2.
	\]
	The contribution of this side is therefore
	\begin{equation}\label{eq:pohozaev-right-flux}
		-\frac L2
		\int_{-R}^{\eta(L)}
		\phi_y(L,y)^2\,\dd y.
	\end{equation}
	
	We also verify that the limit of
	\eqref{eq:pohozaev-right-flux} is well defined as \(R\to\infty\).
	Choose \(y_0<\min_{\T}\eta\).  By
	\eqref{eq:deep-decay-derivatives}, applied to
	\(\phi_y\) and \(\phi_{xy}\),
	\[
	\|\phi_y(\cdot,y)\|_{H^1(\T)}
	\leq
	C\mathrm e^{\pi(y-y_0)/L},
	\qquad y\leq y_0.
	\]
	The one-dimensional Sobolev embedding on \(\T\) then gives
	\[
	|\phi_y(L,y)|
	\leq
	C\mathrm e^{\pi(y-y_0)/L},
	\qquad y\leq y_0.
	\]
	It follows that
	\[
	\int_{-\infty}^{y_0}
	\phi_y(L,y)^2\,\dd y<\infty.
	\]
	Smoothness of \(\phi\) on the remaining finite vertical interval shows
	that
	\[
	\int_{-\infty}^{\eta(L)}
	\phi_y(L,y)^2\,\dd y<\infty.
	\]
	Therefore
	\[
	\lim_{R\to\infty}
	\int_{-R}^{\eta(L)}
	\phi_y(L,y)^2\,\dd y
	=
	\int_{-\infty}^{\eta(L)}
	\phi_y(L,y)^2\,\dd y.
	\]
	
	It remains to consider the artificial lower boundary
	\[
	\Gamma_{-,R}
	=
	\{(x,-R):0<x<L\},
	\]
	whose outward normal is \((0,-1)\).  On this boundary,
	\[
	|X|
	=
	\sqrt{x^2+R^2}
	\leq L+R.
	\]
	Since
	\[
	|S|
	\leq C|\nabla\phi|^2,
	\]
	we obtain
	\begin{align*}
		\left|
		\int_0^L
		J(x,-R)\cdot(0,-1)\,\dd x
		\right|
		&\leq
		C(L+R)
		\int_0^L
		|\nabla\phi(x,-R)|^2\,\dd x
		\\
		&\leq
		C(L+R)
		\|\nabla\phi(\cdot,-R)\|_{L^2(\T)}^2.
	\end{align*}
	For \(R\) sufficiently large, \(-R\leq y_0\), and
	Lemma~\ref{lem:deep-decay} gives
	\[
	\|\nabla\phi(\cdot,-R)\|_{L^2(\T)}
	\leq
	C\mathrm e^{-\pi R/L}.
	\]
	After absorbing the fixed factor depending on \(y_0\) into \(C\), we find
	\[
	\left|
	\int_0^L
	J(x,-R)\cdot(0,-1)\,\dd x
	\right|
	\leq
	C(1+R)\mathrm e^{-2\pi R/L}
	\longrightarrow0\quad \text{as} \quad R\to\infty.
	\]

	Applying the divergence theorem to \(J\) on \(\Omega_{+,R}\), using
	\(\operatorname{div}J=0\), and inserting the boundary contributions computed on the free surface, the
	vertical sides \(x=0\) and \(x=L\), and the lower boundary \(y=-R\) yields
	\begin{align*}
		\int_0^L
		\left[
		x\psi_xG(\eta)\psi
		-
		(\eta-x\eta_x)N(\eta)\psi
		\right]\dd x
		-\frac L2
		\int_{-R}^{\eta(L)}
		\phi_y(L,y)^2\dd y
		+
		\int_0^L
		J(x,-R)\cdot(0,-1)\dd x
		=0.
	\end{align*}
	Passing to the limit \(R\to\infty\) gives
	\begin{equation}\label{eq:pohozaev-half-cell}
		\int_0^L
		\left[
		x\psi_xG(\eta)\psi
		-
		(\eta-x\eta_x)N(\eta)\psi
		\right]\dd x
		=
		\frac L2
		\int_{-\infty}^{\eta(L)}
		\phi_y(L,y)^2\,\dd y.
	\end{equation}
	
	Finally, we pass from the half-cell to the full periodic cell.  The
	functions \(\eta\) and \(\psi\) are even, while \(\eta_x\) and
	\(\psi_x\) are odd.  Since \(\phi\) is even in \(x\),
	\eqref{eq:def-DNO} shows that \(G(\eta)\psi\) is even.
	Moreover, by \eqref{eq:N-definition},
	\[
	N(\eta)\psi
	=
	\frac12\psi_x^2
	-
	\frac12
	\frac{
		\bigl(G(\eta)\psi+\eta_x\psi_x\bigr)^2
	}{
		1+\eta_x^2
	},
	\]
	so \(N(\eta)\psi\) is also even.  Hence
	\[
	x\psi_xG(\eta)\psi
	\qquad\text{and}\qquad
	(\eta-x\eta_x)N(\eta)\psi
	\]
	are both even functions of \(x\).  Doubling
	\eqref{eq:pohozaev-half-cell} therefore gives
	\[
	\int_{\T}
	\left[
	x\psi_xG(\eta)\psi
	-
	(\eta-x\eta_x)N(\eta)\psi
	\right]\dd x
	=
	L
	\int_{-\infty}^{\eta(L)}
	\phi_y(L,y)^2\,\dd y.
	\]
	By the definition \eqref{eq:Sigma-infty} of \(\Sigma_\infty(t)\), this is
	equivalent to
	\[
	\int_{\T}
	x\psi_xG(\eta)\psi\,\dd x
	=
	\Sigma_\infty(t)
	+
	\int_{\T}
	(\eta-x\eta_x)N(\eta)\psi\,\dd x,
	\]
	which is \eqref{eq:pohozaev}.
\end{proof}

\begin{lemma}[Weighted stress identity]\label{lem:harmonic-stress}
Let \((\eta,\psi,P_{\mathrm{ext}})\) be a solution of
\eqref{eq:hydroelastic-system} on \([0,T]\), and let
\(f\in C^1([-L,L])\).  Then, for every fixed \(t\in[0,T]\),
\begin{equation}\label{eq:harmonic-stress}
  \int_{-L}^{L}fN(\eta)\psi\,\dd x
  =-\int_{-L}^{L}\int_{-\infty}^{\eta(x)}
  f_x\phi_x\phi_y\,\dd y\dd x.
\end{equation}
In particular, the identity applies to the generally nonperiodic weights
\(\varrho_0\) and \(\varrho\).
\end{lemma}

\begin{proof}
Set
\[
  X_f=
  \left(
    f\phi_x\phi_y,
    \frac f2(\phi_y^2-\phi_x^2)
  \right).
\]
A direct differentiation, using \(\phi_{xx}+\phi_{yy}=0\), gives
\[
  \operatorname{div}X_f=f_x\phi_x\phi_y.
\]
Apply the divergence theorem on the truncated cell
\(\Omega_R=\Omega(t)\cap\{y>-R\}\).  On the free surface, using again the
non-unit normal \((-\eta_x,1)\), the flux density is
\begin{align*}
  X_f\cdot(-\eta_x,1)
  &=f\left[-\eta_x\phi_x\phi_y
       +\frac12(\phi_y^2-\phi_x^2)\right]_{y=\eta}\\
  &=-f\left[
       \frac12V^2-\frac12B^2+\eta_xVB
     \right]\\
  & =-fN(\eta)\psi.
\end{align*}
Here \(V=\phi_x|_{y=\eta}\) and \(B=\phi_y|_{y=\eta}\).
On the vertical sides the flux is \(\pm f\phi_x\phi_y\), which vanishes
because even periodicity implies \(\phi_x(\pm L,y)=0\).  On \(y=-R\),
Lemma~\ref{lem:deep-decay} gives
\[
  \left|
  \frac12\int_{-L}^{L}
  f(x)\bigl(\phi_y^2-\phi_x^2\bigr)(x,-R)\,\dd x
  \right|
  \leq
  C\|f\|_{L^\infty([-L,L])}
  \|\nabla\phi(\cdot,-R)\|_{L^2(\T)}^2
  \longrightarrow0.
\]
Passing to the limit in the divergence theorem gives
\[
  \int_{-L}^{L}\int_{-\infty}^{\eta(x)}
  f_x\phi_x\phi_y\,\dd y\dd x
  =-\int_{-L}^{L}fN(\eta)\psi\,\dd x,
\]
which is \eqref{eq:harmonic-stress}.  No periodicity of \(f\) was used.  Only
its \(C^1\)-regularity on the closed representative cell is needed.
\end{proof}

\subsection{Localized multiplier identities}

We begin with a localized multiplier identity for the forced gravity system
\eqref{eq:forced-gravity-system} in the infinite-depth setting.  We then adapt
the corresponding gravity-wave multiplier relation to the hydroelastic system
by incorporating the nonlinear
flexural term and the multiplier \(\zeta\).  The resulting identity
\eqref{eq:hydro-multiplier} is the starting point of the stabilization
estimate.

\begin{proposition}[Localized multiplier identity]
\label{prop:alazard-deep}
Let \((\eta,\psi)\) be a sufficiently smooth even solution of
\eqref{eq:forced-gravity-system} on \([0,T]\), and let \(\phi\) be its
associated finite-energy harmonic extension in \(\Omega(t)\).  Then
\begin{align}
  \frac12\int_0^T\mathcal H_g(t)\,\dd t
  +\int_0^T\Sigma_\infty(t)\,\dd t
  &=\int_0^T\int_{\T}
  \left[\frac12\chi\psi+(x-m)\psi_x\right]
  G(\eta)\psi\,\dd x\dd t
  \notag\\
  &\quad-\int_0^T\int_{\T}F\zeta_0\,\dd x\dd t
  -\left[\int_{\T}\zeta_0\psi\,\dd x\right]_0^T
  \notag\\
  &\quad+\int_0^T\iint_{\Omega(t)}
  (\varrho_0)_x\phi_x\phi_y\,\dd y\dd x\dd t.
  \label{eq:alazard-deep}
\end{align}
\end{proposition}

\begin{proof}
Start from \eqref{eq:intermediate-gravity}.  Split
\[
  -m\psi_xG(\eta)\psi
  =-x\psi_xG(\eta)\psi
   +(x-m)\psi_xG(\eta)\psi.
\]
By Lemma~\ref{lem:pohozaev},
\begin{align*}
  -\int_{\T}x\psi_xG(\eta)\psi\,\dd x
  &=-\Sigma_\infty(t)
    -\int_{\T}(\eta-x\eta_x)N(\eta)\psi\,\dd x.
\end{align*}
Substituting this into \eqref{eq:intermediate-gravity} and moving the
\(\Sigma_\infty\)-term to the left gives
\begin{align*}
  \frac12\int_0^T\mathcal H_g(t)\,\dd t
  +\int_0^T\Sigma_\infty(t)\,\dd t
  &=
  \frac12\int_0^T\int_{\T}
  \chi\psi G(\eta)\psi\,\dd x\dd t
  +\int_0^T\int_{\T}
  (x-m)\psi_xG(\eta)\psi\,\dd x\dd t
  \\
  &\quad
  -\int_0^T\int_{\T}
  F\zeta_0\,\dd x\dd t
  -\left[\int_{\T}\zeta_0\psi\,\dd x\right]_0^T
  \\
  &\quad
  -\int_0^T\int_{\T}
  \bigl[\zeta_0+\eta-x\eta_x\bigr]
  N(\eta)\psi\,\dd x\dd t.
\end{align*}
By definition,
\[
  \varrho_0=\zeta_0+\eta-x\eta_x,
\]
so the last term is
\(-\int_0^T\int_{\T}\varrho_0N(\eta)\psi\,\dd x\dd t\).  Applying the weighted stress
identity \eqref{eq:harmonic-stress} at each time yields
\[
  -\int_{\T}\varrho_0N(\eta)\psi\,\dd x
  =\iint_{\Omega(t)}(\varrho_0)_x\phi_x\phi_y\,\dd y\dd x.
\]
This proves \eqref{eq:alazard-deep}.  The passage to infinite depth has
already been justified in Lemmas~\ref{lem:pohozaev} and
\ref{lem:harmonic-stress}: by Lemma~\ref{lem:deep-decay}, their respective
lower-boundary fluxes on \(y=-R\) tend to zero as \(R\to\infty\), with the
Pohozaev flux bounded by
\(C(1+R)\mathrm e^{-2\pi R/L}\) and the weighted stress flux by
\(C\|f\|_{L^\infty}\mathrm e^{-2\pi R/L}\).
\end{proof}
The following identity is the central multiplier relation in the stabilization
argument.  
\begin{proposition}
\label{prop:hydro-multiplier}
Let \((\eta,\psi,P_{\mathrm{ext}})\) be a solution of
\eqref{eq:hydroelastic-system} on \([0,T]\).  Set
\[\mathcal{W}=2g\int_0^T\int_{\T}\chi\eta^2\,\dd x\dd t,\quad \mathcal{O}_{\infty}=\int_0^T\Sigma_\infty(t)\,\dd t,\]
then
\begin{align}
\notag&  \frac12\int_0^T\mathcal H_g(t)\,\dd t
  +\beta\int_0^T\int_{\T}\Bflex(\eta)\zeta\,\dd x\dd t
  +\mathcal{W}+\mathcal{O}_{\infty}\\
 \notag =&\int_0^T\int_{\T}
  \left[\frac52\chi\psit+(x-m)\psi_x\right]
  G(\eta)\psi\,\dd x\dd t
  -\int_0^T\int_{\T}\Pt\zeta\,\dd x\dd t\\
 & -\left[\int_{\T}\zeta\psit\,\dd x\right]_0^T
 +\int_0^T\iint_{\Omega(t)}
  \varrho_x\phi_x\phi_y\,\dd y\dd x\dd t.
  \label{eq:hydro-multiplier}
\end{align}
\end{proposition}

\begin{proof}
	For the hydroelastic waves model \eqref{eq:hydroelastic-system}, the dynamic
	equation can be written in the form
	\[
	\psi_t+g\eta+N(\eta)\psi=-F
	\]
	with
	\[
	F=P_{\mathrm{ext}}+\beta\Bflex(\eta).
	\]
	Thus the pair \((\eta,\psi)\) satisfies
	\eqref{eq:forced-gravity-system} with this choice of \(F\), and
	Proposition~\ref{prop:alazard-deep} gives
	\begin{align}
		\frac12\int_0^T\mathcal H_g(t)\,\dd t
		+\int_0^T\Sigma_\infty(t)\,\dd t
		&=
		\int_0^T\int_{\T}
		\left[
		\frac12\chi\psi+(x-m)\psi_x
		\right]
		G(\eta)\psi\,\dd x\dd t
		\notag\\
		&\quad
		-\int_0^T\int_{\T}
		F\zeta_0\,\dd x\dd t
		-\left[
		\int_{\T}\zeta_0\psi\,\dd x
		\right]_0^T
		\notag\\
		&\quad
		+\int_0^T\iint_{\Omega(t)}
		(\varrho_0)_x\phi_x\phi_y\,\dd y\dd x\dd t.
		\label{eq:hydro-multiplier-start}
	\end{align}
	
	We now pass from the gravity multiplier \(\zeta_0\) to the elastic
	multiplier \(\zeta\).  Since \(\chi=\chi(x)\) is independent of time,
	the classical regularity of the solution allows differentiation under the
	integral sign, and hence
	\begin{align*}
		\frac{\dd}{\dd t}
		\int_{\T}\chi\eta\psi\,\dd x
		&=
		\int_{\T}
		\chi\bigl(\eta_t\psi+\eta\psi_t\bigr)\,\dd x.
	\end{align*}
	Using
	\[
	\eta_t=G(\eta)\psi,
	\qquad
	\psi_t=-g\eta-N(\eta)\psi-F,
	\]
	we obtain
	\begin{align*}
		\frac{\dd}{\dd t}
		\int_{\T}\chi\eta\psi\,\dd x
		&=
		\int_{\T}\chi
		\left[
		\psi G(\eta)\psi
		-g\eta^2
		-\eta N(\eta)\psi
		-\eta F
		\right]\dd x.
	\end{align*}
	Integrating over \(t\in[0,T]\) gives the exact identity
	\begin{equation}\label{eq:exact-zero-identity}
		0=
		\int_0^T\int_{\T}\chi
		\left[
		\psi G(\eta)\psi
		-g\eta^2
		-\eta N(\eta)\psi
		-\eta F
		\right]\dd x\dd t
		-
		\left[
		\int_{\T}\chi\eta\psi\,\dd x
		\right]_0^T.
	\end{equation}
	
	Add twice \eqref{eq:exact-zero-identity} to
	\eqref{eq:hydro-multiplier-start}.  By
	\eqref{eq:zeta0-rho0} and \eqref{eq:zeta-rho},
	\[
	\zeta_0+2\chi\eta
	=
	\partial_x(m\eta)
	+\frac52\chi\eta
	-\frac14\eta
	=
	\zeta.
	\]
	Therefore the forcing terms satisfy
	\begin{align*}
		-\int_0^T\int_{\T}
		F\zeta_0\,\dd x\dd t
		-2\int_0^T\int_{\T}
		\chi\eta F\,\dd x\dd t
		=
		-\int_0^T\int_{\T}
		F\zeta\,\dd x\dd t,
	\end{align*}
	while the endpoint terms satisfy
	\begin{align*}
		-\left[
		\int_{\T}\zeta_0\psi\,\dd x
		\right]_0^T
		-2\left[
		\int_{\T}\chi\eta\psi\,\dd x
		\right]_0^T
		=
		-\left[
		\int_{\T}\zeta\psi\,\dd x
		\right]_0^T.
	\end{align*}
	Similarly, the coefficient of the first observation term becomes
	\[
	\frac12\chi\psi+2\chi\psi
	=
	\frac52\chi\psi.
	\]
	The gravitational contribution added to the right-hand side is
	\[
	-2g\int_0^T\int_{\T}\chi\eta^2\,\dd x\dd t,
	\]
	which we move to the left-hand side.
	
	It remains at this stage to combine the nonlinear term
	\[
	-2\int_0^T\int_{\T}
	\chi\eta N(\eta)\psi\,\dd x\dd t
	\]
	with the bulk term in \eqref{eq:hydro-multiplier-start}.  For each fixed
	\(t\in[0,T]\), Lemma~\ref{lem:harmonic-stress}, applied with
	\[
	f=2\chi\eta,
	\]
	gives
	\[
	-2\int_{\T}
	\chi\eta N(\eta)\psi\,\dd x
	=
	\iint_{\Omega(t)}
	(2\chi\eta)_x\phi_x\phi_y\,\dd y\dd x.
	\]
	Consequently,
	\begin{align*}
		\iint_{\Omega(t)}
		(\varrho_0)_x\phi_x\phi_y\,\dd y\dd x
		-2\int_{\T}
		\chi\eta N(\eta)\psi\,\dd x=
		\iint_{\Omega(t)}
		\left[
		(\varrho_0)_x+(2\chi\eta)_x
		\right]
		\phi_x\phi_y\,\dd y\dd x.
	\end{align*}
	Since
	\[
	\varrho_0
	=
	\zeta_0+\eta-x\eta_x
	\]
	and
	\[
	\varrho
	=
	\zeta+\eta-x\eta_x
	=
	\varrho_0+2\chi\eta,
	\]
	we have
	\[
	(\varrho_0)_x+(2\chi\eta)_x
	=
	\varrho_x.
	\]
	Thus, after integration over \(t\in[0,T]\),
	\begin{align*}
		\int_0^T\iint_{\Omega(t)}
		(\varrho_0)_x\phi_x\phi_y\,\dd y\dd x\dd t
		-
		2\int_0^T\int_{\T}
		\chi\eta N(\eta)\psi\,\dd x\dd t=
		\int_0^T\iint_{\Omega(t)}
		\varrho_x\phi_x\phi_y\,\dd y\dd x\dd t.
	\end{align*}
	
	We now use
	\[
	F=P_{\mathrm{ext}}+\beta\Bflex(\eta)
	\]
	in the forcing term.  Hence
	\begin{align*}
		-\int_0^T\int_{\T}F\zeta\,\dd x\dd t
		&=
		-\int_0^T\int_{\T}
		P_{\mathrm{ext}}\zeta\,\dd x\dd t
		\\
		&\quad
		-\beta\int_0^T\int_{\T}
		\Bflex(\eta)\zeta\,\dd x\dd t.
	\end{align*}
	Moving the second term to the left-hand side, we arrive at the
	uncentered identity
	\begin{align}
	\notag	&\frac12\int_0^T\mathcal H_g(t)\,\dd t
		+\beta\int_0^T\int_{\T}
		\Bflex(\eta)\zeta\,\dd x\dd t+\mathcal{W}
		+\mathcal{O}_{\infty}\\
		\notag		=&
		\int_0^T\int_{\T}
		\left[
		\frac52\chi\psi+(x-m)\psi_x
		\right]
		G(\eta)\psi\,\dd x\dd t
		\quad-\int_0^T\int_{\T}
		P_{\mathrm{ext}}\zeta\,\dd x\dd t\\
		&-\left[
		\int_{\T}\zeta\psi\,\dd x
		\right]_0^T
		\quad+\int_0^T\iint_{\Omega(t)}
		\varrho_x\phi_x\phi_y\,\dd y\dd x\dd t,
		\label{eq:hydro-multiplier-uncentered}
	\end{align}
	where \[\mathcal{W}=2g\int_0^T\int_{\T}\chi\eta^2\,\dd x\dd t,\quad \mathcal{O}_{\infty}=\int_0^T\Sigma_\infty(t)\,\dd t.\]
	It remains to rewrite the three terms involving the spatial means of
	\(\psi\) and \(P_{\mathrm{ext}}\).  Recall that
	\[
	\psit=\psi-\avg{\psi},
	\qquad
	\Pt=P_{\mathrm{ext}}-\avg{P_{\mathrm{ext}}}.
	\]
	By Lemma~\ref{lem:kinetic-boundary},
	\[
	G(\eta)1=0,
	\]
	and therefore
	\[
	G(\eta)\psit=G(\eta)\psi.
	\]
	Moreover,
	\[
	\psit_x=\psi_x.
	\]
	Thus the term
	\[
	\int_{\T}(x-m)\psi_xG(\eta)\psi\,\dd x
	\]
	is unchanged by centering \(\psi\).  The only observation term affected is
	the one containing \(\chi\psi\).
	
	By Lemma~\ref{lem:zero-mode} and the normalization
	\eqref{eq:zero-mean-initial},
	\[
	\int_{\T}\eta(t,x)\,\dd x=0
	\qquad
	(0\leq t\leq T).
	\]
	Define
	\[
	\nu(t)
	:=
	\int_{\T}\chi(x)\eta(t,x)\,\dd x.
	\]
	Since
	\[
	\zeta
	=
	\partial_x(m\eta)
	+\frac52\chi\eta
	-\frac14\eta,
	\]
	periodicity of \(m\eta\) and the zero-mean property of \(\eta\) yield
	\begin{align}
		\int_{\T}\zeta(t,x)\,\dd x
		=
		\frac52\int_{\T}\chi(x)\eta(t,x)\,\dd x=
		\frac52\nu(t).
		\label{eq:zeta-mean-nu}
	\end{align}
	The classical regularity of the solution and the fact that \(\chi\) is
	independent of time justify differentiation under the integral sign.  Using
	the first equation in \eqref{eq:hydroelastic-system},
	\begin{equation*}
		\nu'(t)
		=
		\int_{\T}\chi(x)\eta_t(t,x)\,\dd x
		=
		\int_{\T}\chi(x)G(\eta)\psi(t,x)\,\dd x.
	\end{equation*}
	
	We now compare the right-hand side of
	\eqref{eq:hydro-multiplier-uncentered} with the centered expression in
	\eqref{eq:hydro-multiplier}.  For the first observation term,
	\begin{align*}
		\frac52\int_0^T\int_{\T}
		\chi\psi G(\eta)\psi\,\dd x\dd t
		-
		\frac52\int_0^T\int_{\T}
		\chi\psit G(\eta)\psi\,\dd x\dd t=
		\frac52\int_0^T
		\avg{\psi}(t)\nu'(t)\,\dd t.
	\end{align*}
	For the endpoint term, \eqref{eq:zeta-mean-nu} gives
	\begin{align*}
		-\left[
		\int_{\T}\zeta\psi\,\dd x
		\right]_0^T
		+
		\left[
		\int_{\T}\zeta\psit\,\dd x
		\right]_0^T=
		-\frac52
		\left[
		\avg{\psi}(t)\nu(t)
		\right]_0^T.
	\end{align*}
	Finally, for the pressure term,
	\begin{align*}
		-\int_0^T\int_{\T}
		P_{\mathrm{ext}}\zeta\,\dd x\dd t
		+
		\int_0^T\int_{\T}
		\Pt\zeta\,\dd x\dd t=
		-\frac52
		\int_0^T
		\avg{P_{\mathrm{ext}}}(t)\nu(t)\,\dd t.
	\end{align*}
	Therefore the difference between the uncentered and centered
	right-hand sides is
	\begin{align*}
		&\frac52\int_0^T
		\avg{\psi}(t)\nu'(t)\,\dd t
		-\frac52
		\left[
		\avg{\psi}(t)\nu(t)
		\right]_0^T
		-\frac52\int_0^T
		\avg{P_{\mathrm{ext}}}(t)\nu(t)\,\dd t.
	\end{align*}
	By integration by parts in time,
	\begin{align*}
		\int_0^T\avg{\psi}\,\nu'\,\dd t
		-
		\left[\avg{\psi}\,\nu\right]_0^T
		=
		-\int_0^T
		\frac{\dd}{\dd t}\avg{\psi}\,\nu\,\dd t.
	\end{align*}
	Substituting the integration-by-parts identity into the displayed expression for
	the difference between the uncentered and centered right-hand sides gives
	\[
	-\frac52
	\int_0^T
	\nu(t)
	\left(
	\frac{\dd}{\dd t}\avg{\psi}
	+
	\avg{P_{\mathrm{ext}}}
	\right)\dd t.
	\]
	By the mean evolution identity
	\eqref{eq:mean-psi-evolution},
	\[
	\frac{\dd}{\dd t}\avg{\psi}
	+
	\avg{P_{\mathrm{ext}}}
	=0.
	\]
	Thus the centered and uncentered right-hand sides coincide.  Replacing the
	three corresponding terms in
	\eqref{eq:hydro-multiplier-uncentered} by their centered forms yields
	\eqref{eq:hydro-multiplier}.
\end{proof}

\section{Quantitative estimates for the multiplier identity}\label{sec:estimates}

This section provides the estimates needed to control the right-hand side of
the hydroelastic multiplier identity \eqref{eq:hydro-multiplier}.  We first
bound the surface multiplier \(\zeta\) by the Hamiltonian, then establish the
localized tangential trace estimate \eqref{eq:trace-estimate} and the
mean-zero trace estimate \eqref{eq:mean-zero-trace}, and finally estimate the
remainder terms appearing in \eqref{eq:hydro-multiplier}.  All constants
established in this section are independent of \(T\).

\subsection{Energy control of the surface multiplier}

\begin{lemma}[Energy control of the flexural multiplier]\label{lem:zeta-control}
Assume \(\|\eta_x\|_{L^\infty}\leq\varepsilon_1\).  Put
\[
  A_m=\left\|\frac94-\frac32m_x\right\|_{L^\infty},
\]
and
\begin{equation*}
  C_\zeta
  =\frac{2\|m\|_{L^\infty}^2(1+\varepsilon_1^2)^{5/2}}{\beta}
  +\frac{2\bigl(\|m\|_{L^\infty}^2+2A_m^2\bigr)}{g}.
\end{equation*}
Then
\begin{equation}\label{eq:zeta-control}
  \|\zeta(t)\|_{L^2(\T)}^2\leq C_\zeta\Etot(t).
\end{equation}
\end{lemma}

\begin{proof}
	From the definition of \(\zeta\),
	\[
	\zeta
	=
	m\eta_x
	+
	\left(\frac94-\frac32m_x\right)\eta.
	\]
	Hence, by \((a+b)^2\leq2a^2+2b^2\),
	\begin{align*}
		\|\zeta\|_{L^2(\T)}^2
		&\leq
		2\|m\eta_x\|_{L^2(\T)}^2
		+
		2\left\|
		\left(\frac94-\frac32m_x\right)\eta
		\right\|_{L^2(\T)}^2
		\\
		&\leq
		2\|m\|_{L^\infty}^2
		\|\eta_x\|_{L^2(\T)}^2
		+
		2A_m^2
		\|\eta\|_{L^2(\T)}^2.
	\end{align*}
	
	Since \(\eta\) is periodic, integration by parts gives
	\[
	\|\eta_x\|_{L^2(\T)}^2
	=
	-\int_{\T}\eta\,\eta_{xx}\,\dd x.
	\]
	Therefore, by Cauchy--Schwarz and Young's inequality,
	\begin{align*}
		\|\eta_x\|_{L^2(\T)}^2
		&\leq
		\|\eta\|_{L^2(\T)}
		\|\eta_{xx}\|_{L^2(\T)}
		\\
		&\leq
		\frac12\|\eta\|_{L^2(\T)}^2
		+
		\frac12\|\eta_{xx}\|_{L^2(\T)}^2.
	\end{align*}
	Substituting this estimate for \(\|\eta_x\|_{L^2(\T)}^2\) into the inequality
	derived from the definition of \(\zeta\) gives
	\begin{align}
		\|\zeta\|_{L^2(\T)}^2
		&\leq
		\|m\|_{L^\infty}^2
		\|\eta_{xx}\|_{L^2(\T)}^2
		+
		\bigl(
		\|m\|_{L^\infty}^2+2A_m^2
		\bigr)
		\|\eta\|_{L^2(\T)}^2.
		\label{eq:zeta-intermediate-control}
	\end{align}
	
	We now estimate the two terms on the right-hand side by the Hamiltonian.
	From \eqref{eq:total-energy} and the nonnegativity of its kinetic and
	flexural parts,
	\[
	\frac g2\|\eta\|_{L^2(\T)}^2
	\leq
	\Etot,
	\]
	and hence
	\begin{equation}\label{eq:eta-L2-energy-control}
		\|\eta\|_{L^2(\T)}^2
		\leq
		\frac2g\Etot.
	\end{equation}
	
	For the second derivative, \eqref{eq:flex-energy} and the assumption
	\(\|\eta_x\|_{L^\infty(\T)}\leq\varepsilon_1\) give
	\[
	(1+\eta_x^2)^{5/2}
	\leq
	(1+\varepsilon_1^2)^{5/2},
	\]
	so that
	\begin{align*}
		\beta\Eflex(\eta)
		&=
		\frac{\beta}{2}
		\int_{\T}
		\frac{\eta_{xx}^2}
		{(1+\eta_x^2)^{5/2}}
		\,\dd x
		\\
		&\geq
		\frac{\beta}
		{2(1+\varepsilon_1^2)^{5/2}}
		\|\eta_{xx}\|_{L^2(\T)}^2.
	\end{align*}
	Since \(\beta\Eflex(\eta)\leq\Etot\), we obtain
	\begin{equation}\label{eq:eta-xx-energy-control}
		\|\eta_{xx}\|_{L^2(\T)}^2
		\leq
		\frac{2(1+\varepsilon_1^2)^{5/2}}{\beta}
		\Etot.
	\end{equation}
	
	Finally, inserting
	\eqref{eq:eta-L2-energy-control} and
	\eqref{eq:eta-xx-energy-control} into
	\eqref{eq:zeta-intermediate-control} gives
	\begin{align*}
		\|\zeta\|_{L^2(\T)}^2
		\leq
		\left[
		\frac{
			2\|m\|_{L^\infty}^2
			(1+\varepsilon_1^2)^{5/2}
		}{\beta}
		+
		\frac{
			2\bigl(
			\|m\|_{L^\infty}^2+2A_m^2
			\bigr)
		}{g}
		\right]
		\Etot
		=C_\zeta\Etot,
	\end{align*}
	which proves \eqref{eq:zeta-control}.
\end{proof}

\subsection{Trace estimates}

\begin{lemma}[Localized tangential trace estimate]\label{lem:trace}
Let \((\eta,\psi,P_{\mathrm{ext}})\) be a solution of
\eqref{eq:hydroelastic-system} on \([0,T]\), and fix \(t\in[0,T]\).
Assume \(\|\eta_x(t)\|_{L^\infty}\leq1\), and set
\[
  D(t)=\int_{\T}\chi\bigl(G(\eta)\psi\bigr)^2\,\dd x.
\]
Then
\begin{align}
  \int_{\T}\chi\psi_x^2\,\dd x
  &\leq8D(t)-8\iint_{\Omega(t)}
  \chi_x\phi_x\phi_y\,\dd y\dd x,
  \label{eq:trace-preliminary}
\end{align}
and hence
\begin{equation}\label{eq:trace-estimate}
  \int_{\T}\chi\psi_x^2\,\dd x
  \leq8D(t)+8\|\chi_x\|_{L^\infty}\Etot(t).
\end{equation}
\end{lemma}

\begin{proof}
With \(V=\phi_x|_{y=\eta}\), \(B=\phi_y|_{y=\eta}\), and
\(G=G(\eta)\psi\), one has
\[
  \psi_x=V+\eta_xB,
  \qquad G=B-\eta_xV,
  \qquad
  (1+\eta_x^2)V^2=G^2+2N(\eta)\psi.
\]
Multiply the last identity by \(\chi\) and use
Lemma~\ref{lem:harmonic-stress}:
\[
  \int_{\T}\chi(1+\eta_x^2)V^2\,\dd x
  =D(t)-2\iint_{\Omega(t)}\chi_x\phi_x\phi_y\,\dd y\dd x.
\]
Moreover, from \(B=G+\eta_xV\) we have
\[
  \psi_x=(1+\eta_x^2)V+\eta_xG.
\]
When \(|\eta_x|\leq1\), the elementary inequality
\((a+b)^2\leq2a^2+2b^2\) gives
\[
  \psi_x^2
  \leq2(1+\eta_x^2)^2V^2+2\eta_x^2G^2
  \leq4(1+\eta_x^2)V^2+4G^2.
\]
This proves \eqref{eq:trace-preliminary}.  Finally,
\[
  8\left|\iint\chi_x\phi_x\phi_y\right|
  \leq4\|\chi_x\|_\infty\iint|\nabla\phi|^2
  \leq8\|\chi_x\|_\infty\Etot,
\]
which gives \eqref{eq:trace-estimate}.
\end{proof}

\begin{lemma}[Mean-zero trace estimate]
	\label{lem:mean-zero-trace}
	Let \(\eta\) be a smooth \(2L\)-periodic graph, and let \(\psi\) have the
	finite-energy harmonic extension \(\phi\) used in the definition of
	\(G(\eta)\).  Assume that
	\[
	\|\eta_x\|_{L^\infty(\T)}\leq M.
	\]
	Then
	\begin{equation}\label{eq:mean-zero-trace}
		\|\psit\|_{L^2(\T)}^2
		\leq
		(2+M^2)\frac{L}{\pi}
		\int_{\T}\psi G(\eta)\psi\,\dd x.
	\end{equation}
\end{lemma}

\begin{proof}
	Flatten the fluid domain by the change of variables
	\[
	y=z+\eta(x),
	\qquad
	u(x,z)=\phi(x,z+\eta(x)),
	\qquad
	(x,z)\in\T\times(-\infty,0).
	\]
	The Jacobian of this transformation is one, and
	\[
	u(x,0)=\phi(x,\eta(x))=\psi(x).
	\]
	Moreover,
	\[
	\phi_x=u_x-\eta_xu_z,
	\qquad
	\phi_y=u_z.
	\]
	Hence the boundary representation of the kinetic energy
	\eqref{eq:kinetic-boundary}, followed by the change of variables
	\(y=z+\eta(x)\),
	gives
	\begin{equation}\label{eq:flattened-energy}
		\int_{\T}\psi G(\eta)\psi\,\dd x
		=
		\int_{\T}\int_{-\infty}^{0}
		\left[
		(u_x-\eta_xu_z)^2+u_z^2
		\right]\dd z\dd x.
	\end{equation}
	
	The quadratic form in the integrand is associated with the symmetric matrix
	\[
	\begin{pmatrix}
		1&-\eta_x\\
		-\eta_x&1+\eta_x^2
	\end{pmatrix}.
	\]
	Its determinant is one and its trace satisfies
	\[
	2+\eta_x^2\leq2+M^2.
	\]
	Since the matrix is positive definite, if
	\(\lambda_{\min}\leq\lambda_{\max}\) denote its eigenvalues, then
	\[
	\lambda_{\min}\lambda_{\max}=1,
	\qquad
	\lambda_{\max}
	\leq
	\lambda_{\min}+\lambda_{\max}
	=
	2+\eta_x^2
	\leq2+M^2.
	\]
	Consequently,
	\[
	\lambda_{\min}
	=
	\frac1{\lambda_{\max}}
	\geq
	\frac1{2+M^2}.
	\]
	Therefore, pointwise on
	\(\T\times(-\infty,0)\),
	\begin{equation}\label{eq:flattened-coercivity}
		(u_x-\eta_xu_z)^2+u_z^2
		\geq
		\frac1{2+M^2}
		\left(u_x^2+u_z^2\right).
	\end{equation}
	
	We next expand \(u\) in Fourier series with respect to the periodic
	variable \(x\):
	\[
	u(x,z)
	=
	\sum_{n\in\mathbb Z}
	u_n(z)\mathrm e^{ik_nx},
	\qquad
	k_n=\frac{\pi n}{L},
	\]
	where
	\[
	u_n(z)
	=
	\frac1{2L}
	\int_{\T}u(x,z)\mathrm e^{-ik_nx}\,\dd x.
	\]
	By the finite Dirichlet energy condition
	\eqref{eq:finite-dirichlet-energy}, together with
	\eqref{eq:flattened-energy} and
	\eqref{eq:flattened-coercivity},
	\[
	\int_{\T}\int_{-\infty}^{0}
	\left(
	|u_x|^2+|u_z|^2
	\right)\dd z\dd x
	<\infty.
	\]
	Applying Parseval's identity in the \(x\)-variable yields
	\[
	2L
	\sum_{n\in\mathbb Z}
	\int_{-\infty}^{0}
	\left(
	k_n^2|u_n(z)|^2+|u_n'(z)|^2
	\right)\dd z
	<\infty.
	\]
	In particular, for every \(n\neq0\), since \(k_n\neq0\),
	\[
	\int_{-\infty}^{0}
	\left(
	|u_n(z)|^2+|u_n'(z)|^2
	\right)\dd z
	<\infty.
	\]
	Thus
	\[
	u_n\in H^1(-\infty,0),
	\qquad n\neq0.
	\]
	
	We also record explicitly the behavior of these modes at infinite depth.
	If \(v\in H^1(-\infty,0)\), then \(v\) has an absolutely continuous
	representative on every compact subinterval.  Since \(v\in L^2(-\infty,0)\),
	there exists a sequence \(z_j\to-\infty\) such that
	\(v(z_j)\to0\).  For fixed \(z<0\), the fundamental theorem of calculus
	gives
	\[
	|v(z)|^2-|v(z_j)|^2
	=
	2\operatorname{Re}
	\int_{z_j}^{z}
	v'(s)\overline{v(s)}\,\dd s.
	\]
	Letting \(j\to\infty\), and using
	\(v'v\in L^1(-\infty,0)\), we obtain
	\[
	|v(z)|^2
	=
	2\operatorname{Re}
	\int_{-\infty}^{z}
	v'(s)\overline{v(s)}\,\dd s.
	\]
	Hence
	\[
	|v(z)|^2
	\leq
	2
	\|v'\|_{L^2(-\infty,z)}
	\|v\|_{L^2(-\infty,z)}
	\longrightarrow0
	\qquad
	\text{as }z\to-\infty.
	\]
	Applying this observation to \(v=u_n\) shows that
	\[
	u_n(z)\longrightarrow0
	\qquad
	\text{as }z\to-\infty,
	\qquad n\neq0.
	\]
	
	We may therefore integrate
	\(\frac{\dd}{\dd z}|u_n(z)|^2
	=2\operatorname{Re}(u_n'\overline{u_n})\)
	from \(-\infty\) to \(0\).  For every \(n\neq0\),
	\begin{align*}
		|u_n(0)|^2
		&=
		2\operatorname{Re}
		\int_{-\infty}^{0}
		u_n'(z)\overline{u_n(z)}\,\dd z
		\\
		&\leq
		2
		\int_{-\infty}^{0}
		|u_n'(z)|\,|u_n(z)|\,\dd z
		\\
		&\leq
		\frac1{|k_n|}
		\int_{-\infty}^{0}
		\left(
		|u_n'(z)|^2+k_n^2|u_n(z)|^2
		\right)\dd z.
	\end{align*}
	Since \(n\neq0\),
	\[
	|k_n|=\frac{\pi|n|}{L}\geq\frac{\pi}{L},
	\]
	and hence
	\begin{equation}\label{eq:mode-trace-bound}
		|u_n(0)|^2
		\leq
		\frac{L}{\pi}
		\int_{-\infty}^{0}
		\left(
		|u_n'(z)|^2+k_n^2|u_n(z)|^2
		\right)\dd z.
	\end{equation}
	
	It remains to identify the nonzero Fourier modes with \(\psit\).
	Since \(u(x,0)=\psi(x)\),
	\[
	u_n(0)
	=
	\frac1{2L}
	\int_{\T}
	\psi(x)\mathrm e^{-ik_nx}\,\dd x.
	\]
	The zero Fourier coefficient is therefore
	\[
	u_0(0)
	=
	\frac1{2L}\int_{\T}\psi(x)\,\dd x
	=
	\avg{\psi}.
	\]
	Consequently,
	\[
	\psit(x)
	=
	\psi(x)-\avg{\psi}
	=
	\sum_{n\neq0}
	u_n(0)\mathrm e^{ik_nx}.
	\]
	By Parseval and \eqref{eq:mode-trace-bound},
	\begin{align*}
		\|\psit\|_{L^2(\T)}^2
		&=
		2L\sum_{n\neq0}|u_n(0)|^2
		\\
		&\leq
		\frac{L}{\pi}\,
		2L
		\sum_{n\neq0}
		\int_{-\infty}^{0}
		\left(
		|u_n'(z)|^2+k_n^2|u_n(z)|^2
		\right)\dd z
		\\
		&\leq
		\frac{L}{\pi}
		\int_{\T}\int_{-\infty}^{0}
		\left(
		|u_z(x,z)|^2+|u_x(x,z)|^2
		\right)\dd z\dd x.
	\end{align*}
	Using \eqref{eq:flattened-coercivity} and then
	\eqref{eq:flattened-energy}, we conclude that
	\begin{align*}
		\|\psit\|_{L^2(\T)}^2
		&\leq
		(2+M^2)\frac{L}{\pi}
		\int_{\T}\int_{-\infty}^{0}
		\left[
		(u_x-\eta_xu_z)^2+u_z^2
		\right]\dd z\dd x
		\\
		&=
		(2+M^2)\frac{L}{\pi}
		\int_{\T}\psi G(\eta)\psi\,\dd x,
	\end{align*}
	which is \eqref{eq:mean-zero-trace}.
\end{proof}

\subsection{Estimates for the remainder terms}

We next estimate the four terms on the right-hand side of
\eqref{eq:hydro-multiplier}.
Lemma~\ref{lem:remainder} separates the contributions controlled by the
initial energy from those absorbed into the coercive left-hand side of
\eqref{eq:hydro-multiplier}.

\begin{lemma}\label{lem:remainder}
Assume \(\|\eta_x\|_{L^\infty}\leq1\).  Let \(A_m\) be as in
Lemma~\ref{lem:zeta-control}, and set
\begin{equation*}
  C_\zeta^\sharp
  :=\frac{2^{7/2}\|m\|_{L^\infty}^2}{\beta}
  +\frac{2\bigl(\|m\|_{L^\infty}^2+2A_m^2\bigr)}{g}.
\end{equation*}
Thus Lemma~\ref{lem:zeta-control}, applied with \(\varepsilon_1=1\), gives
\(\|\zeta(t)\|_2^2\leq C_\zeta^\sharp\Etot(t)\) whenever
\(\|\eta_x\|_\infty\leq1\).  Define
\begin{align}
  A_{-1}&=\frac{\lambda \|\chi\|_{L^\infty}}{2}
  +\frac{75L}{8\pi\lambda}\|\chi\|_{L^\infty}
  +\frac{L^2}{4\lambda},\label{eq:A-minus}\\
  A_{+1}&=\frac8\lambda,\label{eq:A-plus}\\
  A_E&=\frac{C_\zeta^\sharp}{2}+1+8\|\chi_x\|_{L^\infty},\label{eq:A-E}\\
  C_{\mathrm{end}}&=2\sqrt{2C_\zeta^\sharp \frac{3L}{\pi}}.
  \notag
\end{align}
For every \(a\in(0,1]\), we have
\begin{align}
 \notag &\left|\int_0^T\int_{\T}
  \left[\frac52\chi\psit+(x-m)\psi_x\right]
  G(\eta)\psi\,\dd x\dd t\right|
  +\left|\int_0^T\int_{\T}\Pt\zeta\,\dd x\dd t\right|\\
 \leq&
  \left(\frac{A_{-1}}a+aA_{+1}\right)\Etot(0)
 \quad+aA_E\int_0^T\Etot(t)\,\dd t.
  \label{eq:remainder-estimate}
\end{align}
Furthermore,
\begin{equation}\label{eq:endpoint-estimate}
  \left|\left[\int_{\T}\zeta\psit\,\dd x\right]_0^T\right|
  \leq C_{\mathrm{end}}\Etot(0),
\end{equation}
and
\begin{equation}\label{eq:bulk-estimate}
  \left|\int_0^T\iint_{\Omega(t)}
  \varrho_x\phi_x\phi_y\,\dd y\dd x\dd t\right|
  \leq\|\varrho_x\|_{L^\infty([0,T]\times[-L,L])}
  \int_0^T\Etot(t)\,\dd t.
\end{equation}
\end{lemma}

\begin{proof}
	Recall that
	\[
	D(t)
	:=
	\int_{\T}
	\chi(x)\bigl(G(\eta)\psi\bigr)^2(t,x)\,\dd x.
	\]
	By the feedback law \eqref{eq:feedback}, we have
	\[
	P_{\mathrm{ext}}
	=
	\lambda\chi G(\eta)\psi.
	\]
	Recall that
	\[
	\Pt(t)
	=
	P_{\mathrm{ext}}(t)-\avg{P_{\mathrm{ext}}(t)}
	\]
	is the mean-zero part of the exterior pressure.
	Since subtraction of the spatial mean is the orthogonal projection of
	\(L^2(\T)\) onto its mean-zero subspace, for every \(t\in[0,T]\),
	\[
	\|\Pt(t)\|_{L^2(\T)}
	\leq
	\|P_{\mathrm{ext}}(t)\|_{L^2(\T)}.
	\]
	Using \(\chi\geq0\), we therefore obtain
	\begin{align*}
		\|\Pt(t)\|_{L^2(\T)}^2
		&\leq
		\lambda^2
		\int_{\T}
		\chi(x)^2
		\bigl(G(\eta)\psi\bigr)^2(t,x)\,\dd x
		\\
		&\leq
		\lambda^2\|\chi\|_{L^\infty(\T)}
		\int_{\T}
		\chi(x)
		\bigl(G(\eta)\psi\bigr)^2(t,x)\,\dd x
		\\
		&=
		\lambda^2\|\chi\|_{L^\infty(\T)}D(t).
	\end{align*}
	The dissipation identity \eqref{eq:energy-dissipation} gives
	\[
	\Etot(T)-\Etot(0)
	=
	-\lambda\int_0^T D(t)\,\dd t.
	\]
	Since \(\Etot(T)\geq0\),
	\begin{equation}\label{eq:D-time-control-proof}
		\int_0^T D(t)\,\dd t
		\leq
		\frac{\Etot(0)}{\lambda}.
	\end{equation}
	Consequently,
	\begin{align}
		\int_0^T
		\|\Pt(t)\|_{L^2(\T)}^2\,\dd t
		&\leq
		\lambda^2\|\chi\|_{L^\infty(\T)}
		\int_0^T D(t)\,\dd t
		\notag\\
		&\leq
		\lambda\|\chi\|_{L^\infty(\T)}
		\Etot(0).
		\label{eq:Ptilde-time-control-proof}
	\end{align}
	
	We first estimate the pressure contribution.  By the Cauchy--Schwarz
	inequality in space and time,
	\begin{align*}
		\left|
		\int_0^T\int_{\T}
		\Pt(t,x)\zeta(t,x)\,\dd x\dd t
		\right|
		\leq
		\left(
		\int_0^T
		\|\Pt(t)\|_{L^2(\T)}^2\,\dd t
		\right)^{\frac{1}{2}}\left(
		\int_0^T
		\|\zeta(t)\|_{L^2(\T)}^2\,\dd t
		\right)^{\frac{1}{2}}.
	\end{align*}
	For \(a\in(0,1]\), Young's inequality gives
	\begin{align*}
		\left|
		\int_0^T\int_{\T}
		\Pt\zeta\,\dd x\dd t
		\right|
		&\leq
		\frac1{2a}
		\int_0^T
		\|\Pt(t)\|_{L^2(\T)}^2\,\dd t
		+
		\frac a2
		\int_0^T
		\|\zeta(t)\|_{L^2(\T)}^2\,\dd t.
	\end{align*}
	By \eqref{eq:Ptilde-time-control-proof} and
	Lemma~\ref{lem:zeta-control}, applied with \(\varepsilon_1=1\),
	\[
	\|\zeta(t)\|_{L^2(\T)}^2
	\leq
	C_\zeta^\sharp\Etot(t),
	\qquad 0\leq t\leq T.
	\]
	Hence
	\begin{equation}\label{eq:pressure-remainder-proof}
		\left|
		\int_0^T\int_{\T}
		\Pt\zeta\,\dd x\dd t
		\right|
		\leq
		\frac{\lambda\|\chi\|_{L^\infty(\T)}}{2a}
		\Etot(0)
		+
		\frac{aC_\zeta^\sharp}{2}
		\int_0^T\Etot(t)\,\dd t.
	\end{equation}
	
	We next estimate the first observation term.  Since
	\(\|\eta_x(t)\|_{L^\infty(\T)}\leq1\), the mean-zero trace estimate
	\eqref{eq:mean-zero-trace}, with \(M=1\), yields
	\begin{align*}
		\|\psit(t)\|_{L^2(\T)}^2
		&\leq
		\frac{3L}{\pi}
		\int_{\T}
		\psi(t,x)G(\eta)\psi(t,x)\,\dd x
		\\
		&\leq
		\frac{6L}{\pi}\Etot(t).
	\end{align*}
	For every \(t\in[0,T]\), the Cauchy--Schwarz inequality therefore gives
	\begin{align*}
		\left|
		\frac52
		\int_{\T}
		\chi(x)\psit(t,x)
		G(\eta)\psi(t,x)\,\dd x
		\right|
		&\leq
		\frac52
		\left(
		\int_{\T}
		\chi(x)\psit(t,x)^2\,\dd x
		\right)^{\frac{1}{2}}
		D(t)^{\frac{1}{2}}
		\\
		&\leq
		\frac52
		\|\chi\|_{L^\infty(\T)}^{\frac{1}{2}}
		\|\psit(t)\|_{L^2(\T)}
		D(t)^{\frac{1}{2}}
		\\
		&\leq
		\frac52
		\sqrt{
			\frac{6L}{\pi}
			\|\chi\|_{L^\infty(\T)}
		}\,
		\Etot(t)^{\frac{1}{2}}D(t)^{\frac{1}{2}}.
	\end{align*}
	Applying Young's inequality in the form
	\[
	AB\leq aA^2+\frac1{4a}B^2
	\]
	gives
	\begin{align}
		\left|
		\frac52
		\int_{\T}
		\chi\psit G(\eta)\psi\,\dd x
		\right|
		\leq
		a\Etot(t)
		+
		\frac{75L}{8\pi a}
		\|\chi\|_{L^\infty(\T)}D(t).
		\label{eq:first-observation-proof}
	\end{align}
	Integrating \eqref{eq:first-observation-proof} over \(t\in[0,T]\) and
	using \eqref{eq:D-time-control-proof}, we obtain
	\begin{align}
		\left|
		\frac52
		\int_0^T\int_{\T}
		\chi\psit G(\eta)\psi\,\dd x\dd t
		\right|
		\leq
		a\int_0^T\Etot(t)\,\dd t
		+
		\frac{75L}{8\pi\lambda a}
		\|\chi\|_{L^\infty(\T)}
		\Etot(0).
		\label{eq:first-observation-time-proof}
	\end{align}
	
	For the second observation term, Lemma~\ref{lem:cutoff} gives
	\[
	|x-m(x)|\leq L\chi(x)
	\qquad
	\text{for }x\in[-L,L].
	\]
	Thus, for every \(t\in[0,T]\),
	\begin{align*}
		\left|
		\int_{\T}
		(x-m(x))\psi_x(t,x)
		G(\eta)\psi(t,x)\,\dd x
		\right|
		&\leq
		L
		\int_{\T}
		\chi(x)
		|\psi_x(t,x)|
		|G(\eta)\psi(t,x)|\,\dd x
		\\
		&\leq
		L
		\left(
		\int_{\T}
		\chi(x)\psi_x(t,x)^2\,\dd x
		\right)^{\frac{1}{2}}
		D(t)^{\frac{1}{2}}.
	\end{align*}
	Young's inequality yields
	\begin{align*}
		\left|
		\int_{\T}
		(x-m)\psi_xG(\eta)\psi\,\dd x
		\right|
		\leq
		a
		\int_{\T}
		\chi\psi_x^2\,\dd x
		+
		\frac{L^2}{4a}D(t).
	\end{align*}
	The localized tangential trace estimate
	\eqref{eq:trace-estimate} gives
	\[
	\int_{\T}
	\chi(x)\psi_x(t,x)^2\,\dd x
	\leq
	8D(t)
	+
	8\|\chi_x\|_{L^\infty(\T)}\Etot(t).
	\]
	Consequently,
	\begin{align}
		\left|
		\int_{\T}
		(x-m)\psi_xG(\eta)\psi\,\dd x
		\right|
		\leq
		\left(
		8a+\frac{L^2}{4a}
		\right)D(t)
		+
		8a\|\chi_x\|_{L^\infty(\T)}\Etot(t).
		\label{eq:second-observation-proof}
	\end{align}
	Integrating \eqref{eq:second-observation-proof} over \(t\in[0,T]\) and
	using \eqref{eq:D-time-control-proof}, we obtain
	\begin{align}
		\left|
		\int_0^T\int_{\T}
		(x-m)\psi_xG(\eta)\psi\,\dd x\dd t
		\right|
		\leq
		\left(
		\frac{8a}{\lambda}
		+
		\frac{L^2}{4\lambda a}
		\right)\Etot(0)
		+
		8a\|\chi_x\|_{L^\infty(\T)}
		\int_0^T\Etot(t)\,\dd t.
		\label{eq:second-observation-time-proof}
	\end{align}
	
	By the triangle inequality,
	\begin{align*}
		\left|
		\int_0^T\int_{\T}
		\left[
		\frac52\chi\psit+(x-m)\psi_x
		\right]
		G(\eta)\psi\,\dd x\dd t
		\right|
		&\leq
		\left|
		\frac52
		\int_0^T\int_{\T}
		\chi\psit G(\eta)\psi\,\dd x\dd t
		\right|
		\\
		&\quad+
		\left|
		\int_0^T\int_{\T}
		(x-m)\psi_xG(\eta)\psi\,\dd x\dd t
		\right|.
	\end{align*}
	Set \[\mathcal{Q}=\int_0^T\int_{\T}
	\left[
	\frac52\chi\psit+(x-m)\psi_x
	\right]
	G(\eta)\psi\,\dd x\dd t.\]
	Combining
	\eqref{eq:pressure-remainder-proof},
	\eqref{eq:first-observation-time-proof}, and
	\eqref{eq:second-observation-time-proof}, we find
	\begin{align*}
		\left|
		\mathcal{Q}
		\right|
		+
		\left|
		\int_0^T\int_{\T}
		\Pt\zeta\,\dd x\dd t
		\right|
		&\leq
		\Bigg[
		\frac1a
		\left(
		\frac{\lambda\|\chi\|_{L^\infty(\T)}}2
		+
		\frac{75L}{8\pi\lambda}
		\|\chi\|_{L^\infty(\T)}
		+
		\frac{L^2}{4\lambda}
		\right)
		+
		\frac{8a}{\lambda}
		\Bigg]\Etot(0)
		\\
		&\quad+
		a
		\left(
		\frac{C_\zeta^\sharp}{2}
		+1
		+8\|\chi_x\|_{L^\infty(\T)}
		\right)
		\int_0^T\Etot(t)\,\dd t.
	\end{align*}
	By the definitions
	\eqref{eq:A-minus}--\eqref{eq:A-E}, this is precisely
	\eqref{eq:remainder-estimate}.
	
	We now estimate the endpoint contribution.  For every \(t\in[0,T]\),
	Cauchy--Schwarz, Lemma~\ref{lem:zeta-control}, and
	\eqref{eq:mean-zero-trace} with \(M=1\) give
	\begin{align*}
		\left|
		\int_{\T}
		\zeta(t,x)\psit(t,x)\,\dd x
		\right|
		&\leq
		\|\zeta(t)\|_{L^2(\T)}
		\|\psit(t)\|_{L^2(\T)}
		\\
		&\leq
		\sqrt{C_\zeta^\sharp\Etot(t)}
		\sqrt{\frac{6L}{\pi}\Etot(t)}
		\\
		&=
		\sqrt{\frac{6LC_\zeta^\sharp}{\pi}}\,
		\Etot(t).
	\end{align*}
	Therefore
	\begin{align*}
		\left|
		\left[
		\int_{\T}
		\zeta(t,x)\psit(t,x)\,\dd x
		\right]_{t=0}^{t=T}
		\right|
		&\leq
		\sqrt{\frac{6LC_\zeta^\sharp}{\pi}}
		\bigl(\Etot(T)+\Etot(0)\bigr).
	\end{align*}
	By the monotonicity of the energy,
	\[
	\Etot(T)\leq\Etot(0),
	\]
	and hence
	\begin{align*}
		\left|
		\left[
		\int_{\T}
		\zeta\psit\,\dd x
		\right]_0^T
		\right|
		&\leq
		2\sqrt{\frac{6LC_\zeta^\sharp}{\pi}}\,
		\Etot(0)
		\\
		&=
		2\sqrt{
			2C_\zeta^\sharp\frac{3L}{\pi}
		}\,
		\Etot(0)
		\\
		&=
		C_{\mathrm{end}}\Etot(0).
	\end{align*}
	This proves \eqref{eq:endpoint-estimate}.
	
	Finally, we estimate the bulk term.  Set
	\[\mathcal{G}=\iint_{\Omega(t)}
	\varrho_x(t,x)
	\phi_x(t,x,y)\phi_y(t,x,y)
	\,\dd y\dd x.\]
	At each \(t\in[0,T]\),
	\begin{align*}
		\left|
		\mathcal{G}
		\right|
		&=
		\left|
		\int_{-L}^{L}
		\int_{-\infty}^{\eta(t,x)}
		\varrho_x(t,x)
		\phi_x(t,x,y)\phi_y(t,x,y)
		\,\dd y\dd x
		\right|
		\\
		&\quad\leq
		\|\varrho_x(t)\|_{L^\infty([-L,L])}
		\int_{-L}^{L}
		\int_{-\infty}^{\eta(t,x)}
		|\phi_x(t,x,y)\phi_y(t,x,y)|
		\,\dd y\dd x
		\\
		&\quad\leq
		\frac12
		\|\varrho_x(t)\|_{L^\infty([-L,L])}
		\int_{-L}^{L}
		\int_{-\infty}^{\eta(t,x)}
		\left(
		|\phi_x(t,x,y)|^2
		+
		|\phi_y(t,x,y)|^2
		\right)
		\,\dd y\dd x.
	\end{align*}
	By Lemma \ref{lem:kinetic-boundary},
	\[
	\frac12
	\iint_{\Omega(t)}
	|\nabla\phi(t,x,y)|^2\,\dd y\dd x
	=
	\frac12
	\int_{\T}
	\psi(t,x)G(\eta)\psi(t,x)\,\dd x
	\leq
	\Etot(t).
	\]
	It follows that
	\[
	\left|
	\iint_{\Omega(t)}
	\varrho_x\phi_x\phi_y\,\dd y\dd x
	\right|
	\leq
	\|\varrho_x(t)\|_{L^\infty([-L,L])}
	\Etot(t).
	\]
	Integrating over \(t\in[0,T]\) gives
	\begin{align*}
		\left|
		\int_0^T
		\iint_{\Omega(t)}
		\varrho_x\phi_x\phi_y\,\dd y\dd x\dd t
		\right|
		&\leq
		\int_0^T
		\|\varrho_x(t)\|_{L^\infty([-L,L])}
		\Etot(t)\,\dd t
		\\
		&\leq
		\|\varrho_x\|_{L^\infty([0,T]\times[-L,L])}
		\int_0^T\Etot(t)\,\dd t,
	\end{align*}
	which is \eqref{eq:bulk-estimate}.
\end{proof}

\section{Stabilization theorem}
\label{sec:main-theorem}

In this section, we combine the flexural coercivity estimate
\eqref{eq:flex-coercivity} with the remainder estimates
\eqref{eq:remainder-estimate}, \eqref{eq:endpoint-estimate}, and
\eqref{eq:bulk-estimate} to prove the main stabilization theorem and its
consequences. 

We first reduce the geometric assumptions to the global slope
\(\|\eta_x\|_{L^\infty}\) and the localized curvature quantity
\(\|\chi\eta_{xx}\|_{L^\infty}\).
\begin{lemma}
\label{lem:localized-geometric-reduction}
Let \((\eta,\psi,P_{\mathrm{ext}})\) be a solution of
\eqref{eq:hydroelastic-system} on \([0,T]\), and fix \(t\in[0,T]\).
Then
\begin{equation}\label{eq:zero-mean-Linfty}
  \|\eta\|_{L^\infty(\T)}
  \leq L\|\eta_x\|_{L^\infty(\T)}.
\end{equation}
Moreover, with \(\varrho\) defined in \eqref{eq:rho},
\begin{align}
  \|\varrho_x\|_{L^\infty}
  &\leq
  \left(
    \left\|\frac94-\frac12m_x\right\|_{L^\infty}
    +\frac32L\|m_{xx}\|_{L^\infty}
  \right)
  \|\eta_x\|_{L^\infty}
  +L\|\chi\eta_{xx}\|_{L^\infty}.
  \label{eq:localized-rho-bound}
\end{align}
\end{lemma}

\begin{proof}
Lemma~\ref{lem:zero-mode} gives
\[
  \int_{\T}\eta(t,x)\,\dd x=0.
\]
If \(\eta\equiv0\), the first assertion is immediate.  Otherwise, continuity
implies that \(\eta\) takes both nonnegative and nonpositive values.  Hence
there is \(x_0\in\T\) with \(\eta(x_0)=0\).  Every point of the circle
\(\T=\R/(2L\mathbb Z)\) can be connected to \(x_0\) by an arc of length at
most \(L\).  Therefore
\[
  |\eta(x)|
  =\left|\int_{x_0}^{x}\eta_x(s)\,\dd s\right|
  \leq L\|\eta_x\|_{L^\infty},
\]
which proves \eqref{eq:zero-mean-Linfty}.

Differentiating
\[
  \varrho
  =(m-x)\eta_x
  +\left(\frac{13}{4}-\frac32m_x\right)\eta
\]
gives
\begin{equation*}
  \varrho_x
  =\left(\frac94-\frac12m_x\right)\eta_x
   +(m-x)\eta_{xx}
   -\frac32m_{xx}\eta.
\end{equation*}
By Lemma~\ref{lem:cutoff}, \(|m-x|\leq L\chi\).  Combining this pointwise
bound with \eqref{eq:zero-mean-Linfty} yields
\eqref{eq:localized-rho-bound}.
\end{proof}
We now prove Theorem~\ref{thm:stabilization}.  For use in the coercivity
estimates, set
\[
  M_1=\|m_x\|_{L^\infty(\T)},
  \qquad
  M_2=\|m_{xx}\|_{L^\infty(\T)},
  \qquad
  M_3=\|m_{xxx}\|_{L^\infty(\T)},
  \qquad
  K_1=\frac{27}{8}+\frac54M_1.
\]

\begin{proof}[Proof of Theorem~\ref{thm:stabilization}]
By Lemma~\ref{lem:zero-mode},
\begin{equation*}
  \int_{\T}\eta(t,x)\,\dd x=0
  \qquad(0\leq t\leq T).
\end{equation*}

All coercivity parameters are fixed before estimating the multiplier
identity.  Define
\begin{equation*}
  q_*:=\frac{25\beta M_3^2}{9g},
  \qquad
  \theta_*:=\frac{1+q_*}{2}.
\end{equation*}
By \eqref{eq:M3-positive} and \eqref{eq:strict-parameter-gap},
\(0<q_*<1\), hence
\begin{equation}\label{eq:theta-range}
  q_*<\theta_*<1.
\end{equation}
For \(0\leq\delta\leq1\), set
\begin{align*}
  R_*(\delta)
  &:={}
  \frac{\frac94-K_1\delta^2}{1+\delta^2}
  -\frac{15LM_2}{4}\delta^2,\\
  K_Y(\delta)
  &:={}
  M_3\left(
    \frac12(1+\delta^2)^{5/4}+\frac34
  \right).
\end{align*}
Whenever \(R_*(\delta)>0\), define
\begin{equation*}
  c_{\mathrm{fl}}(\delta)
  :=(1-\theta_*)R_*(\delta),
  \qquad
  C_{\mathrm{fl}}(\delta)
  :=\frac{K_Y(\delta)^2}{\theta_*R_*(\delta)}.
\end{equation*}

At \(\delta=0\),
\begin{equation*}
  R_*(0)=\frac94,
  \qquad
  c_{\mathrm{fl}}(0)=\frac94(1-\theta_*),
  \qquad
  C_{\mathrm{fl}}(0)=\frac{25M_3^2}{36\theta_*}.
\end{equation*}
In view of \eqref{eq:theta-range},
\begin{equation*}
  \frac12-\frac{2\beta C_{\mathrm{fl}}(0)}{g}
  =\frac12-\frac{25\beta M_3^2}{18g\theta_*}
  =\frac12\left(1-\frac{q_*}{\theta_*}\right)>0.
\end{equation*}
Define the positive flat coercivity margin
\begin{equation*}
  c_E^0
  :=\min\left\{
    \frac12,
    \frac12\left(1-\frac{q_*}{\theta_*}\right),
    \frac92(1-\theta_*)
  \right\}>0.
\end{equation*}
Whenever \(R_*(\delta)>0\), set
\begin{equation}\label{eq:cE-delta}
  c_E(\delta)
  :=\min\left\{
    \frac12,
    \frac12-\frac{2\beta C_{\mathrm{fl}}(\delta)}{g},
    2c_{\mathrm{fl}}(\delta)
  \right\}.
\end{equation}
Since \(R_*(0)=9/4>0\), the functions
\(R_*(\delta)\), \(c_{\mathrm{fl}}(\delta)\), and
\(C_{\mathrm{fl}}(\delta)\) are continuous for \(\delta\) in a
neighborhood of \(0\).  Hence \(c_E(\delta)\), being the minimum of three
continuous functions there, is also continuous at \(0\), and
\(c_E(0)=c_E^0\).  Therefore there exists
\(\delta_0\in(0,1]\) such that for every
\(0\leq\delta\leq\delta_0\),
\begin{equation*}
  \frac94-K_1\delta^2\geq0,
  \qquad
  R_*(\delta)>0,
  \qquad
  c_E(\delta)\geq\frac12c_E^0.
\end{equation*}
In particular, the middle quantity in \eqref{eq:cE-delta} is positive for
\(0\leq\delta\leq\delta_0\).

Set
\begin{equation*}
  C_{\varrho}^{\mathrm{loc}}
  :=\max\left\{
  \left\|\frac94-\frac12m_x\right\|_{L^\infty}
  +\frac32LM_2,\,L\right\}.
\end{equation*}
Choose
\begin{equation}\label{eq:delta-star-explicit}
  0<\delta_*
  \leq
  \min\left\{
    \delta_0,
    \frac{c_E^0}{8C_{\varrho}^{\mathrm{loc}}}
  \right\}.
\end{equation}
The denominator is strictly positive because \(L>0\), so
\(C_{\varrho}^{\mathrm{loc}}>0\).  Fix this choice of \(\delta_*\) for the
remainder of the proof.

Consider a solution satisfying \eqref{eq:localized-geometric-smallness}.
Lemma~\ref{lem:localized-geometric-reduction} gives
\begin{equation*}
  \|\eta(t)\|_{L^\infty}
  \leq L\delta_*,
  \qquad
  \|\eta_x(t)\|_{L^\infty}\leq\delta_*
  \qquad(0\leq t\leq T).
\end{equation*}
Thus Lemma~\ref{lem:flex-coercivity} applies with
\begin{equation*}
  \varepsilon_0=L\delta_*,
  \qquad
  \varepsilon_1=\delta_*.
\end{equation*}
With these values, \(R_{\mathrm{fl}}\) in \eqref{eq:Rfl} equals
\(R_*(\delta_*)\).  By \eqref{eq:M3-positive}, \(M_3>0\), so we choose
\(d_1,d_2\) by Proposition~\ref{prop:flex-parameter-feasibility} with
\(\theta=\theta_*\).  We then obtain
\begin{equation}\label{eq:flex-coercivity-main-constants}
  \int_{\T}\Bflex(\eta)\zeta\,\dd x
  \geq
  c_{\mathrm{fl}}(\delta_*)
  \int_{\T}\frac{\eta_{xx}^2}{(1+\eta_x^2)^{5/2}}\,\dd x
  -C_{\mathrm{fl}}(\delta_*)\int_{\T}\eta^2\,\dd x.
\end{equation}

The bulk coefficient is controlled by the localized curvature hypothesis.
Indeed, if
\[
  A_\varrho:=
  \left\|\frac94-\frac12m_x\right\|_{L^\infty}+\frac32LM_2,
\]
then \eqref{eq:localized-rho-bound} and
\eqref{eq:localized-geometric-smallness} give
\[
  \|\varrho_x\|_\infty
  \leq A_\varrho\|\eta_x\|_\infty
      +L\|\chi\eta_{xx}\|_\infty
  \leq C_{\varrho}^{\mathrm{loc}}
      \bigl(\|\eta_x\|_\infty+\|\chi\eta_{xx}\|_\infty\bigr).
\]
Therefore \eqref{eq:delta-star-explicit} yields
\begin{align}
  \|\varrho_x\|_{L^\infty([0,T]\times[-L,L])}
  &\leq C_{\varrho}^{\mathrm{loc}}\delta_*
  \leq\frac18c_E^0
  \leq\frac14c_E(\delta_*).
  \label{eq:rho-absorption-main}
\end{align}
For the last inequality we used
\(c_E(\delta_*)\geq c_E^0/2\).

Put
\[
  X(t)=\int_{\T}\psi G(\eta)\psi\,\dd x,
  \quad
  Y(t)=\int_{\T}\eta^2\,\dd x,
  \quad
  Z(t)=\int_{\T}
  \frac{\eta_{xx}^2}{(1+\eta_x^2)^{5/2}}\,\dd x.
\]
Then
\begin{equation}\label{eq:energy-XYZ}
  \Etot(t)=\frac12X(t)+\frac g2Y(t)+\frac\beta2Z(t).
\end{equation}
Since
\(\mathcal H_g=\frac12X+\frac g2Y\),
\eqref{eq:flex-coercivity-main-constants} gives
\begin{align}
  \frac12\mathcal H_g(t)
  +\beta\int_{\T}\Bflex(\eta)\zeta\,\dd x
  &\geq
  \frac14X(t)
  +\left(\frac g4-\beta C_{\mathrm{fl}}(\delta_*)\right)Y(t)
  +\beta c_{\mathrm{fl}}(\delta_*)Z(t)
  \notag\\
  &\geq c_E(\delta_*)\Etot(t).
  \label{eq:coercive-LHS-energy}
\end{align}
By the definition of \(c_E(\delta_*)\) in \eqref{eq:cE-delta},
\[
  \frac14
  \geq
  \frac{c_E(\delta_*)}{2},
  \qquad
  \frac g4-\beta C_{\mathrm{fl}}(\delta_*)
  \geq
  \frac{g\,c_E(\delta_*)}{2},
  \qquad
  \beta c_{\mathrm{fl}}(\delta_*)
  \geq
  \frac{\beta\,c_E(\delta_*)}{2}.
\]
Comparing these bounds with the decomposition
\eqref{eq:energy-XYZ} gives the second inequality in
\eqref{eq:coercive-LHS-energy}.  In addition, the remaining terms on the
left-hand side of \eqref{eq:hydro-multiplier},
\[
  2g\int_{\T}\chi\eta^2\,\dd x
  \qquad\text{and}\qquad
  \Sigma_\infty,
\]
are nonnegative and may therefore be discarded in the lower bound.

For brevity set
\[
  c_E:=c_E(\delta_*),
  \qquad
  E_T:=\int_0^T\Etot(t)\,\dd t.
\]
Because \(\delta_*\leq1\), Lemma~\ref{lem:remainder} applies.  Taking
absolute values on the right-hand side of
\eqref{eq:hydro-multiplier} and using
\eqref{eq:coercive-LHS-energy},
\eqref{eq:rho-absorption-main}, and Lemma~\ref{lem:remainder}, we obtain for
every \(a\in(0,1]\)
\begin{align}
  c_EE_T
  &\leq
  \left(\frac{A_{-1}}a+aA_{+1}+C_{\mathrm{end}}\right)\Etot(0)
  +\left(aA_E+\frac{c_E}{4}\right)E_T.
  \label{eq:pre-absorption}
\end{align}
Choose
\begin{equation*}
  a_*:=\min\left\{1,\frac{c_E}{4A_E}\right\}.
\end{equation*}
Since \(A_E>0\), this is a well-defined number in \((0,1]\), and
\(a_*A_E\leq c_E/4\).  Hence \eqref{eq:pre-absorption} gives
\[
  \frac{c_E}{2}E_T
  \leq
  \left(
    \frac{A_{-1}}{a_*}+a_*A_{+1}+C_{\mathrm{end}}
  \right)\Etot(0).
\]
Therefore \eqref{eq:integrated-energy-estimate} holds with
\begin{equation}\label{eq:observability-constant}
  C
  :=\frac2{c_E}
  \left(
    \frac{A_{-1}}{a_*}+a_*A_{+1}+C_{\mathrm{end}}
  \right).
\end{equation}
All quantities on the right depend only on
\(g,\beta,L,\lambda\) and the fixed multiplier \(m\).  In particular they
are independent of \(T\).

Finally, Proposition~\ref{prop:dissipation} implies that \(\Etot\) is
nonincreasing.  Thus
\[
  T\Etot(T)
  \leq\int_0^T\Etot(t)\,\dd t
  \leq C\Etot(0),
\]
which proves \eqref{eq:one-step-decay}.
\end{proof}
We next prove the localized dissipation-observability estimate stated in
Corollary~\ref{cor:dissipation-observability}.

\begin{proof}[Proof of Corollary~\ref{cor:dissipation-observability}]
Integrating the exact dissipation identity
\eqref{eq:energy-dissipation} on \([0,T]\) gives
\begin{equation}\label{eq:exact-dissipation-on-I}
  \Etot(0)-\Etot(T)
  =\lambda
  \int_0^T\int_{\T}
  \chi|\eta_t|^2\,\dd x\dd t.
\end{equation}
By \eqref{eq:one-step-decay},
\[
  \Etot(T)\leq\frac CT\Etot(0).
\]
Substitution in \eqref{eq:exact-dissipation-on-I} yields
\[
  \lambda\int_0^T\int_{\T}\chi|\eta_t|^2
  \geq\left(1-\frac CT\right)\Etot(0).
\]
Since \(T>C\), the coefficient on the right is positive, proving
\eqref{eq:dissipation-observability}.  If \(T\geq2C\), then
\(1-C/T\geq1/2\), which gives
\eqref{eq:dissipation-observability-2C}.
\end{proof}
The stabilization estimate remains valid on every time interval lying
strictly before the localized geometric condition reaches the threshold
\(\delta_*\).

\begin{corollary}
	\label{cor:geometric-exit-time}
	Let \((\eta,\psi,P_{\mathrm{ext}})\) be a solution of
	\eqref{eq:hydroelastic-system} on \([0,T_*)\), with
	\(P_{\mathrm{ext}}\) given by \eqref{eq:feedback}.  Define
	\begin{equation*}
		\mathcal G(t)
		:=
		\|\eta_x(t)\|_{L^\infty(\T)}
		+\|\chi\eta_{xx}(t)\|_{L^\infty(\T)}.
	\end{equation*}
	Assume that
	\[
	\mathcal G(0)<\delta_*,
	\]
	and set
	\begin{equation}\label{eq:Tgeo}
		T_{\mathrm{geo}}
		:=
		\sup\left\{
		\tau\in(0,T_*):
		\sup_{0\leq t\leq\tau}\mathcal G(t)<\delta_*
		\right\}.
	\end{equation}
	Then \(T_{\mathrm{geo}}>0\).  Moreover, for every
	\(0<T<T_{\mathrm{geo}}\),
	\begin{equation}\label{eq:exit-time-integrated}
		\int_0^T\Etot(t)\,\dd t
		\leq C\Etot(0),
		\qquad
		\Etot(T)
		\leq\frac{C}{T}\Etot(0).
	\end{equation}
	If \(T_{\mathrm{geo}}>C\), then
	\eqref{eq:dissipation-observability} also holds on every interval
	\([0,T]\) with
	\[
	C<T<T_{\mathrm{geo}}.
	\]
\end{corollary}

\begin{proof}
	Fix \(T_0<T_*\).  By the regularity assumptions of
	Theorem~\ref{thm:stabilization}, for some \(s\geq5\) we have
	\(\eta\in C([0,T_0];H^{s+3/2}(\T))\).  Sobolev embedding therefore
	implies that
	\[
	(t,x)\longmapsto \eta_x(t,x)
	\qquad\text{and}\qquad
	(t,x)\longmapsto \eta_{xx}(t,x)
	\]
	are continuous on the compact set \([0,T_0]\times\T\).  Hence they are
	uniformly continuous there.  It follows that, whenever \(t\to s\) in
	\([0,T_0]\),
	\[
	\|\eta_x(t)-\eta_x(s)\|_{L^\infty(\T)}
	\longrightarrow0,
	\qquad
	\|\eta_{xx}(t)-\eta_{xx}(s)\|_{L^\infty(\T)}
	\longrightarrow0.
	\]
	Therefore
	\begin{align*}
		|\mathcal G(t)-\mathcal G(s)|
		\leq
		\|\eta_x(t)-\eta_x(s)\|_{L^\infty(\T)}
		+\|\chi\|_{L^\infty(\T)}
		\|\eta_{xx}(t)-\eta_{xx}(s)\|_{L^\infty(\T)}
		\longrightarrow0.
	\end{align*}
	Thus \(t\mapsto\mathcal G(t)\) is continuous on every compact subinterval
	of \([0,T_*)\).
	
	Since \(\mathcal G(0)<\delta_*\), continuity at \(t=0\) yields
	\(\tau_0\in(0,T_*)\) such that
	\[
	\mathcal G(t)<\delta_*
	\qquad
	\text{for all }0\leq t\leq\tau_0.
	\]
	Consequently,
	\[
	\sup_{0\leq t\leq\tau_0}\mathcal G(t)<\delta_*,
	\]
	after decreasing \(\tau_0\) if necessary.  Hence the set in
	\eqref{eq:Tgeo} is nonempty and
	\[
	T_{\mathrm{geo}}\geq\tau_0>0.
	\]
	
	Now let \(0<T<T_{\mathrm{geo}}\).  Since \(T_{\mathrm{geo}}\) is the
	supremum of the set in \eqref{eq:Tgeo}, there exists
	\(\tau\in(T,T_*)\) such that
	\[
	\sup_{0\leq t\leq\tau}\mathcal G(t)<\delta_*.
	\]
	Because \(T<\tau\),
	\[
	\sup_{0\leq t\leq T}\mathcal G(t)
	\leq
	\sup_{0\leq t\leq\tau}\mathcal G(t)
	<\delta_*.
	\]
	Thus the geometric hypothesis of
	Theorem~\ref{thm:stabilization} holds on \([0,T]\), and
	\eqref{eq:exit-time-integrated} follows.
	
	Finally, if \(T_{\mathrm{geo}}>C\) and
	\(C<T<T_{\mathrm{geo}}\), the same argument shows that the hypotheses of
	Theorem~\ref{thm:stabilization} hold on \([0,T]\).  Since \(T>C\),
	Corollary~\ref{cor:dissipation-observability} then gives
	\eqref{eq:dissipation-observability}.
\end{proof}

It remains to justify the iteration leading to the exponential estimate in
Corollary~\ref{cor:optimized-exponential}.

\begin{proof}[Proof of Corollary~\ref{cor:optimized-exponential}]
Fix \(\tau>C\) and put \(q=C/\tau\in(0,1)\).  For each integer
\(n\geq0\), apply Theorem~\ref{thm:stabilization} on \([0,\tau]\) to the
time-shifted solution
\[
  (\eta_n,\psi_n)(s)
  :=(\eta,\psi)(n\tau+s),
  \qquad 0\leq s\leq\tau.
\]
The global hypothesis \eqref{eq:global-localized-geometric-smallness} is
preserved under this time shift.  Hence, for every integer \(n\geq0\),
\[
\Etot((n+1)\tau)
\leq q\,\Etot(n\tau).
\]
Iterating the one-step inequality
\(\Etot((n+1)\tau)\leq q\,\Etot(n\tau)\) gives
\[
\Etot(n\tau)
\leq q^n\Etot(0),
\qquad n\geq0.
\]
For arbitrary \(t\geq0\), set
\[
n=\left\lfloor\frac{t}{\tau}\right\rfloor.
\]
Then
\[
n
\geq
\frac{t}{\tau}-1.
\]
Since \(0<q<1\), the function \(s\mapsto q^s\) is decreasing.  Therefore
\begin{align*}
	\Etot(t)
	&\leq q^n\Etot(0)\\
	&\leq q^{\,t/\tau-1}\Etot(0)\\
	&=\frac{\tau}{C}
	\exp\!\left(
	-\frac{\log(\tau/C)}{\tau}\,t
	\right)\Etot(0),
\end{align*}
which proves \eqref{eq:general-exponential-family}.  Writing
\(r=\tau/C>1\), the rate equals
\[
  \frac1C\frac{\log r}{r}.
\]
Since
\[
  \frac{\mathrm d}{\mathrm dr}\frac{\log r}{r}
  =\frac{1-\log r}{r^2},
\]
it is uniquely maximized at \(r=\mathrm{e}\), i.e. \(\tau=\mathrm{e}C\).  Substitution gives
\eqref{eq:continuous-exponential-decay}.
\end{proof}

The explicit constant \(C\) in \eqref{eq:observability-constant} also allows us to examine its dependence
on the feedback gain \(\lambda\).

\begin{remark}
	\label{rem:lambda-optimization}
	Fix the multiplier \(m\) and the smallness threshold used in the proof of
	Theorem~\ref{thm:stabilization}.  Then \(c_E\), \(A_E\), \(a_*\), and
	\(C_{\mathrm{end}}\) are independent of \(\lambda\).  The dependence on the
	feedback gain enters \eqref{eq:observability-constant} only through
	\(A_{-1}\) and \(A_{+1}\).  Substituting
	\eqref{eq:A-minus} and \eqref{eq:A-plus} into
	\eqref{eq:observability-constant}, we obtain
	\begin{align*}
		C(\lambda)
		&=
		\frac2{c_E}
		\left[
		\frac1{a_*}
		\left(
		\frac{\lambda}{2}\|\chi\|_{L^\infty}
		+\frac{75L}{8\pi\lambda}\|\chi\|_{L^\infty}
		+\frac{L^2}{4\lambda}
		\right)
		+\frac{8a_*}{\lambda}
		+C_{\mathrm{end}}
		\right]\\
		&=
		\frac2{c_E}
		\left(
		P\lambda+\frac{Q}{\lambda}+C_{\mathrm{end}}
		\right),
	\end{align*}
	where
	\begin{equation*}
		P
		:=
		\frac{\|\chi\|_{L^\infty}}{2a_*},
		\qquad
		Q
		:=
		\frac1{a_*}
		\left(
		\frac{75L}{8\pi}\|\chi\|_{L^\infty}
		+\frac{L^2}{4}
		\right)
		+8a_*.
	\end{equation*}
	Since \(P,Q>0\), the function
	\[
	\lambda\longmapsto P\lambda+\frac{Q}{\lambda},
	\qquad \lambda>0,
	\]
	has a unique minimum.  Indeed,
	\[
	\frac{\dd}{\dd\lambda}
	\left(
	P\lambda+\frac{Q}{\lambda}
	\right)
	=
	P-\frac{Q}{\lambda^2},
	\]
	so the minimizing value is
	\begin{equation*}
		\lambda_{\mathrm{opt}}
		=
		\sqrt{\frac{Q}{P}}.
	\end{equation*}
	At this value,
	\[
	P\lambda_{\mathrm{opt}}
	=
	\frac{Q}{\lambda_{\mathrm{opt}}}
	=
	\sqrt{PQ},
	\]
	and therefore
	\begin{equation*}
		C(\lambda_{\mathrm{opt}})
		=
		\frac2{c_E}
		\left(
		2\sqrt{PQ}+C_{\mathrm{end}}
		\right).
	\end{equation*}
	
Thus, within the quantitative estimate produced by the present multiplier
argument, the choice \(\lambda_{\mathrm{opt}}\) balances the contributions
proportional to \(\lambda\) and \(\lambda^{-1}\).  This optimization refers
only to the explicit constant obtained from the present multiplier estimate.
\end{remark}

\section{Linearized  stabilization}
\label{sec:linear-check}
We conclude by considering the linearization about the flat rest state.  In
addition to the dispersion relation and the exact dissipation identities, we
show that the same localized pressure law yields a uniform stabilization
estimate for the linearized dynamics.  This provides the linear counterpart
of Theorem~\ref{thm:stabilization} within the same multiplier framework.

At the rest state \(\eta=0\), the finite-energy harmonic extension of a
nonzero Fourier mode \(\psi_k\mathrm{e}^{ikx}\), \(k=\pi n/L\), is
\[
  \phi_k(y)=\mathrm{e}^{|k|y}\psi_k,
\]
so
\begin{equation*}
  G(0)=|D|,
\end{equation*}
and
\begin{equation*}
  \Bflex'(0)\eta=\partial_x^4\eta.
\end{equation*}
Without feedback, the linearized system is
\[
  \eta_t=|D|\psi,
  \qquad
  \psi_t+g\eta+\beta\partial_x^4\eta=0,
\]
which agrees with the infinite-depth linearized Cosserat system considered in
\cite{WanYangControl}.  Hence
\[
  \eta_{tt}+|D|(g+\beta\partial_x^4)\eta=0.
\]
For \(\mathrm{e}^{i(kx-\omega t)}\),
\begin{equation*}
  \omega(k)^2=(g+\beta k^4)|k|.
\end{equation*}
Thus \(\omega(k)\sim\sqrt g\,|k|^{1/2}\) at low frequency and
\(\omega(k)\sim\sqrt\beta\,|k|^{5/2}\) at high frequency.

For the linearized system with the localized pressure feedback
\(P_{\mathrm{ext}}=\lambda\chi\eta_t\), the corresponding energy satisfies
the following exact dissipation identities.

\begin{proposition}[Linear dissipation identities]
\label{prop:linear-high-order-dissipation}
Consider the linearized system
\begin{equation}\label{eq:linear-closed-loop}
  \eta_t=|D|\psi,
  \qquad
  \psi_t+g\eta+\beta\partial_x^4\eta
  +\lambda\chi\eta_t=0.
\end{equation}
For every integer \(j\geq0\), define
\begin{equation}\label{eq:linear-time-energy-j}
  \mathcal H_j(t)
  :=\frac12\int_{\T}
  \left(
    \psi^{(j)}|D|\psi^{(j)}
    +g|\eta^{(j)}|^2
    +\beta|\partial_x^2\eta^{(j)}|^2
  \right)\dd x,
\end{equation}
where \(\eta^{(j)}=\partial_t^j\eta\) and
\(\psi^{(j)}=\partial_t^j\psi\).  Then
\begin{equation}\label{eq:linear-time-energy-identity}
  \frac{\dd}{\dd t}\mathcal H_j(t)
  =-\lambda\int_{\T}\chi(x)
  |\partial_t^{j+1}\eta(t,x)|^2\,\dd x\leq0.
\end{equation}
\end{proposition}

\begin{proof}
Because \(\chi\) is independent of time, applying \(\partial_t^j\) to
\eqref{eq:linear-closed-loop} gives
\[
  \eta^{(j)}_t=|D|\psi^{(j)},
  \qquad
  \psi^{(j)}_t+g\eta^{(j)}
  +\beta\partial_x^4\eta^{(j)}
  +\lambda\chi\eta^{(j)}_t=0.
\]
Differentiating \eqref{eq:linear-time-energy-j} and using
\(\eta^{(j)}_t=|D|\psi^{(j)}\), we obtain
\begin{align*}
	\frac{\dd}{\dd t}\mathcal H_j
	&=
	\int_{\T}
	\left(
	\psi^{(j)}_t\eta^{(j)}_t
	+g\eta^{(j)}\eta^{(j)}_t
	+\beta\partial_x^2\eta^{(j)}
	\partial_x^2\eta^{(j)}_t
	\right)\dd x.
\end{align*}
By periodicity,
\[
\int_{\T}
\partial_x^2\eta^{(j)}
\partial_x^2\eta^{(j)}_t\,\dd x
=
\int_{\T}
\partial_x^4\eta^{(j)}
\eta^{(j)}_t\,\dd x.
\]
Hence
\[
\frac{\dd}{\dd t}\mathcal H_j
=
\int_{\T}
\left(
\psi^{(j)}_t
+g\eta^{(j)}
+\beta\partial_x^4\eta^{(j)}
\right)
\eta^{(j)}_t\,\dd x.
\]
The differentiated dynamic equation gives
\[
\psi^{(j)}_t
+g\eta^{(j)}
+\beta\partial_x^4\eta^{(j)}
=
-\lambda\chi\eta^{(j)}_t.
\]
Therefore
\[
\frac{\dd}{\dd t}\mathcal H_j
=
-\lambda\int_{\T}
\chi|\eta^{(j)}_t|^2\,\dd x,
\]
which proves \eqref{eq:linear-time-energy-identity}.
\end{proof}

The dissipation identity by itself does not imply decay of the full energy,
since it controls only \(\eta_t\) on the support of \(\chi\).  In the flat
geometry the multiplier calculation used in the nonlinear problem simplifies:
the bulk term generated by the quadratic water-wave nonlinearity is absent,
while the fourth-order elastic contribution retains the coercive structure
identified in Lemma~\ref{lem:unique-flex-multiplier}.  This gives the following
linear stabilization result.

\begin{theorem}[Linear stabilization]
\label{thm:linear-stabilization}
Let
\[
  M_3=\|m_{xxx}\|_{L^\infty(\T)},
\]
and assume
\[
  g>\frac{25}{9}\,\beta M_3^2.
\]
Let \((\eta,\psi)\) be a sufficiently regular, real-valued, even solution of
\eqref{eq:linear-closed-loop} on \([0,T]\), with
\[
  \int_{\T}\eta(0,x)\,\dd x=0,
\]
and put \(\mathcal H_{\mathrm{lin}}=\mathcal H_0\), where
\(\mathcal H_0\) is defined in \eqref{eq:linear-time-energy-j}.  There exists
\[
  C_{\mathrm{lin}}
  =C_{\mathrm{lin}}(g,\beta,L,\lambda,m)>0
\]
such that, for every \(T>0\),
\[
  \int_0^T\mathcal H_{\mathrm{lin}}(t)\,\dd t
  \leq C_{\mathrm{lin}}\mathcal H_{\mathrm{lin}}(0),
  \qquad
  \mathcal H_{\mathrm{lin}}(T)
  \leq\frac{C_{\mathrm{lin}}}{T}\mathcal H_{\mathrm{lin}}(0).
\]
If \(T>C_{\mathrm{lin}}\), then
\[
  \mathcal H_{\mathrm{lin}}(0)
  \leq
  \frac{\lambda}{1-C_{\mathrm{lin}}/T}
  \int_0^T\int_{\T}\chi(x)|\eta_t(t,x)|^2\,\dd x\dd t.
\]
If the solution is defined for all \(t\geq0\), then
\[
  \mathcal H_{\mathrm{lin}}(t)
  \leq
  \exp\!\left(1-\frac{t}{\mathrm e C_{\mathrm{lin}}}\right)
  \mathcal H_{\mathrm{lin}}(0),
  \qquad t\geq0.
\]
\end{theorem}

\begin{proof}
Since \(|D|\) annihilates constants, the first equation in
\eqref{eq:linear-closed-loop} shows that the mean of \(\eta\) is conserved.
Hence \(\int_{\T}\eta(t,x)\,\dd x=0\) on \([0,T]\).

We first record the coercivity of the elastic term.  For the multiplier
\(\zeta\) in \eqref{eq:zeta-rho}, the flat identity
\eqref{eq:quadratic-r} with \(r=5/2\) gives
\[
  \int_{\T}\eta_{xxxx}\zeta\,\dd x
  =\frac94\|\eta_{xx}\|_{L^2}^2
   -2\int_{\T}m_{xx}\eta_x\eta_{xx}\,\dd x
   -\frac32\int_{\T}m_{xxx}\eta\eta_{xx}\,\dd x.
\]
Periodic integration by parts and Cauchy--Schwarz yield
\[
  \left|-2\int_{\T}m_{xx}\eta_x\eta_{xx}\,\dd x\right|
  =\left|\int_{\T}m_{xxx}\eta_x^2\,\dd x\right|
  \leq M_3\|\eta\|_{L^2}\|\eta_{xx}\|_{L^2},
\]
while
\[
  \left|\frac32\int_{\T}m_{xxx}\eta\eta_{xx}\,\dd x\right|
  \leq\frac32M_3\|\eta\|_{L^2}\|\eta_{xx}\|_{L^2}.
\]
Set
\[
  q=\frac{25\beta M_3^2}{9g}<1
\]
and choose \(\theta\in(q,1)\).  Young's inequality then gives
\[
  \int_{\T}\eta_{xxxx}\zeta\,\dd x
  \geq
  \frac94(1-\theta)\|\eta_{xx}\|_{L^2}^2
  -\frac{25M_3^2}{36\theta}\|\eta\|_{L^2}^2.
\]
Consequently, if
\[
  \mathcal H_{g,\mathrm{lin}}
  =\frac12\int_{\T}\psi|D|\psi\,\dd x
   +\frac g2\int_{\T}\eta^2\,\dd x,
\]
then
\[
  \frac12\mathcal H_{g,\mathrm{lin}}
  +\beta\int_{\T}\eta_{xxxx}\zeta\,\dd x
  \geq c_0\mathcal H_{\mathrm{lin}},
\]
where
\[
  c_0
  :=\min\left\{
    \frac12,
    \frac12\left(1-\frac q\theta\right),
    \frac92(1-\theta)
  \right\}>0.
\]

We next use the flat version of the multiplier identity.  Let \(\phi\) be the
finite-energy harmonic extension of \(\psi\) to
\(\T\times(-\infty,0)\), set
\[
  \widetilde\psi=\psi-\avg{\psi},
  \qquad
  \widetilde P_{\mathrm{ext}}
  =P_{\mathrm{ext}}-\avg{P_{\mathrm{ext}}},
\]
and define
\[
  \Sigma_0(t)
  :=L\int_{-\infty}^{0}\phi_y(t,L,y)^2\,\dd y\geq0.
\]
Repeating the proof of Proposition~\ref{prop:hydro-multiplier} in the flat
half-plane, with \(G(0)=|D|\) and with the quadratic water-wave term absent,
gives the exact identity
\begin{align}\label{eq:linear-multiplier-identity}
 \notag &\frac12\int_0^T\mathcal H_{g,\mathrm{lin}}(t)\,\dd t
  +\beta\int_0^T\int_{\T}\eta_{xxxx}\zeta\,\dd x\dd t+\mathcal{W}+\mathcal{O}\\
  =&\int_0^T\int_{\T}
    \left[\frac52\chi\widetilde\psi+(x-m)\psi_x\right]
    \eta_t\,\dd x\dd t
-\int_0^T\int_{\T}\widetilde P_{\mathrm{ext}}\zeta\,\dd x\dd t
  -\left[\int_{\T}\zeta\widetilde\psi\,\dd x\right]_0^T,
\end{align}
where \[\mathcal{W}=2g\int_0^T\int_{\T}\chi\eta^2\,\dd x\dd t,\quad \mathcal{O}=\int_0^T\Sigma_0(t)\,\dd t.\]
The term involving \(\chi\eta^2\) and the boundary contribution
\(\Sigma_0\) are nonnegative.  By the definition of \(c_0\), the remaining
gravity and elastic terms on the left are bounded below by
\[
  c_0\int_0^T\mathcal H_{\mathrm{lin}}(t)\,\dd t.
\]

It remains to control the right-hand side.  Put
\[
  D(t)=\int_{\T}\chi|\eta_t|^2\,\dd x.
\]
Proposition~\ref{prop:linear-high-order-dissipation} with \(j=0\) gives
\[
  \int_0^T D(t)\,\dd t
  \leq\frac1\lambda\mathcal H_{\mathrm{lin}}(0).
\]
The Fourier representation in the flat half-plane gives
\[
  \|\widetilde\psi\|_{L^2}^2
  \leq\frac{L}{\pi}\int_{\T}\psi|D|\psi\,\dd x
  \leq\frac{2L}{\pi}\mathcal H_{\mathrm{lin}}.
\]
Moreover, the weighted stress identity used in the proof of
Lemma~\ref{lem:trace}, now with \(\eta=0\), yields
\[
  \int_{\T}\chi\psi_x^2\,\dd x
  \leq D(t)
  +2\|\chi_x\|_{L^\infty}\mathcal H_{\mathrm{lin}}(t).
\]
Finally, set
\[
  A_m=\left\|\frac94-\frac32m_x\right\|_{L^\infty}.
\]
From \eqref{eq:zeta-rho}, periodic integration by parts, and
\(2\|\eta_x\|_{L^2}^2\leq
\|\eta\|_{L^2}^2+\|\eta_{xx}\|_{L^2}^2\),
\begin{align*}
  \|\zeta\|_{L^2}^2
  &\leq
  2\|m\|_{L^\infty}^2\|\eta_x\|_{L^2}^2
  +2A_m^2\|\eta\|_{L^2}^2
  \\
  &\leq
  \left[
    \frac{2\|m\|_{L^\infty}^2}{\beta}
    +\frac{2(\|m\|_{L^\infty}^2+2A_m^2)}{g}
  \right]\mathcal H_{\mathrm{lin}}.
\end{align*}
Since
\(P_{\mathrm{ext}}=\lambda\chi\eta_t\), subtraction of the mean does not
increase the \(L^2\)-norm and hence
\[
  \|\widetilde P_{\mathrm{ext}}\|_{L^2}
  \leq
  \lambda\|\chi\|_{L^\infty}^{\frac12}D(t)^{\frac12}.
\]
Together with \(|x-m|\leq L\chi\), these estimates imply, with a constant
\(C_1\) depending only on \(g,\beta,L,\lambda\) and \(m\),
\[
  \left|
  \int_{\T}
    \left[\frac52\chi\widetilde\psi+(x-m)\psi_x\right]
    \eta_t\,\dd x
  \right|
  +\left|\int_{\T}\widetilde P_{\mathrm{ext}}\zeta\,\dd x\right|
  \leq
  C_1D(t)+C_1\bigl(\mathcal H_{\mathrm{lin}}(t)D(t)\bigr)^{\frac12}.
\]
The endpoint term satisfies
\[
  \left|\int_{\T}\zeta\widetilde\psi\,\dd x\right|
  \leq C_1\mathcal H_{\mathrm{lin}}(t).
\]
Since \(\mathcal H_{\mathrm{lin}}\) is nonincreasing, the two endpoint
values are bounded by \(2C_1\mathcal H_{\mathrm{lin}}(0)\).  Applying
Young's inequality to the mixed term and using the bound on
\(\int_0^T D(t)\,\dd t\), we obtain
\[
  \left|\text{right-hand side of \eqref{eq:linear-multiplier-identity}}\right|
  \leq
  \frac{c_0}{2}
  \int_0^T\mathcal H_{\mathrm{lin}}(t)\,\dd t
  +C_2\mathcal H_{\mathrm{lin}}(0),
\]
where \(C_2=C_2(g,\beta,L,\lambda,m)\).  Absorbing the first term gives
\[
  \int_0^T\mathcal H_{\mathrm{lin}}(t)\,\dd t
  \leq\frac{2C_2}{c_0}\mathcal H_{\mathrm{lin}}(0).
\]
Thus we may take \(C_{\mathrm{lin}}=2C_2/c_0\).  Since the energy is
nonincreasing,
\[
  T\mathcal H_{\mathrm{lin}}(T)
  \leq\int_0^T\mathcal H_{\mathrm{lin}}(t)\,\dd t,
\]
which proves the one-step decay estimate.

Integrating \eqref{eq:linear-time-energy-identity} for \(j=0\) gives
\[
  \mathcal H_{\mathrm{lin}}(0)-\mathcal H_{\mathrm{lin}}(T)
  =\lambda\int_0^T D(t)\,\dd t.
\]
Combining this equality with the one-step estimate proves the stated
localized observability inequality when \(T>C_{\mathrm{lin}}\).  Finally,
apply the one-step estimate on successive intervals of length
\(\tau>C_{\mathrm{lin}}\).  The same iteration as in the proof of
Corollary~\ref{cor:optimized-exponential} yields
\[
  \mathcal H_{\mathrm{lin}}(t)
  \leq
  \frac{\tau}{C_{\mathrm{lin}}}
  \exp\!\left(
    -\frac{\log(\tau/C_{\mathrm{lin}})}{\tau}\,t
  \right)
  \mathcal H_{\mathrm{lin}}(0).
\]
Choosing \(\tau=\mathrm e C_{\mathrm{lin}}\) gives the announced exponential
decay.
\end{proof}
\section*{Acknowledgments}
This work was supported by the National Natural Science Foundation of China (Grant No.~12561101).

\section*{Data availability}
No data were generated or analyzed in this study.

\section*{Conflict of interest}
The authors declare no competing interests.

\end{document}